\pdfoutput=1
\RequirePackage{ifpdf}
\ifpdf % We are running pdfTeX in pdf mode
\documentclass[pdftex]{sigma}
\else
\documentclass{sigma}
\fi

\numberwithin{equation}{section}

\newtheorem{thm}{Theorem}[section]

\newtheorem{Lemma}[thm]{Lemma}
\newtheorem{prop}[thm]{Proposition}
 { \theoremstyle{definition}

\newtheorem{Remark}[thm]{Remark} }

\renewcommand{\le}{\leqslant}
\renewcommand{\ge}{\geqslant}

\newcommand{\dbar}{{\rm d}\hspace*{-0.1em}\bar{}\hspace*{0.1em}}

\usepackage{tikz}
\usetikzlibrary{calc,arrows,shapes,positioning}
\usetikzlibrary{decorations.markings}
\tikzstyle arrowstyle=[scale=1]
\tikzstyle directed=[postaction={decorate,decoration={markings,
 mark=at position .65 with {\arrow[arrowstyle]{stealth}}}}]
\tikzstyle reverse directed=[postaction={decorate,decoration={markings,
 mark=at position .65 with {\arrowreversed[arrowstyle]{stealth};}}}]

\begin{document}
\allowdisplaybreaks

\newcommand{\arXivNumber}{1711.01590}

\renewcommand{\thefootnote}{}

\renewcommand{\PaperNumber}{056}

\FirstPageHeading

\ShortArticleName{Asymptotics of Polynomials Orthogonal with respect to a Logarithmic Weight}

\ArticleName{Asymptotics of Polynomials Orthogonal\\ with respect to a Logarithmic Weight\footnote{This paper is a~contribution to the Special Issue on Orthogonal Polynomials, Special Functions and Applications (OPSFA14). The full collection is available at \href{https://www.emis.de/journals/SIGMA/OPSFA2017.html}{https://www.emis.de/journals/SIGMA/OPSFA2017.html}}}

\Author{Thomas Oliver CONWAY and Percy DEIFT}

\AuthorNameForHeading{T.O.~Conway and P.~Deift}

\Address{Department of Mathematics, Courant Institute of Mathematical Sciences, New York University,\\ 251 Mercer Str., New York, NY 10012, USA}
\Email{\href{mailto:oliver.conway641@gmail.com}{oliver.conway641@gmail.com}, \href{mailto:deift@cims.nyu.edu}{deift@cims.nyu.edu}}

\ArticleDates{Received November 29, 2017, in final form May 30, 2018; Published online June 12, 2018}

\Abstract{In this paper we compute the asymptotic behavior of the recurrence coefficients for polynomials orthogonal with respect to a logarithmic weight $w(x){\rm d}x = \log \frac{2k}{1-x}{\rm d}x$ on $(-1,1)$, $k > 1$, and verify a conjecture of A.~Magnus for these coefficients. We use Riemann--Hilbert/steepest-descent methods, but not in the standard way as there is no known parametrix for the Riemann--Hilbert problem in a neighborhood of the logarithmic singularity at $x=1$.}


\Classification{33C47; 34E05; 34M50}

\renewcommand{\thefootnote}{\arabic{footnote}}
\setcounter{footnote}{0}

\section{Introduction}
In this paper we consider the sequence of polynomials $\{p_n(x)\}$ that are orthonormal with respect to a logarithmic measure $w(x){\rm d}x$ on $[-1,1)$ and 0 elsewhere
\begin{gather*}
\int_{-1}^1 p_i(x)p_j(x)w(x){\rm d}x = \delta_{ij},
\end{gather*}
where
\begin{gather*}
w(x) =\log\frac{2k}{1-x}.
\end{gather*}
We will only consider the case $k > 1$, for which the measure is strictly positive on its support, $w(x) \ge c >0$ for $x \in [-1,1)$.

The polynomials $\{p_n(x)\}$ obey the well-known three term recurrence relation
\begin{gather*}
xp_n(x) = b_np_{n+1}(x) + a_np_n(x) + b_{n-1}p_{n-1}(x),
\end{gather*}
where the recurrence coefficients $a_n \in \mathbb{R}, b_n > 0$ are of particular interest.
\begin{Remark}
Note that the notation for $a_n$, $b_n$ in this paper is the same as the notation in \cite{DeiftOPs,DeiftSD}, but the reverse of the notation in \cite{Kuijlaars,Magnus}.
\end{Remark}

 Logarithmic singularities like the ones exhibited in the above weight present various obstacles to the asymptotic analysis of the associated polynomials $\{p_n(x)\}$ and their recurrence coefficients. The standard Riemann--Hilbert approach requires that explicit local solutions can be found in the neighborhood of singularities of the underlying measure for the associated Riemann--Hilbert problem. Because no such local solution is known near $x = 1$ for the logarithmic case, the standard approach cannot be applied directly.

The approach in this paper hinges instead on a comparison with the Legendre polyno\-mials~$\{\tilde{p}_n\}$ which are orthogonal with respect to the weight
 \begin{gather*}
 \tilde{w}(x) = \begin{cases}
 1, & \text{ for $x \in [-1,1]$}, \\
 0, & \text{ otherwise},
 \end{cases}
 \end{gather*}
 and which have recurrence coefficients $\tilde{a}_n$, $\tilde{b}_n$. The asymptotics of $\tilde{a}_n$ and $\tilde{b}_n$ are well known. We have, from \cite{Kuijlaars},
 \begin{gather*}
 \tilde{a}_n = 0,\qquad \tilde{b}_n = \frac{1}{2} + \frac{1}{16n^2}+ O\left(\frac{1}{n^3}\right).
 \end{gather*}
 The result of this comparison, and the main result of this paper is that
 \begin{thm} \label{MainResult}
 As $n\to \infty$, the differences between the log-orthogonal coefficients and the Le\-gendre coefficients have the following asymptotic behavior
 \begin{gather}
 a_n - \tilde{a}_n = \frac{2C}{(n\log n)^2} + O\left(\frac{1}{n^2(\log n)^3}\right), \label{MainTheoremDiff1}\\
 b_n - \tilde{b}_n = \frac{C}{(n \log n)^2} + O\left(\frac{1}{n^2(\log n)^3}\right),\label{MainTheoremDiff2}
 \end{gather}
 where $C = -\frac{3}{32} \approx -0.094$. In particular this implies that, as $n \to \infty$, the recurrence coeffi\-cients~$a_n$,~$b_n$ have the following asymptotic behavior
 \begin{gather*}
 a_n = \frac{2C}{(n\log n)^2} + O\left(\frac{1}{n^2(\log n)^3}\right),\\
 b_n = \frac{1}{2} + \frac{1}{16n^2} + \frac{C}{(n \log n)^2} + O\left(\frac{1}{n^2(\log n)^3}\right).
 \end{gather*}
 \end{thm}
 This result, and more, was conjectured by Alphonse Magnus in \cite{Magnus}, up to the precise value for the constant~$C$. We are indebted to his work, which provided the inspiration for the above result. For the weight $w(x) = -\log x$ on $[0,1]$, in particular, Magnus conjectured that as $n \to \infty$
 \begin{gather}
 a_n = \frac{1}{2} - \frac{1}{8n^2} - \frac{2C}{(n\log n)^2} + o\left(\frac{1}{n^2(\log n)^2}\right),\label{aMag}\\
 b_n = \frac{1}{4} - \frac{1}{32n^2} + \frac{C}{(n \log n)^2} + o\left(\frac{1}{n^2(\log n)^2}\right),\label{bMag}
 \end{gather}
and on numerical grounds he conjectured further that $C \sim -0.044$. The weight $-\log x$ on $[0,1]$ corresponds to the case $k=1$ for which our analysis is not yet complete. Nevertheless, our prelimary calculations confirm (\ref{aMag}) and (\ref{bMag}) with stronger estimates $O\big(\frac{1}{n^2(\log n)^3}\big)$ on the error terms, and with the explicit value $-3/64$ for $C$. As $-3/64 \sim -0.047$ this is close to Magnus's conjectured value.

The proof proceeds as follows: The polynomials $\{p_n\}$ have a well-known expression in terms of the solution $Y$ of a Riemann--Hilbert problem due to Fokas--Its--Kitaev, see Section~\ref{section3.1}. We follow a now standard series of transformations as in~\cite{DeiftOPs,Kuijlaars}, that transform the original Riemann--Hilbert problem for $Y$ into a modified problem that is close to a Riemann--Hilbert problem whose asymptotics can be inferred directly.

Here is where our analysis diverges from the standard analysis. The standard analysis compares the modified Riemann--Hilbert problem to a~model Riemann--Hilbert problem (parametrix) that explicitly solves the Riemann--Hilbert problem locally around the critical points at the edge of the measure. The new difficulty of our problem is that no explicit local solution is known near the logarithmic singularity. The main new feature of our analysis is that, for the problem at hand, no local solution is needed.

The strategy here is to compare the modified Riemann--Hilbert problem above to the corresponding modified Riemann--Hilbert problem for the Legendre case. In the standard approach, a~comparison is made by examining the quotient of solutions to two Riemann--Hilbert problems as the solution to a Riemann--Hilbert problem in its own right. However, for the case at hand, the log singularity renders this approach ineffective. Instead we compare the problems through the operator theory that underlies Riemann--Hilbert problems, see Section~\ref{MachinerySection}, or see \cite{DeiftNLS} for more details.

In Theorem \ref{ABInverse} we prove a general and very useful formula that allows for an effective comparison of two Riemann--Hilbert problems on the same contour. In our case, where we compare the log case with the Legendre case, this formula can be approximated by quantities solely related to the Legendre case, for which the asymptotics are well known, and this leads, eventually, to the formulas~(\ref{MainTheoremDiff1}) and~(\ref{MainTheoremDiff2}). It is a remarkable, and unexpected, development that the log problem can be approximated adequately by the Legendre problem.

The proof is laid out within this paper as follows:
\begin{itemize}\itemsep=0pt
\item In Section \ref{PreliminariesSection}, we prove several preliminary properties and estimates on the auxilliary functions that will arise in our analysis.
\item In Section \ref{RHPsSection}, we introduce the Riemann--Hilbert problem for the orthogonal polynomials and transform this problem into one which can be controlled as $n\to \infty$.
\item In Section \ref{MachinerySection}, we prove bounds, uniform in $n$, on the resolvent operators that are associated with the Riemann--Hilbert problem in the standard way and prove Theorem~\ref{ABInverse}.
\item In Section \ref{ProofSection}, we calculate the asymptotics of the recurrence coefficients by combining the results of the previous sections with calculations derived from the known asymptotics of Legendre quanties.
\item The above calculations are given in Appendices~\ref{FProofsSection},~\ref{NormsSection}, and~\ref{IntegralSection}.
\end{itemize}

Note that the case $k=1$ corresponding to the weight $\log\frac{2}{1-x}$ on the interval $[-1,1)$, for which $w(x) > 0$ on its support, but for which $w(x) \ge c > 0$ fails, is not considered in this paper. Recall that questions concerning orthogonal polynomials on the line with respect to a measure of compact support correspond to questions involving associated orthogonal polynomials on the unit circle, where the vanishing of the weight $\log\frac{2}{1-x}$ at the point $-1$ corresponds to a Fisher--Hartwig singularity for the related problem on the circle, see~\cite{DeiftIsing}. The methods in this paper do not extend immediately to cover this case, but will be addressed in a future paper.

This paper is in the line of questioning concerning the effect that singularities and zeroes in the measure have on the asymptotic behavior of orthogonal polynomials. Here we have in mind the extensive history of the problem dating originally back to the work of Szeg{\H{o}}, proving the strong limit theorem for Toeplitz determinants for non-vanishing smooth weights on the circle without singularities \cite{Szego}, and whose result has been extended to include a wide class of algebraic singularities, in particular, by the work of M.E.~Fisher and R.E.~Hartwig, see \cite{DeiftIsing} for a history. Of additional interest is the recent work of Barry Simon et al.\ concerning sum laws of Szeg\H{o} type where zeros of the form $e^{-1/x^2}$, for example, in the weight are investigated (see~\cite{Simon} for an up to date discussion). The logarithimic singularities explored in this paper are of practical interest in both physics and mathematics as described in \cite{Magnus,VanAssche}.

\section[The $\phi$ function and the Szeg\H{o} function]{The $\boldsymbol{\phi}$ function and the Szeg\H{o} function}\label{PreliminariesSection}
In this section, we introduce several functions central to our analysis and detail their properties.
\subsection[The map $\phi$ and the Szeg\H{o} function $F$]{The map $\boldsymbol{\phi}$ and the Szeg\H{o} function $\boldsymbol{F}$}\label{Prel1}

The first part of our work closely follows the work from the papers \cite{DeiftSD,Kuijlaars,Venakides}, and uses the Riemann--Hilbert approach to orthogonal polynomials, first put forward by Fokas, Its, and Kitaev \cite{FIK}. We apply similar sets of transformations to the Riemann--Hilbert problems associated with the weights $w$ and $\tilde{w}$,
\begin{gather*}
Y \to T \to Q,\qquad \tilde{Y} \to \tilde{T} \to \tilde{Q}.
\end{gather*}
The problems $Q$ and $\tilde{Q}$ are then close in a ``certain'' sense, allowing the comparison. The prob\-lem~$\tilde{Q}$ can be asymptotically solved as in~\cite{Kuijlaars}, as repeated in the appendix, also see Section~\ref{LegendreSummary}.

In the calculations that follow $\big(z^2-1\big)^{1/2}$ refers to the branch of the square root of $z^2 -1$ that is analytic in $\mathbb{C}\backslash [-1,1]$ and is positive for $z > 1$. For $z \in (-1,1)$, $\big(z^2 - 1\big)^{1/2}_\pm$ refers to the boundary value of $\big(z^2-1\big)$ as $z' \to z$ with $z' \in \mathbb{C}_\pm$ respectively, $\mathbb{C}_\pm = \{z = x \pm iy\,|\, x\in \mathbb{R},\, y >0\}$. Additionally, when $x > 0$, $\sqrt{x}$ will refer to the positive square root of~$x$.

The map
\begin{gather*}
\phi(z) = z + \big(z^2-1\big)^{1/2}, \qquad z \in \mathbb{C}\backslash [-1,1]
\end{gather*}
plays a key role in the transformations $Y \to T$ and $\tilde{Y} \to \tilde{T}$. In particular $\phi$ plays the role of normalizing these problems at infinity. The function $\phi$ has the following properties,
\begin{prop}\label{PhiProps}
The function $\phi$, and its boundary values, $\phi_\pm$, on the interval $[-1,1]$ have the following properties:
\begin{gather}
0)\quad \phi(z)\text{ is analytic for $z \in \mathbb{C} \backslash [-1,1]$}, \label{prop0}\\
1)\quad \phi(z) = 2z + O\left(\frac{1}{z}\right) \text{ as $z \to \infty$},\label{prop1}\\
2)\quad \phi(z) = 1 + \sqrt{2}(z-1)^{1/2} + O(z-1) \text{ as $z \to 1$},\label{1branch}\\
3)\quad \phi(z) = -1 + \sqrt{2}i(z+1)^{1/2} + O(z+1) \text{ as $z \to -1$},\label{2branch}\\
4)\quad \phi(\overline{z}) = \overline{\phi(z)}, \label{prop4}\\
5)\quad \vert \phi(z)\vert > 1,\label{prop5}\\
6)\quad \phi_\pm(x) = x \pm i \sqrt{1-x^2},\label{prop6}\\
7) \quad \phi_+(x)\phi_-(x) = 1,\label{prop7}\\
8) \quad \vert \phi_\pm(x)\vert = 1.\label{prop8}
\end{gather}
The function $(z-1)^{1/2}$ in~\eqref{1branch} is analytic in $\mathbb{C}\backslash (-\infty, 1]$ and positive for $z > 1$ and the function $(z+1)^{1/2}$ in~\eqref{2branch} is analytic in $\mathbb{C}\backslash (-1, \infty]$ and is positive imaginary from $z < -1$.
\end{prop}
\begin{proof}
We will prove statement~\eqref{prop5}; the remaining statements follow directly from the definition of $\phi$.

First note that for all $z \in \mathbb{C} \backslash [-1,1]$,
\begin{gather*}
\phi(z) \times \big(z - \big(z^2 - 1\big)^{1/2}\big) = z^2 - \big(z^2 - 1\big) = 1.
\end{gather*}
Therefore $\phi(z) \neq 0$. The function $\frac{1}{\phi}$ is therefore analytic in the region $\mathbb{C}\backslash [-1.1]$. By pro\-per\-ty~\eqref{prop1}, we see that $\frac{1}{\phi}(z) \to 0$ as $z \to \infty$, and by property~\eqref{prop8}, we see that $\big\vert \frac{1}{\phi}(z) \big\vert \to 1$ as $z \to [-1,1]$. Therefore, the maximum modulus principle guarantees that $\big\vert \frac{1}{\phi}(z) \big\vert < 1$ for $z \in\mathbb{C} \backslash [-1,1]$ and so $\vert \phi (z) \vert > 1$ in the same region.
\end{proof}

For the log-orthogonal polynomial case, we make use of the associated Szeg\H{o} function $F$ similar to the non-Legendre cases analyzed in \cite{Kuijlaars},
\begin{gather*}
F(z) = \exp \left(\big(z^2-1\big)^{1/2} \int_{-1}^1 \frac{\log w(s)}{\big(s^2-1\big)^{1/2}_+}\frac{\dbar s}{s-z}\right),
\end{gather*}
where
\begin{gather*}
\dbar s := \frac{{\rm d}s}{2\pi i}
\end{gather*}
and similarly $\dbar t, \dbar z, \dots$ mean $\frac{{\rm d}t}{2\pi i}, \frac{{\rm d}z}{2\pi i}, \dots$. In the calculations that follow we shall use the following lemma repeatedly without further comment:
\begin{Lemma}\label{NoFurtherComment}
The function $w(s) = \log \frac{2k}{1-s}$ has an analytic extension from $[-1,1)$ to $\mathbb{C} \backslash [1, \infty)$ where $\log$ takes on its principle value.

The function $L(s) := \log w(s)$ has an analytic extension from $(-1, 1)$ to $\mathbb{C} \backslash ( (-\infty, 1-2k] \cup [1, \infty))$, where again $\log$ takes its principle value.
\end{Lemma}
\begin{proof}
The proof is left to the reader.
\end{proof}
\begin{prop}\label{FProps}
The function $F$, and its boundaray values, $F_\pm$, on the interval $[-1,1]$ have the following properties:
\begin{gather}
0) \quad F(z) \text{ is analytic for $z \in \mathbb{C} \backslash [-1,1]$}, \label{prop0-2.3}\\
1) \quad F(z) = F_\infty + \frac{F_1}{z} + O\left(\frac{1}{z^2}\right) \text{ as $z \to \infty$}, \text{ for suitable constants $F_\infty, F_1 \in \mathbb{R}$},\label{prop1-2.3}\\
2) \quad \frac{F^2}{w}(z) = 1 \mp \frac{i\pi}{w(z)} - \frac{\pi^2}{2 w^2(z)} + O\left(\frac{1}{w^3(z)}\right)\text{ uniformly as $z \to 1$},\label{prop2-2.3}\\
3) \quad \frac{F^2}{w}(z) = 1 + O(\vert z+1\vert^{1/2}) \text{ uniformly as $z \to -1$},\label{prop3-2.3}\\
4) \quad F(\overline{z}) = \overline{F(z)},\label{prop4-2.3}\\
5) \quad F_+(x)F_-(x) = w(x).\label{prop5-2.3}
\end{gather}
\end{prop}
\begin{proof}
With the exception of statements~\eqref{1branch} and~\eqref{2branch} of Proposition~\ref{PhiProps}, the above statements are straightforward to prove. Statements~\eqref{prop2-2.3} and~\eqref{prop3-2.3} are proven in Appendix~\ref{AppendixA} as Proposition~\ref{FAsymptoticsZApp} and Proposition~\ref{-1FAsymptotics}.
\end{proof}

Additionally, as a consequence of Proposition~\ref{FProps}, the Szeg\H{o} function satisfies.
\begin{prop} \label{FCancellation}
\begin{gather*}
\frac{F^2}{w_+}(x) + \frac{F^2}{w_-}(x) - 2 = - \frac{3\pi^2}{ \log^2 \frac{2k}{\vert x - 1\vert}} + O\left(\frac{1}{\log^3 \vert 1-x \vert}\right) %\label{FAsymptotics}
\end{gather*}
for $x > 1$ as $x \to 1$.
\end{prop}
\begin{proof}
As the estimate in statement~\eqref{prop2-2.3} of Proposition \ref{FProps} is uniform in $0 < \vert \theta \vert < \pi$, we can compute for $x > 1$, $x\to 1$,
\begin{gather*}
\frac{F^2}{w_+}(x) + \frac{F^2}{w_-}(x) - 2 = -\frac{i \pi}{w_+(x)} + \frac{i\pi}{w_-(x)} - \frac{\pi^2}{ \log^2 \frac{2k}{\vert x-1\vert}} + O\left(\frac{1}{\log^3 \vert 1-x \vert}\right) \nonumber\\
\hphantom{\frac{F^2}{w_+}(x) + \frac{F^2}{w_-}(x) - 2}{} = \frac{i\pi (-w_- + w_+)}{w_+(x)w_-(x)} - \frac{\pi^2}{ \log^2 \frac{2k}{\vert x-1\vert}} + O\left(\frac{1}{\log^3 \vert 1-x \vert}\right) \nonumber\\
\hphantom{\frac{F^2}{w_+}(x) + \frac{F^2}{w_-}(x) - 2}{}= -\frac{2\pi^2}{w_+(x)w_-(x)} - \frac{\pi^2}{ \log^2 \frac{2k}{\vert x-1\vert}} + O\left(\frac{1}{\log^3 \vert 1-x \vert}\right) \nonumber\\
\hphantom{\frac{F^2}{w_+}(x) + \frac{F^2}{w_-}(x) - 2}{} = - \frac{3\pi^2}{ \log^2 \frac{2k}{\vert x -1\vert}} + O\left(\frac{1}{\log^3 \vert 1-x \vert}\right)
\end{gather*}
as $x \to 1$.
\end{proof}

We will also need the following technical estimate on the difference $\frac{F^2}{w_\pm}(1 + r) - \frac{F^2}{w_\pm} (1 + \tilde{r})$ for $r, \tilde{r} > 0$ suitably small and close together. This estimate is needed to estimate the integral in Proposition~\ref{HIntegralProp}.
\begin{prop} \label{F-FAsymptotics} Fix $R > 0$ and suppose $r, \tilde{r} > 0$ obey
\begin{gather*}
r, \tilde{r} = O\left(\frac{1}{n^2}\right), \qquad n\left(\frac{r}{\tilde{r}} - 1\right) \in [-R, R].
\end{gather*}
Then
\begin{gather*}
\frac{F^2}{w_+}(1 + r) - \frac{F^2}{w_+} (1 + \tilde{r}) + \frac{F^2}{w_-}(1 + r) - \frac{F^2}{w_-} (1 + \tilde{r}) = O_R\left(\frac{1}{n\log^3n}\right).
\end{gather*}
\end{prop}
\begin{proof}
See Proposition \ref{F-FAsymptoticsApp}.
\end{proof}

\section{The Riemann--Hilbert approach} \label{RHPsSection}
In this section we define the main Riemann--Hilbert problem and transform the problem into a~form which is amenable to asymptotic analysis.
\begin{Remark}
The notation $O_n$ is used as follows. For matrices
\begin{gather*}
A(z) = A^{(n)}(z) = \left(\begin{matrix}
A_{11}^{(n)}(z) & A^{(n)}_{12}(z)\\
A_{21}^{(n)}(z) & A^{(n)}_{22}(z)
\end{matrix}\right), \qquad B(z)= \left(\begin{matrix}
B_{11}(z) & B_{12}(z)\\
B_{21}(z) & B_{22}(z)
\end{matrix}\right)
\end{gather*}
we say $A(z) = O_n(B(z))$ as $z \to z_0$ if, for each pair $i,j = 1, 2$, and for any $n\ge 0$, there exists $\varepsilon$ and $c_n$ such that
\begin{gather*}
\big\vert A^{(n)}_{ij}(z)\big\vert \le c_n \vert B_{ij}(z)\vert
\end{gather*}
for all $\vert z- z_0 \vert \le \varepsilon$. Similarly, for scalar functions $A_n(z)$ and $B(z)$, we say that $A_n(z) = O(B(z))$ if there exists $c_n$ such that
\begin{gather*}
\vert A_n(z) \vert \le c_n\vert B(z)\vert
\end{gather*}
\end{Remark}

\subsection[Riemann--Hilbert problem for $Y$]{Riemann--Hilbert problem for $\boldsymbol{Y}$}\label{section3.1}
We will denote by $\pi_n(x) = x^n + \cdots$, $\tilde{\pi}_n = x^n + \cdots$ the monic, orthogonal polynomials associated with the weights $w, \tilde{w}$ respectively. The polynomials $\{\pi_n\}$ are described by the following Riemann--Hilbert problem introduced by Fokas, Its, and Kitaev \cite{FIK} (see also \cite{DeiftOPs}) for a 2$\times$2 matrix valued function $Y(z) = Y^{(n)}(z)$:
\begin{enumerate}\itemsep=0pt
\item[(a)]
\begin{gather}
Y^{(n)}(z) \text{ is analytic for $z \in \mathbb{C} \backslash [-1,1]$}. \label{YRHP1}
\end{gather}
\item[(b)] $Y^{(n)}$ has boundary values, $Y_\pm := \lim\limits_{\substack{z \to x, z \in \mathbb{C}_\pm}} Y(z)$, for $x \in \mathbb{R}$ pointwise almost
 everywhere and in the following $L^2$ sense:
 \begin{gather*}
 \text{for any $R > 1$}, \qquad \lim_{\epsilon \downarrow 0} \int_{-R}^R \vert Y(x\pm i \epsilon) - Y_\pm(x)\vert^2{\rm d}x =0
 \end{gather*}
and
\begin{gather} Y^{(n)}_+(x) = Y^{(n)}_-(x) \left(\begin{matrix}
1 & w(x) \\
0 & 1
\end{matrix}\right) \qquad \text{for $x \in (-1,1)$}. \label{YRHP2}
\end{gather}
\item[(c)] $Y^{(n)}$ has the following asymptotics as $z \to \infty$
\begin{gather}
 Y^{(n)}(z) = \left(I + O_n\left(\frac{1}{z}\right)\right)\left(\begin{matrix}
z^n & 0\\
0 & z^{-n}
\end{matrix}\right). \label{YRHP3}
\end{gather}
Of course, for $\vert x \vert > 1$, $Y_+(x) = Y_-(x) = Y(x)$.
\end{enumerate}

Define the Cauchy operator, $C_{[-1,1]}$ by
\begin{gather*}
C_{[-1,1]}f = \int_{-1}^1 \frac{f(s)}{s-z} \dbar s
\end{gather*}
and let $C_{[-1,1]}^\pm$ denote its boundary values from $\mathbb{C}_\pm$ respectively. We give a brief summary of the properties of the Cauchy operator and its boundary values on general contours in Section~\ref{Machinery1}.
\begin{prop}[Fokas--Its--Kitaev]\label{Fokas}
Let $\gamma_n > 0$ be the leading coefficient of $p_n$, the orthonormal polynomials with respect to the weight $w$, $p_n = \gamma_n\pi_n$. The function
\begin{gather}
Y^{(n)}(z) = \left(\begin{matrix}
\pi_n(z) & (C_{[-1,1]}\pi_n w)(z) \\
 2\pi i \gamma_{n-1}^2 \pi_{n-1}(s) & 2\pi i\gamma_{n-1}^2 (C_{[-1,1]}\pi_{n-1} w)(z)
\end{matrix}\right)\label{YSol}
\end{gather}
uniquely solves the RHP for $Y^{(n)}$.
\end{prop}
\begin{proof}
Aside from some special attention that must be paid to the $L^2$ sense in which the boundary values $Y_\pm$ are achieved, the proof proceeds as in \cite{DeiftOPs}. The $L^2$ condition in (\ref{YRHP2}) plays the following key role in proving uniqueness:

Note that, by the $L^2$ condition in (\ref{YRHP2}), for any $a, b \in \mathbb{R}$
\begin{gather*}
\lim_{\epsilon \downarrow 0} \int_a^b \left\vert \det Y(x \pm i\epsilon) - \det Y_\pm(x) \right\vert {\rm d}x \nonumber\\
\qquad{}=\lim_{\epsilon \downarrow 0} \int_a^b \left\vert Y_{11}Y_{22}(x \pm i\epsilon) - ( Y_{11}Y_{22})_\pm (x) - \left(Y_{12}Y_{21}(x \pm i\epsilon) - ( Y_{12}Y_{21})_\pm (x) \right)\right\vert {\rm d}x = 0.
\end{gather*}
However, the jump condition in (\ref{YRHP2}) implies $\det Y_+(x) = \det Y_-(x)$ for almost all $x \in \mathbb{R}$. Therefore,
\begin{gather*}
\lim_{\epsilon \downarrow0} \int_a^b \det Y(x + i \epsilon) = \lim_{\epsilon \downarrow 0}\int_a^b \det Y(x - i \epsilon)
\end{gather*}
and it follows, by standard calculations, that $\int_C \det Y(z) {\rm d}z = 0$ for all closed contours $C \subset \mathbb{C}$. Hence $\det Y(z)$ is entire by Morera's theorem. Condition~(\ref{YRHP3}) and Liouville's theorem then implies that $\det Y(z) = 1$ for all $z$. In particular, $Y(z)$ is invertible for all $z \in \mathbb{C}\backslash [-1,1]$.

Let $\tilde{Y}$ be a second solution to the RHP (\ref{YRHP1})--(\ref{YRHP3}) and define $H(z) = \tilde{Y}(z)Y^{-1}(z)$, and repeating the same argument, we find that $H(z)$ is entire, and $H(z) \to I$ as $z\to \infty$ and hence $H(z) = 1$ for all $z$ and therefore $Y = \tilde{Y}$.

We must also demonstrate that the function $Y^{(n)}(z)$ defined in (\ref{YSol}) achieves its boundary values in the aforementioned $L^2$ sense. A standard argument as in \cite{Kuijlaars} shows that
\begin{enumerate}\itemsep=0pt
\item[(d)] $Y^{(n)}$ has the following uniform asymptotics as $z \to -1$, $z \in \mathbb{C}\backslash [-1,1]$
\begin{gather} Y^{(n)}(z) = O_n\left(\begin{matrix}
1 & \log(\vert z+1\vert) \\
1 & \log(\vert z +1\vert)
\end{matrix}\right).\label{YRHP4}
\end{gather}
\end{enumerate}

 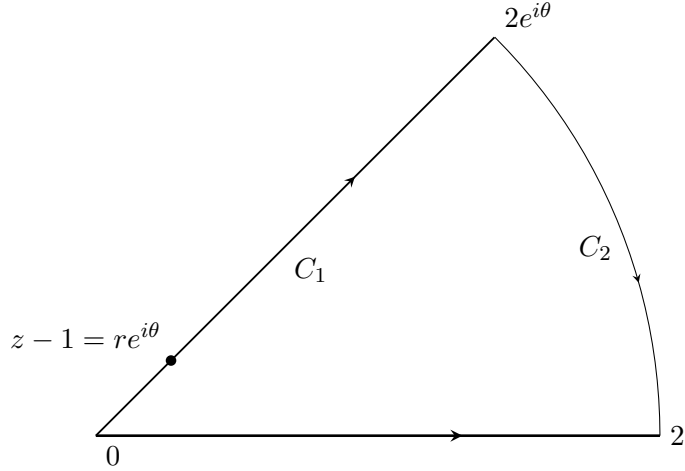
\begin{figure}[t]
\centering
\begin{tikzpicture}[scale=5.]
\coordinate (up) at (.5,.2);
\coordinate (dn) at (.5,-.2);
\draw[line width = 1,directed] (-1,0) to (.5,0) ;
\draw[line width = .7,directed] (-1,0) to (.06,1.06);
\draw [black, domain= 45:0, directed] plot ({1.5 * cos(\x) -1}, {1.5 * sin(\x)});
\fill (-.8, .2) circle (.4pt);
\node[below right] at (-1,0) {$0$};
\node[above left] at (-.8, .2) {$z-1= re^{i\theta}$};
\node[below right] at (-.5, .5){$C_1$};
\node[left] at (.4, .5){$C_2$};
\node[above right] at (.06, 1.06) {$2e^{i\theta}$};
\node[right] at (.5, 0){$2$};
\end{tikzpicture}
\caption{Definition of $C_1$, $C_2$.}\label{C1DefRHPS}
\end{figure}

 We must also analyze $Y(z)$ for $z$ near~1. Let $z = 1 + re^{i\theta}$ where $r > 0$ and $-\pi < \theta < \pi$, and recall from Lemma~\ref{NoFurtherComment} that $w(s)$ has to analytic extension in $\mathbb{C}\backslash [1, \infty)$. The following estimate holds uniformly for $-\pi < \theta < \pi$,
\begin{gather*}
\left\vert\int_{-1}^1 \frac{\pi_n(s) w(s)}{s-z} \dbar s\right\vert = \left\vert \int_{-1}^1 \frac{\pi_n(s)w(s)}{s-1 - re^{i\theta}} \dbar s \right\vert= \left\vert \int_{0}^2 \frac{\pi_n(1-s)w(1-s)}{s + re^{i\theta}} \dbar s \right\vert \nonumber\\
\hphantom{\left\vert\int_{-1}^1 \frac{\pi_n(s) w(s)}{s-z} \dbar s\right\vert}{} = \left\vert \int_{C_1}\frac{\pi_n(1-s)w(1-s)}{s + re^{i\theta}} \dbar s + \int_{C_2}\frac{\pi_n(1-s)w(1-s)}{s + re^{i\theta}} \dbar s \right\vert,
\end{gather*}
where $C_1$, $C_2$ are depicted in Fig.~\ref{C1DefRHPS}. The integrand is clearly bounded over $C_2$ is clearly analytic and hence the integral of $C_2$ is $O(1)$
\begin{gather*}
 \left\vert \int_{C_1}\frac{\pi_n(1-s)w(1-s)}{s + re^{i\theta}} \dbar s\right\vert \le c_n \int_{C_1}\left\vert \frac{w(1-s)}{s + re^{i\theta}}\right\vert \vert {\rm d}s\vert.
 \end{gather*}
Making the change of variables $s = re^{i\theta}t$,
 \begin{gather*}
c_n \int_{C_1}\left\vert \frac{w(1-s)}{s + re^{i\theta}}\right\vert \vert {\rm d}s\vert = c_n \int_0^{2/r}\left| \frac{w(1 - re^{i\theta}t)}{re^{i\theta}t + re^{i\theta}} \right| r{\rm d}t \nonumber\\
 \hphantom{c_n \int_{C_1}\left\vert \frac{w(1-s)}{s + re^{i\theta}}\right\vert \vert {\rm d}s\vert}{}
 = c_n \int_0^{2/r}\left| \frac{w(1 - re^{i\theta}t)}{t +1} \right| {\rm d}t = c_n \int_0^{2/r}\left| \frac{\log \frac{re^{i\theta}t}{2k}}{t +1} \right| {\rm d}t \nonumber\\
 \hphantom{c_n \int_{C_1}\left\vert \frac{w(1-s)}{s + re^{i\theta}}\right\vert \vert {\rm d}s\vert}{}\le c_n \int_0^{2/r}\left| \frac{\log t}{t +1} \right| {\rm d}t + c_n \log r\int_0^{2/r}\left| \frac{1}{t +1} \right| {\rm d}t \nonumber\\
 \hphantom{c_n \int_{C_1}\left\vert \frac{w(1-s)}{s + re^{i\theta}}\right\vert \vert {\rm d}s\vert}{}\le c_n \int_0^{1}\left| \frac{\log t}{t +1} \right| {\rm d}t + c_n \int_1^{2/r}\left| \frac{\log t}{t +1} \right| {\rm d}t + O_n\big(\log^2r\big) \nonumber\\
 \hphantom{c_n \int_{C_1}\left\vert \frac{w(1-s)}{s + re^{i\theta}}\right\vert \vert {\rm d}s\vert}{} \le c_n \log r \int_{1}^{2/r} \left\vert\frac{1}{t+1}\right\vert {\rm d}t + O_n\big(\log^2t\big)= O_n\big(\log^2 r\big)
\end{gather*}
proving
\begin{gather*}
\left\vert\int_{-1}^1 \frac{\pi_n(s) w(s)}{s-z} \dbar s\right\vert = O_n\big(\log^2 \vert z -1 \vert\big)
\end{gather*}
uniformly for $z \in \mathbb{C} \backslash [-1,1]$, as $z\to 1$. Therefore,
\begin{enumerate}\itemsep=0pt
\item[(e)] $Y^{(n)}$ has the following uniform asymptotics as $z \to 1$, $z \in \mathbb{C}\backslash [-1,1]$
\begin{gather}
Y^{(n)}(z) = O_n\left(\begin{matrix}
1 & \log^2(\vert z-1\vert) \\
1 & \log^2(\vert z-1\vert)
\end{matrix}\right).\label{YRHP5}
\end{gather}
\end{enumerate}
Lastly, note that the function $Y^{(n)}(z)$ clearly is continuous up to the boundary for all $x \in \mathbb{R}\backslash \{-1, 1\}$ from the properties of the Cauchy operator. Therefore, combining~(\ref{YRHP4}) and~(\ref{YRHP5}), it is clear that $Y$ achieves its boundary values $Y_\pm$ in the $L^2$ sense described in~(\ref{YRHP2}).
\end{proof}

Let $\sigma_3$ refer to the 3rd Pauli matrix
\begin{gather*}
\sigma_3 = \left(\begin{matrix}
1 & 0\\
0 & -1
\end{matrix}\right).
\end{gather*}
The recurrence coefficients $a_n$, $b_n$ associated with the polynomials $\{\pi_n\}$ have the following expressions in terms of the solution $Y^{(n)}(z)$:
\begin{prop}\label{Referee1}
Let $Y^{(n)}_1$ be the coefficient of $\frac{1}{z}$ in the Laurent expansion $Y^{(n)}(z) z^{-n\sigma_3} =I + \frac{Y^{(n)}_1}{z} + O\big(\frac{1}{z^2}\big)$ as $z \to \infty$. Then
\begin{gather*}
a_n = \big(Y_1^{(n)}\big)_{11} - \big(Y_1^{(n+1)}\big)_{11}, \qquad
b_{n-1}^2 = \big(Y_1^{(n)}\big)_{12} \big(Y_1^{(n)}\big)_{21},
\end{gather*}
and
\begin{gather*}
1 = \big(Y^{(n)}_1\big)_{12}\big(Y^{(n+1)}_1\big)_{21}.
\end{gather*}
\end{prop}
\begin{proof}See \cite{DeiftOPs}.
\end{proof}

References for the calculations that follow are \cite{DeiftOPs,DeiftSD,Kuijlaars,Venakides}.
\subsection[The transformation $Y \to T$]{The transformation $\boldsymbol{Y \to T}$}
The first step in the asymptotic analysis of $Y$, is to provide a reformulation of the RHP described by (\ref{YRHP1})--(\ref{YRHP3}) in terms of a RHP normalized at infinity, and now with a bounded, but highly oscillatory, jump function.
\begin{Remark}
In the standard Riemann--Hilbert theory, the RHP is converted to a problem involving a bounded singular operator on the contour. The replacement of the unbounded jump $\left(\begin{smallmatrix}
1 & w \\
0 & 1
\end{smallmatrix}\right)$ with a bounded jump makes it possible to apply Riemann--Hilbert theory directly.
\end{Remark}

Recall the definitions of $\phi$ and $F$ from Section \ref{PreliminariesSection} and define
\begin{gather}
T(z) = \left(\frac{2^n}{F_\infty}\right)^{\sigma_3} Y(z) \left(\frac{F(z)}{\phi^n(z)}\right)^{\sigma_3}, \qquad z \in \mathbb{C}\backslash[-1,1].\label{TDef}
\end{gather}
Since $Y$, $\phi$, and $F$ are all analytic in the region $\mathbb{C} \backslash [-1,1]$, $T$ clearly satisfies the same analyticity condition (\ref{YRHP1}) as~$Y$.

For $x \in (-1,1)$
\begin{gather*}
T_+(x) = \left(\frac{2^n}{F_\infty}\right)^{\sigma_3} Y_+(x) \left(\frac{F_+(x)}{\phi^n_+(x)}\right)^{\sigma_3} = \left(\frac{2^n}{F_\infty}\right)^{\sigma_3} Y_-(x) \left(\begin{matrix}
1 & w(x)\\
0 & 1
\end{matrix}\right) \left(\frac{F_+(x)}{\phi^n_+(x)}\right)^{\sigma_3} \nonumber\\
\hphantom{T_+(x)}{}
= T_-(x) \left(\frac{F_-(x)}{\phi^n_-(x)}\right)^{-\sigma_3} \left(\begin{matrix}
1 & w(x)\\
0 & 1
\end{matrix}\right) \left(\frac{F_+(x)}{\phi^n_+(x)}\right)^{\sigma_3} \\
\hphantom{T_+(x)}{}
= T_-(x)\left(\begin{matrix}
\dfrac{F_+}{F_-} \left(\dfrac{\phi_-}{\phi_+}\right)^n & \dfrac{ \phi_-^n\phi_+^n}{F_+F_-} w\vspace{1mm}\\
0 & \dfrac{F_-}{F_+} \left(\dfrac{\phi_+}{\phi_-}\right)^n
\end{matrix}\right)
= T_-(x)\left(\begin{matrix}
\dfrac{F_+^2}{w} \phi_+^{-2n} & 1\\
0 & \dfrac{F_-^2}{w} \phi_-^{-2n}
\end{matrix}\right),
\end{gather*}
where we have used property~\eqref{prop7} of Proposition \ref{PhiProps} and property~\eqref{prop5-2.3} of Proposition \ref{FProps}.
Combining property~\eqref{prop1} of %both
Proposition \ref{PhiProps} and property~\eqref{prop1-2.3} of Proposition \ref{FProps} with (\ref{TDef}), we see that
\begin{gather*}
T(z) = I + O_n\left(\frac{1}{z}\right) \qquad \text{as $z\to \infty$}
\end{gather*}
Finally by properties~\eqref{1branch} and \eqref{2branch} of Proposition~\ref{PhiProps} and properties \eqref{prop2-2.3} and \eqref{prop3-2.3} of Proposition~\ref{FProps} and by conditions~(\ref{YRHP4}) and~(\ref{YRHP5}) of the RHP for $Y$, we see that
\begin{gather*}
T(z) = O_n\left(\begin{matrix}
1 & \log(\vert z + 1\vert) \\
1 & \log(\vert z + 1\vert)
\end{matrix}\right) \qquad \text{as $z \to -1$},\nonumber\\
T(z) = O_n\left(\begin{matrix}
\log^{1/2}(\vert z-1\vert) & \log^{3/2}(\vert z - 1\vert) \\
\log^{1/2}(\vert z-1\vert) & \log^{3/2}(\vert z - 1\vert)
\end{matrix}\right) \qquad \text{as $z \to 1$}.
\end{gather*}
To summarize, $T(z)$ solves the following standard RHP:
\begin{enumerate}\itemsep=0pt
\item[(a)] $T(z)$ is analytic for $z \in \mathbb{C} \backslash [-1,1]$, % \label{TRHP1} \\
\item[(b)] $T^{(n)}$ has continuous boundary values for $x \in (-1,1)$, which we denote by~$T^{(n)}_+(x)$ and $T^{(n)}_-(x)$ and
\begin{gather} T_+(x) = T_-(x)\left(\begin{matrix}
\dfrac{F_+^2}{w} \phi_+^{-2n} & 1\\
0 & \dfrac{F_-^2}{w} \phi_-^{-2n}
\end{matrix}\right) \qquad \text{for $x \in (-1,1)$}, \label{TRHP2}
\end{gather}
\item[(c)] $T(z)$ has the following asymptotics as $z \to \infty$
\begin{gather*} T(z) =I + O_n\left(\frac{1}{z}\right) \qquad \text{as $z \to \infty$}, % \label{TRHP3}\\
\end{gather*}
\item[(d)] $T^{(n)}$ has the following asymptotics as $z \to -1$
\begin{gather*}
 T(z) = O_n\left(\begin{matrix}
1 & \log(\vert z + 1\vert) \\
1 & \log(\vert z + 1\vert)
\end{matrix}\right), % \label{TRHP4} \\
\end{gather*}
\item[(e)] $T^{(n)}$ has the following asymptotics as $z \to 1$
\begin{gather*}
 T(z) = O_n\left(\begin{matrix}
\log^{1/2}(\vert z-1\vert) & \log^{3/2}(\vert z - 1\vert) \\
\log^{1/2}(\vert z-1\vert) & \log^{3/2}(\vert z - 1\vert)
\end{matrix}\right). %\label{TRHP5}
\end{gather*}
\end{enumerate}
It follows from the definitions of $T$ (\ref{TDef}) and the fact that $\det(Y) \equiv 1$ that $\det(T) \equiv 1$.
\begin{Remark}
Properties \eqref{prop2-2.3} and \eqref{prop3-2.3} of Proposition \ref{FProps} imply that $\frac{F^2}{w}$ is bounded uniformly as we approach the interval $[-1,1]$ and therefore the jump matrix for $T$ in (\ref{TRHP2}) is indeed bounded.\label{F2/wRemark}
\end{Remark}

\subsection[The transformation $T \to Q$]{The transformation $\boldsymbol{T \to Q}$}\label{TToQ}

The transformations in these sections are well known to those familiar with the Riemann--Hilbert approach to orthogonal polynomials. The next transformation is a slight modification of the typical ``lens'' transformation. The difference is that it collects factors in the upper and lower lips near the singularity at~1. For our particular weight, these factors cancel each other to leading order, due to Proposition~\ref{FCancellation}. We ``deform'' the RHP on $[-1,1]$ to a problem on the oriented contour $\Sigma$ as depicted in Fig.~\ref{SigmaDef}.

\begin{figure}[t]\centering
\begin{tikzpicture}[scale=10.]
\fill (1,0) circle (.006cm);
\draw[line width = 1,directed] (1.1,0) to (1, 0);
\draw[line width = 1, loosely dotted] (.9,0.2) to (.8, .16);
\draw[line width = 1, loosely dotted] (.94,0.2) to (1, .108);
\draw[line width = 1, loosely dotted] (.9,-0.2) to (.8, -.16);
\draw[line width = 1, loosely dotted] (.94,-0.2) to (1, -.108);
\draw[line width = 1, directed] (0,0) to (1, 0);
\coordinate (up) at (.5,.2);
\coordinate (dn) at (.5,-.2);
\draw[line width = 1, domain=135:0 , directed] plot ({1 + .1* cos(\x) }, {.1 * sin(\x)});
\draw[line width = 1, domain=225:360 , directed] plot ({1+.1* cos(\x) }, {.1 * sin(\x)});
\draw[line width = 1,directed] (0,0) to [out=45,in=180] (up) to [out=0, in=180-45] (.932, .068);
\draw[line width = 1,directed] (0,0) to [out=-45,in=180] (dn) to [out=0, in=-180+45] (.932,-.068);
\node[below left] at (0,0) {$-1$};
\node[below left] at (1,0) {$1$};
\node[right] at (1.1, 0){$1+\delta$};
\node[above] at (.92,.2) {$\Sigma_1$};
\node[below] at (.92,-.2) {$\Sigma_2$};
\node [above] at (.2, .2) {$\Omega_0$};
\node [below right] at (.2, .1){$\Omega_1$};
\node [above right] at (.2, -.1){$\Omega_2$};
\node [above] at (.5, .2) {$+$};
\node [below] at (.5, .2) {$-$};
\node [above] at (.5, -.2) {$+$};
\node [below] at (.5, -.2) {$-$};
\node [above] at (.5, 0) {$+$};
\node [below] at (.5, 0) {$-$};
\node [above] at (1.05, 0) {$-$};
\node [below] at (1.05, 0) {$+$};
\end{tikzpicture}
\caption{Definition of $\Sigma = \Sigma_1 \cup \Sigma_2 \cup [-1,1] \cup [1+\delta, 1]$.} \label{SigmaDef}
\end{figure}
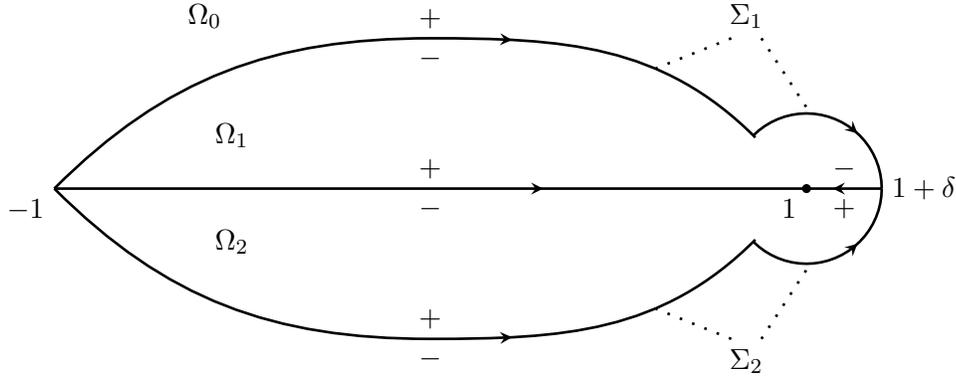

The transformation in this section is based on the factorization of the jump matrix of $T$ in~(\ref{TRHP2}) into a product of 3 matrices on $(-1,1)$:
\begin{gather*}
 \left(\begin{matrix}
\dfrac{F_+^2}{w} \phi_+^{-2n} & 1\\
0 & \dfrac{F_-^2}{w} \phi_-^{-2n}
\end{matrix}\right) = \left(\begin{matrix}
1& 0 \\
 \dfrac{F_-^2}{w} \phi_-^{-2n} & 1
\end{matrix}\right)\left(\begin{matrix}
0 & 1 \\
-1 & 0
\end{matrix}\right) \left(\begin{matrix}
1& 0 \\
 \dfrac{F_+^2}{w} \phi_+^{-2n} & 1
\end{matrix}\right).
\end{gather*}
Since $F$, $\phi$ are analytic in $\mathbb{C}\backslash [-1,1]$, and, by Lemma~\ref{NoFurtherComment}, $\log \frac{2k}{1-s}$ has a unique analytic continuation to $\mathbb{C} \backslash [1, \infty)$, the matrix valued function
\begin{gather*}
\left(\begin{matrix}
1 & 0\\
\dfrac{F^2}{w} \phi^{-2n}(z) & 1
\end{matrix}\right)
\end{gather*}
is an analytic function on $\mathbb{C} \backslash [-1, \infty)$ and, by Remark~\ref{F2/wRemark} and Proposition~\ref{PhiProps} is continuous up to the boundary. Now define
\begin{gather}
Q(z)=
\begin{cases}
T(z) & \text{for $z \in \Omega_0$},\\
T(z)\left(\begin{matrix}
1& 0 \\
- \dfrac{F^2}{w} \phi^{-2n} & 1
\end{matrix}\right) & \text{for $z \in \Omega_1$},\\
T(z)\left(\begin{matrix}
1& 0 \\
 \dfrac{F^2}{w} \phi^{-2n} & 1
\end{matrix}\right) & \text{for $z \in \Omega_2$}.
\end{cases}\label{QDef}
\end{gather}
A straightforward calculations shows $Q = Q^{(n)}$ solves the following RHP:
\begin{enumerate}\itemsep=0pt
\item[(a)] $Q(z)$ is analytic for $z \in \mathbb{C} \backslash \Sigma$, % \label{QRHP1} \\
\item[(b)] $Q$ has continuous boundary values for $s \in \Sigma\backslash \{-1, 1\}$ and
$ Q_+(s) = Q_-(s) v_{\Sigma} (s)$, $s \in \Sigma$, % \label{QRHP2}\\
 where
 \begin{gather*}
 v_{\Sigma}(s) = \begin{cases}
\left(\begin{matrix}
1& 0 \\
 \dfrac{F^2}{w} \phi^{-2n} & 1
\end{matrix}\right) &\text{for $s \in \Sigma_1 \cup \Sigma_2$}, \\
\left(\begin{matrix}
1& 0 \\
\left( \dfrac{F^2}{w_+} + \dfrac{F^2}{w_-}\right) \phi^{-2n} & 1
\end{matrix}\right) &\text{for $s \in [1, 1+\delta]$}, \\
\left(\begin{matrix}
0& 1 \\
-1 & 0
\end{matrix}\right) &\text{for $s \in [-1,1]$},
\end{cases}
\end{gather*}
\item[(c)] as $z \to \infty$
\begin{gather*} Q(z) =I + \frac{Q_1}{z} + O_n\left(\frac{1}{z^2}\right), % \label{QRHP3}\\
\end{gather*}
\item[(d)] $Q$ has the following asymptotics as $z \to -1$
\begin{gather}Q(z) = O_n\left(\begin{matrix}
\log(\vert z + 1\vert) & \log(\vert z + 1\vert) \\
\log(\vert z + 1\vert) & \log(\vert z + 1\vert)
\end{matrix}\right), \label{QRHP4}
\end{gather}
\item[(e)] $Q$ has the following asymptotics as $z \to 1$
\begin{gather}
 Q(z) = O_n\left(\begin{matrix}
\log^{3/2}(\vert z-1\vert) & \log^{3/2}(\vert z - 1\vert) \\
\log^{3/2}(\vert z-1\vert) & \log^{3/2}(\vert z - 1\vert)
\end{matrix}\right).\label{QRHP5}
\end{gather}
\end{enumerate}

The endpoint conditions in (\ref{QRHP4}) and (\ref{QRHP5}) guarantee that the solution $Q$ to the above RHP achieves its boundary values in the $L^2$ sense. The same argument as in Proposition \ref{Fokas} proves that this solution is unique. Therefore, any solution to the above RHP must be the solution we obtain from the RHP for $Y$ through the transformations~(\ref{TDef}) and~(\ref{QDef}).

\subsection[The associated Legendre problem, $\tilde{Q}$]{The associated Legendre problem, $\boldsymbol{\tilde{Q}}$}
The same steps as in Section~\ref{TToQ} can be repeated for the problem with the Legendre weight $\tilde{w}(s) = 1$, $s \in [-1,1]$, in which case all of the above calculations hold true provided we take $F \equiv 1$.

This results in the following RHP on $\Sigma$ for $\tilde{Q} = \tilde{Q}^{(n)}$:
\begin{enumerate}\itemsep=0pt
\item[(a)]
 \begin{gather}\tilde{Q}(z) \text{ is analytic for $z \in \mathbb{C} \backslash \Sigma$.} \label{TildeQRHP1}
 \end{gather}
\item[(b)] $\tilde{Q}$ has continuous boundary values for $x \in \Sigma \backslash\{-1, 1\}$ and
$\tilde{Q}_+(s) = \tilde{Q}_-(s) v_{\Sigma} (s)$
\begin{gather} \tilde{v}_{\Sigma}(s) = \begin{cases}
\left(\begin{matrix}
1& 0 \\
 \phi^{-2n} & 1
\end{matrix}\right) &\text{for $s \in \Sigma_1 \cup \Sigma_2$}, \\
\left(\begin{matrix}
1& 0 \\
2 \phi^{-2n} & 1
\end{matrix}\right) &\text{for $s \in [1, 1+\delta]$}, \\
\left(\begin{matrix}
0& 1 \\
-1 & 0
\end{matrix}\right) &\text{for $s \in [-1,1]$},
\end{cases} \label{TildeQRHP2}
\end{gather}
\item[(c)] as $z \to \infty$
\begin{gather*}
\tilde{Q}(z) =I + \frac{\tilde{Q}_1}{z} + O_n\left(\frac{1}{z^2}\right), %\label{TildeQRHP3}\\
\end{gather*}
\item[(d)] $\tilde{Q}$ has the following asymptotics as $z \to -1$
\begin{gather*}
\tilde{Q}(z) = O_n\left(\begin{matrix}
\log(\vert z + 1\vert) & \log(\vert z + 1\vert) \\
\log(\vert z + 1\vert) & \log(\vert z + 1\vert)
\end{matrix}\right), %\label{TildeQRHP4}
\end{gather*}
\item[(e)] $\tilde{Q}$ has the following asymptotics as $z \to 1$
\begin{gather}
 \tilde{Q}(z) =O_n\left(\begin{matrix}
\log(\vert z - 1\vert) & \log(\vert z - 1\vert) \\
\log(\vert z - 1\vert) & \log(\vert z - 1\vert)
\end{matrix}\right).\label{TildeQRHP5}
\end{gather}
\end{enumerate}

\subsection{A formula for recurrence coefficients}
\begin{prop}\label{RecCoeffFormula} The following formulas hold for the recurrence coefficients in terms of the modified RHPs $Q$, $\tilde{Q}$,
\begin{gather*}
a_n - \tilde{a}_n = \big(Q_1^{(n)}\big)_{11} -\big(\tilde{Q}_1^{(n)}\big)_{11} - \big( \big(Q_1^{(n+1)}\big)_{11}- \big(\tilde{Q}_1^{(n+1)}\big)_{11}\big),\\
b_{n-1}^2 - \tilde{b}_{n-1}^2 = \big(\big(Q_1^{(n)}\big)_{12} - \big(\tilde{Q}_1^{(n)}\big)_{12}\big) \big(\big(Q_1^{(n)}\big)_{21} - \big(Q_1^{(n+1)}\big)_{21}\big)\nonumber \\
\hphantom{b_{n-1}^2 - \tilde{b}_{n-1}^2 =}{} + \big(\tilde{Q}_1^{(n)}\big)_{12} \big[ \big(Q_1^{(n)}\big)_{21} - \big(Q_1^{(n+1)}\big)_{21}- \big(\big(\tilde{Q}_1^{(n)}\big)_{21} - \big(\tilde{Q}_1^{(n+1)}\big)_{21}\big)\big].
\end{gather*}
\end{prop}
\begin{proof}Note, from Proposition \ref{Referee1}, that the recurrence coefficients $a_n$, $b_n$ are related to the RHP for $Y$ via the formulas
\begin{gather}
a_n = \big(Y^{(n)}_1\big)_{11} - \big(Y^{(n+1)}_1\big)_{11}, \label{anY}\\
b_{n-1}^2 = \big(Y^{(n)}_1\big)_{12}\big(Y^{(n)}_1\big)_{21},\label{bnY}
\end{gather}
and, again from Proposition \ref{Referee1}, $Y$ satisfies the following symmetry relation:
\begin{gather}
1 = \big(Y^{(n)}_1\big)_{12}\big(Y^{(n+1)}_1\big)_{21}.\label{symY}
\end{gather}
Now for $z$ sufficiently large (outside of the contour $\Sigma$)
\begin{gather*}
\left(\frac{2^n}{F_\infty}\right)^{\sigma_3}Y(z)\left(\frac{F(z)}{\phi^n(z)}\right)^{\sigma_3} = Q(z)
\end{gather*}
and
\begin{gather*}
\left(I + \frac{Y_1}{z} + O_n\big(z^{-2}\big)\right)z^{n\sigma_3} = \left(\frac{F_\infty}{2^n}\right)^{\sigma_3}\left(I + \frac{Q_1}{z} + O_n\big(z^{-2}\big)\right)\left(\frac{\phi^n(z)}{F(z)}\right)^{\sigma_3}.
\end{gather*}
By Propositions \ref{PhiProps} and \ref{FProps}
\begin{gather*}
\phi(z) = 2z + O_n\big(z^{-1}\big), \\
F(z) = F_\infty + \frac{F_1}{z} + O_n\big(z^{-2}\big).
\end{gather*}
Hence
\begin{gather*}
\frac{\phi^n(z)}{F(z)}= \frac{2^n}{F_\infty+ \frac{F_1}{z} +O_n\big(z^{-2}\big)} z^n + O_n\big(z^{n-2}\big)
 = \frac{2^n}{F_\infty} z^n -\frac{2^nF_1}{F_\infty^2}z^{n-1} + O_n\big(z^{n-2}\big),\\
\frac{F(z)}{\phi^n(z)}=\frac{F_\infty}{2^nz^n + O_n\big(z^{n-2}\big)} +\frac{F_1}{2^nz^{n+1} + O_n(z^{n-1})} + O_n\big(z^{n-2}\big)\nonumber\\
\hphantom{\frac{\phi^n(z)}{F(z)}}{} = \frac{F_\infty}{2^n} z^{-n}+ \frac{F_1}{2^n}z^{-n-1} + O_n\big(z^{-n-2}\big).
\end{gather*}
So
\begin{gather*}
\left(I + \frac{Y_1}{z} + O_n\big(z^{-2}\right)\big)z^{n\sigma_3} = \left(\frac{F_\infty}{2^n}\right)^{\sigma_3}\left( I + \frac{Q_1}{z} + O_n\big(z^{-2}\right)\big) \nonumber\\
\hphantom{\left(I + \frac{Y_1}{z} + O_n\big(z^{-2}\right)\big)z^{n\sigma_3} =}{} \times\left(\left(\frac{2^n}{F_\infty}\right)^{\sigma_3} +\frac{1}{z} \left(\begin{matrix}
 - \frac{2^nF_1}{F_\infty^2} & 0\\
0 & \frac{F_1}{2^n}
\end{matrix}\right) + O_n\big(z^{-2}\right)\big)z^{n\sigma_3},
\\
 \frac{Y_1}{z}=\frac{1}{z} \left(\frac{F_\infty}{2^n}\right)^{\sigma_3}Q_1\left(\frac{2^n}{F_\infty}\right)^{\sigma_3} + \frac{1}{z} \left(\begin{matrix}
 - \dfrac{F_1}{F_\infty} & 0\\
0 & \dfrac{F_1}{F_\infty}
\end{matrix}\right), \\
 Y_1=\left(\frac{F_\infty}{2^n}\right)^{\sigma_3}Q_1\left(\frac{2^n}{F_\infty}\right)^{\sigma_3} + \left(\begin{matrix}
 - \dfrac{F_1}{F_\infty} & 0\\
0 & \dfrac{F_1}{F_\infty}
\end{matrix}\right),
\end{gather*}
and so the formula for the recurrence coefficients in (\ref{anY})--(\ref{bnY}) can be expressed in terms of the RHP for $Q$ by
\begin{gather*}
a_n = \big(Q_1^{(n)}\big)_{11} - \big(Q_1^{(n+1)}\big)_{11}, \qquad
b_{n-1}^2 = \big(Q_1^{(n)}\big)_{12} \big(Q_1^{(n)}\big)_{21}.
\end{gather*}
The symmetry consideration (\ref{symY}) becomes
\begin{gather*}
1 = 4 \big(Q_1^{(n)}\big)_{12} \big(Q_1^{(n+1)}\big)_{21}
\end{gather*}
and so
\begin{gather*}
a_n = \big(Q_1^{(n)}\big)_{11} - \big(Q_1^{(n+1)}\big)_{11}, \\
b_{n-1}^2 - \frac{1}{4} = \big(Q_1^{(n)}\big)_{12} \big(\big(Q_1^{(n)}\big)_{21} - \big(Q_1^{(n+1)}\big)_{21}\big).
\end{gather*}
The same formulas hold for $\tilde{a}_{n}$ and $\tilde{b}_{n-1}^2$, replacing $Q_1$, $Y_1$ with $\tilde{Q}_1$, $\tilde{Y}_1$. Therefore,
\begin{gather*}
a_n - \tilde{a}_n = \big(Q_1^{(n)}\big)_{11} -\big(\tilde{Q}_1^{(n)}\big)_{11} - \big( \big(Q_1^{(n+1)}\big)_{11}- \big(\tilde{Q}_1^{(n+1)}\big)_{11}\big)\\
b_{n-1}^2 - \tilde{b}_{n-1}^2 = \big(Q_1^{(n)}\big)_{12} \big(\big(Q_1^{(n)}\big)_{21} - \big(Q_1^{(n+1)}\big)_{21}\big)- \big(\tilde{Q}_1^{(n)}\big)_{12} \big(\big(\tilde{Q}_1^{(n)}\big)_{21} - \big(\tilde{Q}_1^{(n+1)}\big)_{21}\big) \\
\hphantom{b_{n-1}^2 - \tilde{b}_{n-1}^2}{} = \big(\big(Q_1^{(n)}\big)_{12} - \big(\tilde{Q}_1^{(n)}\big)_{12}\big) \big(\big(Q_1^{(n)}\big)_{21} - \big(Q_1^{(n+1)}\big)_{21}\big) \\
\hphantom{b_{n-1}^2 - \tilde{b}_{n-1}^2=}{} + \big(\tilde{Q}_1^{(n)}\big)_{12} \big[ \big(Q_1^{(n)}\big)_{21} - \big(Q_1^{(n+1)}\big)_{21}- \big(\big(\tilde{Q}_1^{(n)}\big)_{21} - \big(\tilde{Q}_1^{(n+1)}\big)_{21}\big)\big]
\end{gather*}
as desired.
\end{proof}

\section{A comparison of resolvents}\label{MachinerySection}
In this section we develop the machinery that will allow us to compare two Riemann--Hilbert problems using non-local information.
\subsection{Comparing Riemann--Hilbert problem through the Cauchy transform}
For a bounded, oriented, rectifiable contour $\Sigma$, consider $C_\Sigma$ acting on $L^p(\Sigma)$, $1 \le p < \infty$, as follows
\begin{gather*}
(C_\Sigma f)(z) = \int_\Sigma \frac{f(s)}{s-z} \dbar s,
\end{gather*}
where recall $\dbar s = \frac{{\rm d}s}{2\pi i}$. For almost all $z \in \Sigma$, the limits
\begin{gather*}
(C_\Sigma^\pm f)(z) := \lim_{z' \to z^\pm} \int_\sigma \frac{f(s)}{s - z'} \dbar s
\end{gather*}
exist. Furthermore,
\begin{gather*}
(C^+f)(z) - (C^-f)(z) = f(z)%\label{What}
\end{gather*}
for almost all $z \in \Sigma$, and
\begin{gather*}
(C^+f)(z) + (C^-f)(z) = iHf(z),
\end{gather*}
where $H$ is the Hilbert transform
\begin{gather*}
Hf(z) := \lim_{\epsilon \to 0} \int_{\substack{\Sigma\\ \vert z - z'\vert > \epsilon}} \frac{f(z')}{z' - z} \frac{{\rm d}z'}{\pi},%\label{No}
\end{gather*}
which exists for almost all $z \in \Sigma$. An oriented, rectifiable contour $\Sigma$ is called Carleson (or Ahlfors--David) if
\begin{gather*}
K_\Sigma = \sup_{\substack{z \in \Sigma \\ r > 0}} \frac{\left\vert U_r(z) \cap \Sigma \right\vert}{r} < \infty,
\end{gather*}
where $U_r(z)$ denotes a ball of radius $r$ centered at $z$.

The fundamental theorem, starting with the work of Calderon \cite{Calderon}, continuing through the work of Coifman, Meyer, and McIntosh \cite{ CMM}, and culminating in the work of Guy David \cite{David}, is the following: The operators $C_\Sigma^\pm$ are bounded from $L^2(\Sigma) \to L^2(\Sigma)$ (and in fact in $L^p(\Sigma)$ for all $1 < p < \infty$) if and only if $\Sigma$ is Carleson. Note, in particular, that the contour in this paper is clearly Carleson.

We will say that a pair of functions $A_\pm \in L^p(\Sigma)$ belongs to $\partial C(L^p(\Sigma))$ iff $A_\pm = C_\Sigma^\pm f$ for some (unique) $f \in L^p(\Sigma)$.
\begin{prop}
\label{ABCR}
Let $\Sigma$ be a Carleson curve, and let $A_\pm \in I+ \partial C(L^p(\Sigma))$ and $B_\pm \in I +\partial C(L^q(\Sigma))$ for $1 < p,q < \infty$ where $1\ge \frac{1}{r} = \frac{1}{p} + \frac{1}{q}$. Then if~$A(z)$,~$B(z)$ are the extensions of~$A_\pm$,~$B_\pm$ respectively, then $(AB)(z)$ is the extension of
\begin{gather*}
(AB)_\pm \in I + \partial C\big(L^p(\Sigma)+L^q(\Sigma)+L^r(\Sigma)\big).
\end{gather*}
If $\Sigma$ is compact, then
\begin{gather*}
(AB)_\pm \in I + \partial C\big(L^r(\Sigma)\big).
\end{gather*}
\end{prop}
\begin{proof}
Suppose $f \in L^p(\Sigma)$ and $g \in L^q(\Sigma)$, then $f(Hg), (Hf)g \in L^r(\Sigma)$, and
\begin{gather*}
\frac{i}{2}\left[\left(C_\Sigma f(Hg) + C_\Sigma (Hf) g\right)\right](z) \nonumber\\
\qquad {}= \int_\Sigma f(s) \lim_{\epsilon\to 0}\left(\int_{\vert s-t\vert > \epsilon} g(t) \frac{\dbar t}{t-s}\right)\frac{ \dbar s}{s-z} + \int_\Sigma \lim_{\epsilon\to 0}\left(\int_{\vert s-t\vert > \epsilon} f(s) \frac{\dbar s}{s-t} \right)g(t) \frac{\dbar t}{t-z}\nonumber\\
\qquad{} = \lim_{\epsilon \to 0}\int_{\substack{s,t \in \Sigma\\
\vert s-t\vert > \epsilon}} \frac{f(s)g(t) }{t-s} \left(\frac{1}{s-z} - \frac{1}{t-z}\right)\dbar s\dbar t = \lim_{\epsilon \to 0}\int_{\substack{s,t \in \Sigma\\
\vert s-t\vert > \epsilon}} f(s) g(t) \frac{1}{s-z} \frac{1}{t-z} \dbar s \dbar t \nonumber\\
\qquad{} = (C_\Sigma fC_\Sigma g)(z),
\end{gather*}
where we have used the fact that the Hilbert transform converges in $L^p(\Sigma)$ and $L^q(\Sigma)$. Therefore, if $A_\pm = I + C_\Sigma^\pm f$ and $B_\pm = I + C_\Sigma^\pm g$, (and hence $A = I + C_\Sigma f$, $B = I+ C_\Sigma g$), then $(AB)(z)$ is the extension of
\begin{gather*}
(AB)_\pm = I + C_\Sigma^\pm f + C_\Sigma^\pm g + C_\Sigma^\pm f C_\Sigma^\pm g = I + C_\Sigma^\pm \left( f + g + \frac{i}{2}( f(Hg) + (Hf)g)\right)
\end{gather*}
and, since $f \in L^p(\Sigma), g \in L^q(\Sigma)$,
\begin{gather*}
f+ g+ \frac{i}{2}( f(Hg) + (Hf)g) \in L^p(\Sigma)+L^q(\Sigma)+L^r(\Sigma),
\end{gather*}
which completes the proof.
\end{proof}
\begin{Remark}
Note in the case $r = 1$ that $\partial C(L^1(\Sigma))$ is not contained in $L^1(\Sigma)$.
\end{Remark}
\begin{Remark}
Suppose $A_\pm \in I+ \partial C(L^p(\Sigma))$ and suppose that $A_\pm^{-1}(x)$ exists a.e. and lies in $I + \partial C(L^q(\Sigma))$, $\frac{1}{r} = \frac{1}{p} + \frac{1}{q} \le 1$. Let $A(z), B(z)$ be the analytic extensions of $A_\pm, A_\pm^{-1}$ off $\Sigma$ respectively. Then $A(z)$ is invertible in $\mathbb{C}\backslash \Sigma$ and $A^{-1}(z) = B(z)$. Indeed by Proposition \ref{ABCR} $(AB)_\pm = I+ C_\Sigma^\pm h$ for some $h \in L^p(\Sigma) +L^q(\Sigma) + L^r(\Sigma)$, but $(AB)_+ = (AB)_- = I$ and so $h = C_\Sigma^+h - C_\Sigma^-h = 0$. Thus $AB(z) = A(z)B(z) = I$. \label{InverseRemark}
\end{Remark}
\begin{thm}\label{ABInverse}
Suppose $v_A, v_B \in L^\infty(\Sigma)$ and $A_\pm$ solves the RHP:
\begin{gather*}
A_+(x) = A_-(x) v_A(x), \qquad \text{for $x \in \Sigma$}, \qquad A_\pm \in I + \partial C\big( L^p(\Sigma)\big),
\end{gather*}
and $B_\pm$ the RHP:
\begin{gather*}
B_+(x) = B_-(x) v_B(x), \qquad \text{for $x \in \Sigma$}, \qquad B_\pm \in I + \partial C\big( L^p(\Sigma)\big).
\end{gather*}
Suppose further that $v_B^{-1} \in L^\infty(\Sigma)$ and $B_\pm^{-1} \in I + \partial C(L^q(\Sigma))$ where $\frac{1}{p} +\frac{1}{q} = 1$, then if $A, B$ are the analytic extensions of $A_\pm$ and $B_\pm$ respectively,
\begin{gather*}
AB^{-1} = I + C_\Sigma A_-\big(v_A v_B^{-1} - I\big)B_-^{-1}.
\end{gather*}
\end{thm}
\begin{proof}
The following calculation holds almost everywhere on $\Sigma$,
\begin{gather*}
 C_\Sigma^+ A_-\big(v_A v_B^{-1} - I\big)B_-^{-1} = C_\Sigma^- A_-\big(v_A v_B^{-1} - I\big)B_-^{-1} + A_-\big(v_Av_B^{-1} - I\big) B_-^{-1} \nonumber\\
\hphantom{C_\Sigma^+ A_-\big(v_A v_B^{-1} - I\big)B_-^{-1}}{} = C_\Sigma^- A_-\big(v_A v_B^{-1} - I\big)B_-^{-1} + A_+B_+^{-1} - A_-B_-^{-1}.
\end{gather*}
Therefore, almost everywhere on $\Sigma$,
 \begin{gather}
 I+C_\Sigma^+ A_-\big(v_A v_B^{-1} - I\big)B_-^{-1} - A_+B_+^{-1} = I+C_\Sigma^- A_-\big(v_A v_B^{-1} - I\big)B_-^{-1} - A_-B_-^{-1}. \label{IsEntire}
\end{gather}
However, using Proposition \ref{ABCR},
\begin{gather*}
I + C_\Sigma^\pm A_-\big(v_A v_B^{-1} - I\big)B_-^{-1} - \big(AB^{-1}\big)_\pm \in \partial C\big(L^1(\Sigma)\big).
\end{gather*}
Thus
\begin{gather*}
I + C_\Sigma^\pm A_-\big(v_A v_B^{-1} - I\big)B_-^{-1} - \big(AB^{-1}\big)_\pm= C_\Sigma^\pm h
\end{gather*}
for some $h \in L^1(\Sigma) + L^p(\Sigma) + L^q(\Sigma)$. However, (\ref{IsEntire}) then implies that $h\equiv 0$, and therefore
\begin{gather*}
I + C_\Sigma A_-\big(v_A v_B^{-1} - I\big)B_-^{-1} - AB^{-1} \equiv 0
\end{gather*}
completing the proof.
\end{proof}

Proposition \ref{ABInverse} clearly implies for $z \in \mathbb{C}\backslash \Sigma$,
\begin{gather*}
A(z) = B(z) + \big[C_\Sigma A_-\big(v_A v_B^{-1} - I\big)B_-^{-1}\big](z)B(z)\nonumber\\
\hphantom{A(z)}{} = B(z) + \big[C_\Sigma B_-\big(v_A v_B^{-1} - I\big)B_-^{-1}\big](z)B(z)\\
\hphantom{A(z)=}{}+ \big[C_\Sigma (A_--B_-)\big(v_A v_B^{-1} - I\big)B_-^{-1}\big](z)B(z).%\label{AToB}
\end{gather*}
If it is possible to control the $L^2(\Sigma)$ difference $A_--B_-$ then quantities related to $A$ away from the contour $\Sigma$ can be estimated by quantities that only rely on $B$ and the jump func\-tions~$v_A$,~$v_B$. The results of this section will lay the groundwork of controlling this difference, with the remainder of the work being done in the following section.

For compact $\Sigma$ and $f \in L^p(\Sigma)$, $p \ge 1$, we have that $(C_\Sigma f)(z) = O\big(\frac{1}{z}\big)$ as $z \to \infty$. The contour $\Sigma$ defined in Section~\ref{RHPsSection} is clearly compact, we will use this fact repeatedly without further comment.

\subsection{The resolvent bounds}\label{Machinery1}
\qquad Let $h,f$ refer to matrix valued functions, $h \in L^\infty(\Sigma)$ and $f \in L^2(\Sigma)$, and define the operator $C_{h}$ acting on $L^2(\Sigma)$ by
\begin{gather*}
(C_h f)(x) = C_\Sigma^- (f (h-I)) (x) = \lim_{z \to x^-} \int_\Sigma \frac{f(s)(h(s)-I)}{s-z} \dbar s \qquad \text{for $x\in \Sigma$}.
\end{gather*}
Clearly $C_h$ is bounded from $L^2(\Sigma)\to L^2(\Sigma)$. The importance of this definition lies in its role within the general theory of Riemann--Hilbert problems, see for example~\cite{ClancyGohberg,DeiftNLS} and the references therein. In particular, let $\mu$ be any $L^2(\Sigma)$ solution of the equation
\begin{gather}
(1 - C_{v_\Sigma}) \mu = I .\label{Something300}
\end{gather}
Note that $I\in L^2(\Sigma)$ since $\Sigma$ is bounded. If $1-C_{v_\Sigma}$ is an invertible operator from~$L^2(\Sigma)$ to~$L^2(\Sigma)$, then
\begin{gather}
\mu = (1-C_{v_\Sigma})^{-1}I \label{ResolventIntro}
\end{gather}
is the unique solution to (\ref{Something300}). In this case the unique solution $X$ to the RHP normalized at infinity with jump $v_\Sigma$ is given by
\begin{gather*}
X(z) = I + C_\Sigma (\mu (v_\Sigma - I)).%\label{13}
\end{gather*}
Our \textit{first} step to controlling the difference $Q-\tilde{Q}$, will be showing that the resolvent in~(\ref{ResolventIntro}) is bounded uniformly of~$n$ for the following modified version of the Legendre Riemann--Hilbert problem:

Define, as in (\ref{TildeQRHP2}), for the same contour $\Sigma$, and subcontours $\Sigma_1$, $\Sigma_2$ as depicted in Fig.~\ref{SigmaDef},
\begin{gather*}
\tilde{v}_\Sigma = \begin{cases}
\left(\begin{matrix}
1 & 0\\
\phi^{-2n}(s) & 1
\end{matrix}\right) & \text{for $s \in \Sigma_1\cup\Sigma_2$},\\
\left(\begin{matrix}
1 & 0\\
2\phi^{-2n}(s) & 1
\end{matrix}\right) & \text{for $s \in [1, 1+\delta]$},\\
\left(\begin{matrix}
0 & 1\\
-1 & 0
\end{matrix}\right) & \text{for $s \in [-1,1]$},
\end{cases}
\end{gather*}
then we have the following theorem:
\begin{thm}\label{LegOpBound}
The operator $1-C_{\tilde{v}_\Sigma}$ is invertible for all $n\ge 0$. Moreover $(1-C_{\hat{v}_\Sigma})^{-1}$ is bounded uniformly of $n$ in operator norm from $L^2(\Sigma) \to L^2(\Sigma)$.
\end{thm}
\begin{proof}
The following proposition follows from Proposition 2.14 from \cite{DeiftNLS}, see also \cite{ClancyGohberg}.
\begin{prop}\label{NLSProp}
A family of operators $(1 - C_{h_n})^{-1}$ is bounded uniformly with respect to $n$ if and only if the inhomogeneous RHP
\begin{gather*}
m_+(s) = m_- (s)h_n(s) + g(s) \qquad \text{for $s \in \Sigma$}, \quad m_\pm \in \partial C\big(L^2(\Sigma)\big)
\end{gather*}
is $($uniquely$)$ solvable for all $g \in L^2 (\Sigma)$ with $\vert\vert m_\pm\vert\vert_{L^2} \le c\vert\vert g\vert\vert_{L^2}$ with $c$ independent of both~$n$ and~$g$.
\end{prop}
We want to apply Proposition \ref{NLSProp} to $h_n = \tilde{v}_\Sigma$. To this end, suppose a solution $m_\pm(s) = m_\pm(g,n,s)$ exists to
\begin{gather*}
m_+(s) = m_- (s)\tilde{v}_\Sigma(s) + g(s) \qquad \text{for $s \in \Sigma$}, \qquad m_\pm \in \partial C\big(L^2(\Sigma)\big) %\label{inhomRHP}
\end{gather*}
and note $m(z) := (C_\Sigma f)(z)$ for $z \in \mathbb{C}\backslash \Sigma$ is the analytic continuation of~$m_\pm$. It will follow from the calculation below that such a solution indeed exists. For $z \in \mathbb{C}\backslash [-1,1]$, let
\begin{gather*}
\tilde{v}(z) = \left(\begin{matrix}
1 & 0\\
\phi^{-2n}(z) & 1
\end{matrix}\right)
\end{gather*}
 and, in the notation of Fig.~\ref{SigmaDef}, set
\begin{gather*}
t(z) = \begin{cases}
m(z) - C_\Sigma \tilde{g}(z) & \text{for $z \in \Omega_0$},\\
m(z)\tilde{v}(z) - C_\Sigma \tilde{g}(z) & \text{for $z \in \Omega_1$},\\
m(z) \tilde{v}^{-1}(z)- C_\Sigma \tilde{g}(z) & \text{for $z \in \Omega_2$},
\end{cases}
\end{gather*}
where
\begin{gather*}
\tilde{g}(s) = \begin{cases}
g(s) & \text{for $s \in \Sigma_1$},\\
g(s) \tilde{v}^{-1}(s) & \text{for $s \in \Sigma_2$},\\
g(s) \tilde{v}^{-1}(s)& \text{for $s \in [1, 1+\delta]$},\\
0 & \text{for $s \in [-1,1]$}.
\end{cases}
\end{gather*}

Since $\tilde{v}(s)$ is uniformly bounded, it is clear that $\Vert \tilde{g}\Vert_{L^2(\Sigma)} \le c \Vert g \Vert_{L^2(\Sigma)}$, and it is also clear that $\Vert m_\pm \Vert_{L^2(\Sigma)} \le c \Vert t_\pm\Vert_{L^2(\Sigma)} + c\Vert g \Vert_{L^2(\Sigma)}$. For $s \in\Sigma_1, \Sigma_2$, or $[1, 1+\delta]$, straightforward calculations will show that $t_+(s) = t_-(s)$. Therefore $t(z)$ is actually analytic for $z \in \mathbb{C}\backslash [-1,1]$. Lastly, for $z \in (-1,1)$, first note that $\tilde{g} = 0$, and therefore $C_\Sigma^+\tilde{g} = C_\Sigma^-\tilde{g}$, which we label $C_\Sigma \tilde{g}$. Another calculations shows:
\begin{gather*}
t_+ = t_- \left(\begin{matrix}
\left(\dfrac{\phi_-}{\phi_+}\right)^n & 1\\
0 & \left(\dfrac{\phi_+}{\phi_-}\right)^n
\end{matrix}\right) + C_\Sigma \tilde{g}\left(\begin{matrix}
\left(\dfrac{\phi_-}{\phi_+}\right)^n & 1\\
0 & \left(\dfrac{\phi_+}{\phi_-}\right)^n
\end{matrix}\right) + g\left(\begin{matrix}
1 & 0 \\
\phi_+^{-2n} & 1
\end{matrix}\right) - C_\Sigma\tilde{g}.
\end{gather*}
So $t(z)$ solves a RHP:
\begin{enumerate}\itemsep=0pt
\item[(a)]
\begin{gather}
t(z) \text{ is analytic for $z \in \mathbb{C} \backslash [-1,1]$} ,\label{TsmallRHP1}
\end{gather}
\item[(b)] for $x \in (-1,1)$,
\begin{gather}
 t_+(x) =t_-(x)\left(\begin{matrix}
\left(\dfrac{\phi_-}{\phi_+}\right)^n & 1\\
0 & \left(\dfrac{\phi_+}{\phi_-}\right)^n
\end{matrix}\right) + \hat{g}(x), \label{TsmallRHP2}
\end{gather}
\item[(c)] as $z \to \infty$
\begin{gather}
t(z) =O_n\left(\frac{1}{z}\right),\label{TsmallRHP3}
\end{gather}
\end{enumerate}
where
\begin{gather*}
\hat{g} = C_\Sigma \tilde{g}\left(\begin{matrix}
\left(\dfrac{\phi_-}{\phi_+}\right)^n & 1\\
0 & \left(\dfrac{\phi_+}{\phi_-}\right)^n
\end{matrix}\right) + g\left(\begin{matrix}
1 & 0 \\
\phi_+^{-2n} & 1
\end{matrix}\right) - C_\Sigma\tilde{g}.
\end{gather*}
Since the restriction of $C_\Sigma$ to $[-1, 1]$ is bounded as an operator from $L^2(\Sigma)$ to $L^2([-1,1])$, and since $\vert \phi_\pm(x)\vert = 1$ for $x \in [-1,1]$. It is clear that $\Vert \hat{g}\Vert_{L^2([-1,1])} \le c\Vert g \Vert_{L^2(\Sigma)}$.

Next, let
\begin{gather*}
u(z) = t(z) \phi^{n \sigma_3}(z)
\end{gather*}
it is clear that $\Vert t_\pm\Vert_{L^2([-1,1])}\le \Vert u_\pm\Vert_{L^2([-1,1])}$. The jump conditions becomes
\begin{gather*}
u_+(x) = u_-(x) \left(\begin{matrix}
1 & 1 \\
0 & 1
\end{matrix}\right)+ \hat{g}\phi_+^{n\sigma_3}
\end{gather*}
for $x \in [-1,1]$, and, as $z \to \infty$,
\begin{gather*}
u(z)= t(z) \phi^{n\sigma_3}(z) = O_n\left(\frac{1}{z}\right) \left(\begin{matrix}
z^n + O_n(z^{n-1}) & 0\\
0 & z^{-n} + O_n\big(z^{-n-1}\big)
\end{matrix}\right) \\
\hphantom{u(z)}{} =O_n\left(\begin{matrix}
z^{n-1} & z^{-n-1} \\
z^{n-1} & z^{-n-1}
\end{matrix}\right).
\end{gather*}
In particular, a solution to the RHP described below can be transformed in a bounded fashion to the solution~$m$,
\begin{gather*}
u(z) \text{ is analytic in $\mathbb{C} \backslash [-1,1]$}, \\
u_+(x) = u_-(x) \left(\begin{matrix}
1 & 1 \\
0 & 1
\end{matrix}\right)+ \hat{g}\phi_+^{n\sigma_3}(x)\text{ for }x \in [-1,1], \\
u(z) =O_n\left(\begin{matrix}
z^{n-1} & z^{-n-1} \\
z^{n-1} & z^{-n-1}
\end{matrix}\right) \text{ as $z \to \infty$}.%\label{inhomRHPU}
\end{gather*}
This RHP can be solved as follows. Let $\hat{g} = (\hat{g}_1, \hat{g}_2)$ and $u = (u_1,u_2)$ where $\hat{g}_1$, $\hat{g}_2$, $u_1$, $u_2$ are the column vectors of~$\hat{g}$ and~$u$. Then
\begin{gather*}
u_{1+} = u_{1-} + \hat{g}_1 \phi_+^n, \qquad
u_{2+} = u_{2-} + u_{1-} + \hat{g}_2 \phi_+^n.
\end{gather*}
Thus, since $u_1 = O_n(z^{n-1})$ as $z \to \infty$, any polynomial $p_{n-1}$ with $\deg(p_{n-1}) \le n-1$ provides a potential solution for $u_1$ through Plemelj's formula,
\begin{gather}
u_1(z) = p_{n-1}(z) + \int_{-1}^1 \frac{\hat{g}_1(s) \phi_+^n(s)}{s-z} \dbar s\label{U1Def}
\end{gather}
and, since $u_2 = O_n\big(z^{-n-1}\big)$, if the solution exists, it must be, again by Plemelj's formula,
\begin{gather}
u_2 = \int_{-1}^1 \frac{u_{1-}(s) + \hat{g}_2(s) \phi_+^{-n}(s)}{s-z}\dbar s\nonumber\\
\hphantom{u_2}{} = -\frac{1}{z}\int_{-1}^1\left( u_{1-}(s) + \hat{g}_2(s) \phi_+^{-n}(s)\right)\sum_{k=0}^\infty \left( \frac{s}{z}\right)^k\dbar s.\label{U2Def}
\end{gather}
The condition $u_2 = O_n\big(z^{-n-1}\big)$ is therefore equivalent to a system of $n$ equations for the $n$ coefficients of $p_{n-1}$. Specifically, we have
\begin{gather}
\int_{-1}^1 s^k\big(u_{1-}(s) + \hat{g}_2(s)\phi_+^{-n}\big){\rm d}s = 0 \qquad \text{for $0 \le k \le n-1$},\nonumber\\
\int_{-1}^1 s^kp_{n-1}(s) {\rm d}s = -\int_{-1}^1 \big(C_{\Sigma}^- \hat{g} \phi_+^n\big)(s) +\hat{g}_2(s)\phi_+^{-n}{\rm d}s \qquad \text{for $0 \le k \le n-1$}.\label{MSomething1}
\end{gather}
As $\det\big(\int_{-1}^1 s^{k+j}{\rm d}s \big)_{k,j = 0}^{n-1} \neq 0$, equation (\ref{MSomething1}) is a system of $n$ independent, linear equations for the $n$ unknowns, $a_0, a_1, \dots, a_{n-1}$, where
\begin{gather*}
p_{n-1}(s) = a_0 + a_1s + \dots + a_{n-1}s^{n-1}.
\end{gather*}
It follows that the solution $u$ exists and is unique. Moreover,
\begin{gather*}
 \int_{-1}^1 s^k\big(u_{1-}(s) + \hat{g}_2(s)\phi_+^{-n}\big){\rm d}s = 0 \qquad \text{for $0 \le k \le n-1$}, \nonumber\\
 \int_{-1}^1 P_k(s)\big(p_{n-1}(s) + C_{[-1,1]}^-(\hat{g}_1 \phi_+^{n}) + \hat{g}_2(s)\phi_+^{-n}\big){\rm d}s = 0.
 \end{gather*}
For all polynomials with $\deg (P_k) = k\le n-1$
 \begin{gather*}
 \int_{-1}^1 \big(\vert p_{n-1}\vert^2(s) + \overline{p_{n-1}(s)}C_{[-1,1]}^-(\hat{g}_1 \phi_+^{n}) + \overline{p_{n-1}(s)}\hat{g}_2(s)\phi_+^{-n}\big){\rm d}s = 0,\\
\vert\vert p_{n-1}\vert\vert_{L^2(-1,1)}^2= -\int_{-1}^1 \overline{p_{n-1}(s)}\big( C_{[-1,1]}^-(\hat{g}_1 \phi_+^{n}) +\hat{g}_2(s)\phi_+^{-n}\big){\rm d}s \nonumber\\
\hphantom{\vert\vert p_{n-1}\vert\vert_{L^2(-1,1)}^2}{} \le \vert\vert p_{n-1}\vert\vert_{L^2} \vert\vert C_{[-1,1]}^-(\hat{g}_1 \phi_+^n) + \hat{g}_2 \phi_+^{-n}\vert\vert_{L^2}. %\label{137M}
\end{gather*}
So, using the fact that $\vert \phi_+\vert = 1$ and the boundedness of the Cauchy operator, we have
\begin{gather*}
\vert \vert p_{n-1} \vert \vert_{L^2} \le c\vert \vert \hat{g}_2 \vert \vert_{L^2}.
\end{gather*}
Most importantly, we notice that $c$ is independent of $n$. Therefore, by using the boundedness of the Cauchy transform and the fact that $\vert \phi_\pm\vert = 1$, as well as~(\ref{U1Def}) and~(\ref{U2Def}), we see that,
\begin{gather*}
\vert \vert (u_1, u_2)_\pm \vert \vert_{L^2([-1,1])} \le c'' \vert \vert \hat{g}\vert\vert_{L^2(\Sigma)} \le c'''\vert \vert g\vert\vert_{L^2(\Sigma)},
\end{gather*}
where all the constants are independent of $n$. Up until now, we have been assuming solutions~$m$ and~$t$ existed. However, reversing our construction,
\begin{gather*}
t(z) = u(z) \phi^{-n\sigma_3}
\end{gather*}
is a solution to the RHP for $t$, (\ref{TsmallRHP1})--(\ref{TsmallRHP3}) with $\Vert t_\pm \Vert_{L^2([-1,1])} \le \Vert u_\pm\Vert_{L^2([-1,1])} \le c \Vert g \Vert_{L^2(\Sigma)}$, and
\begin{gather*}
m(z) = \begin{cases}
t(z) + C_\Sigma \tilde{g}(z) & \text{for $z \in \Omega_0$},\\
(t(z)+ C_\Sigma \tilde{g}(z)) \tilde{v}^{-1}(z) & \text{for $z \in \Omega_1$},\\
(t(z)+ C_\Sigma \tilde{g}(z)) \tilde{v}(z) & \text{for $z \in \Omega_2$}
\end{cases}
\end{gather*}
solves
\begin{gather}
m(z) \text{ is analytic in $\mathbb{C} \backslash \Sigma$}, \nonumber\\
m_+(s) = m_-(s) \tilde{v}_\Sigma(s)+ g(s) \text{ for $s \in \Sigma$},\nonumber\\
m(z) = O_n\left(\frac{1}{z}\right) \text{ as $z \to \infty$} \label{mDecay}
\end{gather}
and satisfies $\Vert m_\pm\Vert_{L^2(\Sigma)} \le c \Vert t \Vert_{L^2([-1,1])} + c \Vert \tilde{g}\Vert_{L^2(\Sigma)}\le c\Vert g \Vert_{L^2(\Sigma)}$.

All that remains is to show $m_\pm(z) \in \partial C(L^2(\Sigma))$. To that end, let $\tilde{\mu}(s) = m_-(s)$ and define for $z \in \mathbb{C}\backslash \Sigma$,
\begin{gather}
H(z) := (C_{\Sigma}( \tilde{\mu} (\tilde{v}_\Sigma -I) + g))(z). \label{HDef}
\end{gather}
Then for $s \in \Sigma$
\begin{gather*}
H_+(s) =( (C_{\Sigma}^+ \tilde{\mu} (\tilde{v}_\Sigma -I) + g))(s)=( (C_{\Sigma}^- \tilde{\mu} (\tilde{v}_\Sigma -I) + g))(s) + \tilde{\mu}(s)(\tilde{v}_\Sigma(s) - I)) + g(s) \\
\hphantom{H_+(s)}{} = H_-(s) + m_+ - m_-.
\end{gather*}
Therefore
\begin{gather*}
H_+ - m_+ = H_- - m_-.
\end{gather*}
It follows from the same argument as in Proposition \ref{Fokas} that the function $H(z) - m(z)$ is entire. From (\ref{mDecay}) and (\ref{HDef}), $H(z) - m(z) = O_n\big(\frac{1}{z}\big)$ as $z\to\infty$. Thus, by Liouville's theorem,
\begin{gather*}
m(z) \equiv H(z) = (C_{\Sigma}( \tilde{\mu} (\tilde{v}_\Sigma -I) + g))(z)
\end{gather*}
and so
\begin{gather*}
m_\pm \in \partial C(L^2(\Sigma))
\end{gather*}
concluding the proof of Theorem \ref{LegOpBound}.
\end{proof}

Theorem \ref{LegOpBound} has, as a corollary, the equivalent result for the modified log-orthogonal problem:
\begin{thm} \label{LogOpBound} The operator $1-C_{v_\Sigma}$ is uniformly invertible for all sufficiently large $n$.
\end{thm}
\begin{proof}
Note that
\begin{gather*}
v_\Sigma - \tilde{v}_\Sigma = \begin{cases}
\left(\begin{matrix}
0 & 0\\
\left(\dfrac{F^2}{w} - 1\right) \phi^{-2n} & 0
\end{matrix}\right) & \text{for $s \in \Sigma_1 \cup \Sigma_2$},\\
\left(\begin{matrix}
0 & 0\\
\left(\dfrac{F^2}{w_+} + \dfrac{F^2}{w_-} - 2\right) \phi^{-2n} & 0
\end{matrix}\right) & \text{for $s \in [1, 1+\delta]$}, \\
0 & \text{for $s \in [-1, 1]$}.
\end{cases}
\end{gather*}
So it follows from Propositions \ref{PhiProps} and~\ref{FProps} that
\begin{gather*}
\Vert v_\Sigma - \tilde{v}_\Sigma\Vert_{L^\infty(\Sigma)} \to 0,%\label{LInftyDiff}
\end{gather*}
which, in particular, implies that
\begin{gather*}
\Vert C_{v_\Sigma} - C_{\tilde{v}_\Sigma} \Vert_{L^2 \to L^2}\to 0. %\label{L2Diff}
\end{gather*}
Combined with Theorem \ref{LegOpBound}, we see that
\begin{gather}
 \big\Vert (C_{v_\Sigma} - C_{\tilde{v}_\Sigma} ) (1 - C_{\tilde{v}_\Sigma} )^{-1} \big\Vert_{L^2 \to L^2}\to 0.\label{NotIntL2Bound}
\end{gather}
Writing
\begin{gather}
 \left(1 - C_{v_\Sigma}\right) \left(1 - C_{\tilde{v}_\Sigma}\right)^{-1} = 1 - \left(C_{v_\Sigma} - C_{\tilde{v}_\Sigma}\right) \left(1 - C_{\tilde{v}_\Sigma}\right)^{-1} \label{TrivRewrite}
\end{gather}
and using (\ref{NotIntL2Bound}), we see that for $n$ sufficiently large, the series
\begin{gather*}
\big( 1 - (C_{v_\Sigma} - C_{\tilde{v}_\Sigma}) (1 - C_{\tilde{v}_\Sigma})^{-1} \big) ^{-1}=
\sum_{k = 0}^\infty \big[ (C_{v_\Sigma} - C_{\tilde{v}_\Sigma} ) (1 - C_{\tilde{v}_\Sigma} )^{-1} \big]^k
\end{gather*}
converges in operator norm and
\begin{gather*}
 \big\Vert\big( 1 - (C_{v_\Sigma} - C_{\tilde{v}_\Sigma}) (1 - C_{\tilde{v}_\Sigma})^{-1} \big) ^{-1}\big\Vert_{L^2 \to L^2} \le C < \infty.
 \end{gather*}
It follows from (\ref{TrivRewrite}) that $(1-C_{v_\Sigma})^{-1}$ exists and is given by
\begin{gather*}
 (1 - C_{v_\Sigma})^{-1} = (1 - C_{\tilde{v}_\Sigma})^{-1}\big( 1 - (C_{v_\Sigma} - C_{\tilde{v}_\Sigma})(1 - C_{\tilde{v}_\Sigma})^{-1} \big) ^{-1}
\end{gather*}
for $n$ sufficiently large and hence
\begin{gather*}
 \Vert(1 - C_{v_\Sigma})^{-1} \Vert_{L^2 \to L^2} \le \big\Vert(1 - C_{\tilde{v}_\Sigma})^{-1}\big\Vert_{L^2 \to L^2} \big\Vert \big( 1 - (C_{v_\Sigma} - C_{\tilde{v}_\Sigma}) (1 - C_{\tilde{v}_\Sigma})^{-1} \big) ^{-1}\big\Vert_{L^2 \to L^2} \nonumber\\
\hphantom{\Vert(1 - C_{v_\Sigma})^{-1} \Vert_{L^2 \to L^2}}{} \le C \big\Vert(1 - C_{\tilde{v}_\Sigma})^{-1}\big\Vert_{L^2 \to L^2}.
 \end{gather*}
 So by Theorem \ref{LegOpBound}, $(1 - C_{v_\Sigma})^{-1}$ is bounded in operator norm uniformly for $n$ sufficiently large.
\end{proof}

\begin{Remark}
In fact, a standard RHP argument shows that $(1 - C_{v_\Sigma})^{-1}$ exists and is bounded in $L^2 \to L^2$ for all $n\ge 0$. The point of the above theorem is that the bound is independent of~$n$.
\end{Remark}

\subsection[A formula for $Q_1 - \tilde{Q}_1$]{A formula for $\boldsymbol{Q_1 - \tilde{Q}_1}$}

\begin{prop}\label{QQProp}
\begin{gather}
Q_1 - \tilde{Q}_1 = -\int_\Sigma \mu (v_\Sigma - \tilde{v}_\Sigma) \tilde{\mu}^{-1}\dbar z,\label{QQFormula}
\end{gather}
where
\begin{gather*}
\mu = (1 - C_{v_\Sigma} )^{-1} I, \qquad \tilde{\mu} = (1 - C_{\tilde{v}_\Sigma} )^{-1} I.
\end{gather*}
\end{prop}
\begin{proof}
The functions $\mu$ and $\tilde{\mu}$ are given in terms of the RHP solutions $Q$ and $\tilde{Q}$, see \cite{DeiftNLS}, by
\begin{gather*}
\mu = Q_-, \qquad \tilde{\mu} = \tilde{Q}_-,
\end{gather*}
and so
\begin{gather*}
Q_\pm = I + C_\Sigma^\pm (\mu (v_\Sigma - I)),
\end{gather*}
and it is straightforward to show that
\begin{gather*}
\tilde{Q}_\pm^{-1} = I +C_\Sigma^\pm \big(\big(v_\Sigma^{-1} - I\big) \tilde{\mu}^{-1}\big).
\end{gather*}
However $\tilde{\mu} = \tilde{Q}_- \in L^2(\Sigma)$ and $\det\big(\tilde{Q}_\pm\big) \equiv 1$, therefore $\tilde{\mu}^{-1} \in L^2(\Sigma)$ and clearly $v_\Sigma^{-1} \in L^\infty(\Sigma)$, and so $\tilde{Q}_\pm^{-1} \in I + \partial C(L^2(\Sigma))$. From Theorem \ref{ABInverse}, we have
\begin{gather}
Q(z)\tilde{Q}^{-1}(z) = I + C_\Sigma Q_- \big(v_\Sigma \tilde{v}_\Sigma^{-1} - I\big) \tilde{Q}_-^{-1},\label{RefComment1}
\end{gather}
but
\begin{gather}
v_\Sigma \tilde{v}_\Sigma^{-1}(s) - I = \begin{cases} \left(\begin{matrix}
0 & 0\\
\left(\frac{F^2}{w}(s) - 1\right)\phi^{-2n}(s) & 0
\end{matrix}\right) & \text{for $s\in \Sigma_1\cup \Sigma_2$} \\
\left(\begin{matrix}
0 & 0\\
\left(\frac{F^2}{w_+}(s)+\frac{F^2}{w_-}(s) - 2\right)\phi^{-2n}(s) & 0
\end{matrix}\right) & \text{for $s\in 1+\delta,1]$} \\
\mathbf{0} & \text{for $s\in [-1,1]$}
\end{cases} \nonumber\\
\hphantom{v_\Sigma \tilde{v}_\Sigma^{-1}(s) - I}{} = v_\Sigma - \tilde{v}_\Sigma.\label{RefComment2}
\end{gather}
Therefore
\begin{gather*}
C_\Sigma Q_- \big(v_\Sigma \tilde{v}_\Sigma^{-1} - I\big) \tilde{Q}_-^{-1} = C_\Sigma \mu_- (v_\Sigma - \tilde{v}_\Sigma) \tilde{\mu}_-^{-1}
= -\frac{1}{z}\int_\Sigma \mu_- (v_\Sigma - \tilde{v}_\Sigma) \tilde{\mu}^{-1}\dbar s + O\big(z^{-2}\big)
\end{gather*}
as $z \to \infty$. Therefore
\begin{gather*}
I + \frac{Q_1 - \tilde{Q}_1}{z} + O\big(z^{-2}\big) = Q\tilde{Q}^{-1}(z) = I -\frac{1}{z}\int_\Sigma \mu_- (v_\Sigma - \tilde{v}_\Sigma) \tilde{\mu}^{-1}\dbar s + O\big(z^{-2}\big).
\end{gather*}
So
\begin{gather*}
Q_1 - \tilde{Q}_1 = -\int_\Sigma \mu_- (v_\Sigma - \tilde{v}_\Sigma) \tilde{\mu}^{-1}\dbar s,
\end{gather*}
which completes the proof.
\end{proof}

\begin{Remark}The goal of this section is to develop the machinery that allows effective comparison of the difference of the RHPs $Q$ and~$\tilde{Q}$. As mentioned in the introduction, the standard approach of comparing the RHPs for $Q$, $\tilde{Q}$ by their quotient breaks down. Indeed, for $x \in [1, 1 + \delta]$,
\begin{gather*}
\big(Q \tilde{Q}^{-1}\big)_+ = Q_- \left(\begin{matrix}
1 & 0\\
\left(\dfrac{F^2}{w_+} + \dfrac{F^2}{w_-}\right) \phi^{-2n} & 1
\end{matrix}\right)\left(\begin{matrix}
1 & 0\\
-2\phi^{-2n} & 1
\end{matrix}\right)\tilde{Q}_-^{-1} \nonumber\\
\hphantom{\big(Q \tilde{Q}^{-1}\big)_+}{} = Q_- \tilde{Q}_-^{-1} \tilde{Q}_- \left(\begin{matrix}
1 & 0\\
\left(\dfrac{F^2}{w_+} + \dfrac{F^2}{w_-} -2\right)\phi^{-2n} & 1
\end{matrix}\right)\tilde{Q}_-^{-1}\nonumber\\
\hphantom{\big(Q \tilde{Q}^{-1}\big)_+}{}
=Q_- \tilde{Q}_- \left(I + \left(\dfrac{F^2}{w_+} + \dfrac{F^2}{w_-} - 2\right) \phi^{-2n} \left(\begin{matrix}
\tilde{Q}_{12} \tilde{Q}_{22} & -\tilde{Q}_{12}^2\\
\tilde{Q}_{22}^2 & -\tilde{Q}_{12} \tilde{Q}_{22}
\end{matrix}\right)_-\right).
\end{gather*}
So $Q\tilde{Q}^{-1}$ is the solution to a standard RHP on $\Sigma$ with jump
\begin{gather}
J :=I + \left(\frac{F^2}{w_+} + \frac{F^2}{w_-} - 2\right) \phi^{-2n} \left(\begin{matrix}
\tilde{Q}_{12} \tilde{Q}_{22} & -\tilde{Q}_{12}^2\\
\tilde{Q}_{22}^2 & -\tilde{Q}_{12} \tilde{Q}_{22}.
\end{matrix}\right)_-.\label{JDef}
\end{gather}
Typically, we would like to draw conclusions about $Q\tilde{Q}^{-1}$ by demonstrating that $\Vert J -I\Vert_{L^\infty} \to 0$. However, while the jump functions associated with $Q$, $\tilde{Q}$ are close in an $L^\infty$ sense, and specifically, near the point~$1$,
\begin{gather*}
 \left(\frac{F^2}{w_+}(z) + \frac{F^2}{w_-}(z) - 2\right) \phi^{-2n}(z)= \frac{c}{\log^2 (\vert z -1\vert)} + O\left(\frac{1}{\log^3(\vert z-1\vert)}\right), \qquad \text{as $z \to 1$},
\end{gather*}
this behavior is counteracted by the singular behavior of $\tilde{Q}$. From (\ref{TildeQRHP5}), we have that $\tilde{Q}_{12}, \tilde{Q}_{22} = O_n(\log (\vert z -1\vert))$ as $z\to1$. More is true, in fact, $\tilde{Q}_{12}, \tilde{Q}_{22} = c_n \log(\vert z -1 \vert) (1 + o(1))$ as $z\to1$ where $c_n \sim cn^{1/2}$, which we will not prove, but which can be inferred from the behavior of $\tilde{\mu}$ described in Section \ref{LegendreSummary}. Therefore, from (\ref{JDef}), it is clear that
\begin{gather*}
\left\Vert J \right\Vert_{L^\infty([1, 1+\delta])} \not \to 0 \qquad \text{as $n \to \infty$}.
\end{gather*}
Thus it is not clear how to proceed in an analysis of the quotient $Q\tilde{Q}^{-1}$. Ultimately, this difficulty motivated our examination of the difference $Q -\tilde{Q}$.
\end{Remark}
\begin{Remark}We note that formulae (\ref{RefComment1}) and (\ref{RefComment2}) can be used, in principle, not only to derive our basic formula (\ref{QQFormula}), but also to obtain asymptotic information about the orthogonal polynomials for general values of $z$, not just for $z=\infty$. Certainly for general $z$ the analysis should go through as long as $z$ is at a fixed distance away from the interval $[-1,1]$, but as $z$ approaches the end points $\pm1$, in particular, technical difficulties arise that we have not yet been able to address. Also, as $z \to \infty$, one can, in principle, use~(\ref{RefComment1}) and~(\ref{RefComment2}) to obtain lower order terms for the recurrence coefficients, but again certain technical difficulties arise that we have not yet overcome.
\end{Remark}

\section{Conclusion of the proof of Theorem \ref{MainResult}}\label{ProofSection}
In this section we combine the previous results of the paper in order to prove the main result, Theorem \ref{MainResult}.
\subsection{Calculation of the integral in (\ref{QQFormula})}

The integral in (\ref{QQFormula}) involves both $\tilde{\mu}$ and $\mu$. The asymptotics of $\tilde{\mu}$ are known, see \cite{Kuijlaars}, or Appendix \ref{LegendreSummary} for a summary. Theorems~\ref{LegOpBound} and~\ref{LogOpBound} together with the known asymptotics of $\tilde{\mu}$ give us control over $\mu$ in an $L^2(\Sigma)$ sense, which will be sufficient for our purposes.

Namely, in Proposition \ref{NormsApp} we prove:
\begin{prop}\label{Norms}The functions $\mu$, $\tilde{\mu}$ obey the following estimates:
\begin{gather}
\big\vert \big\vert \tilde{\mu}^{(n)}\big( v_{\Sigma}^{(n)} - \tilde{v}_\Sigma^{(n)} \big) \big\vert \big\vert_{L^2(\Sigma)} = O\left(\frac{1}{n^{1/2}\log^2 n }\right),\label{Mu-MuL2Pre}
\\
\big\vert \big\vert \mu^{(n)} - \tilde{\mu}^{(n)} \big\vert \big\vert_{L^2(\Sigma)} = O\left(\frac{1}{n^{1/2}\log^2 n }\right),\label{Mu-MuL2}
\\
\big\vert \big\vert \tilde{\mu}^{(n)}\big( v_{\Sigma}^{(n)} - \tilde{v}_\Sigma^{(n)} \big) - \tilde{\mu}^{(n+1)}\big( v_{\Sigma}^{(n+1)} - \tilde{v}_\Sigma^{(n+1)} \big) \big\vert \big\vert_{L^2(\Sigma)} = O\left(\frac{1}{n^{3/2}\log^2 n }\right),\label{Mu-MuL2HardPre}
\\
\big\vert \big\vert \big(\mu^{(n)} - \tilde{\mu}^{(n)}\big) - \big(\mu^{(n+1)} - \tilde{\mu}^{(n+1)}\big) \big\vert \big\vert_{L^2(\Sigma)} = O\left(\frac{1}{n^{3/2}\log^2 n }\right).\label{Mu-MuL2Hard}
\end{gather}
\end{prop}
\begin{proof}See Proposition \ref{NormsApp}.
\end{proof}

Proposition \ref{Norms} together with Theorems \ref{LegOpBound} and \ref{LogOpBound} allow us to estimate the integral in (\ref{QQFormula}) by a similar integral involving only Legendre quantities. This is done through the following proposition:
\begin{prop}\label{Replacement}
The integrals can be estimated by
\begin{gather*}
-\int_\Sigma \mu (v_\Sigma - \tilde{v}_\Sigma) \tilde{\mu}^{-1}\dbar s = -\int_\Sigma \tilde{\mu} (v_\Sigma - \tilde{v}_\Sigma) \tilde{\mu}^{-1}\dbar s + O\left(\frac{1}{n\log^4n}\right)
\end{gather*}
and
\begin{gather*}
-\int_\Sigma \mu^{(n)} \big(v^{(n)}_\Sigma - \tilde{v}^{(n)}_\Sigma\big) \big(\tilde{\mu}^{(n)}\big)^{-1}\dbar s + \int_\Sigma \mu^{(n+1)} \big(v^{(n+1)}_\Sigma - \tilde{v}^{(n+1)}_\Sigma\big) \big(\tilde{\mu}^{(n+1)}\big)^{-1}\dbar s \nonumber\\
\qquad{} = -\int_\Sigma \tilde{\mu}^{(n)} \big(v^{(n)}_\Sigma - \tilde{v}^{(n)}_\Sigma\big) \big(\tilde{\mu}^{(n)}\big)^{-1}\dbar s + \int_\Sigma \tilde{\mu}^{(n+1)} \big(v^{(n+1)}_\Sigma - \tilde{v}^{(n+1)}_\Sigma\big) \big(\tilde{\mu}^{(n+1)}\big)^{-1}\dbar s \nonumber\\
 \qquad\quad{} + O\left(\frac{1}{n^2\log^4n}\right).
\end{gather*}
\end{prop}
\begin{proof}
Recall that $\det(\mu), \det(\tilde{\mu}) = \det(Q_-), \det(\tilde{Q}_-) = 1$. Therefore the norms (\ref{Mu-MuL2}) and (\ref{Mu-MuL2Hard}) also imply
\begin{gather*}
\big\vert \big\vert (v_\Sigma - \tilde{v}_\Sigma) \tilde{\mu}^{-1} \big\vert \big\vert_{L^2(\Sigma)} = O\left(\frac{1}{n^{1/2}\log^2 n }\right),
\\
\big\vert \big\vert \big(v^{(n)}_\Sigma - \tilde{v}^{(n)}_\Sigma\big) \big(\tilde{\mu}^{(n)}\big)^{-1} - \big(v^{(n+1)}_\Sigma - \tilde{v}^{(n+1)}_\Sigma\big) \big(\tilde{\mu}^{(n+1)}\big)^{-1} \big\vert \big\vert_{L^2(\Sigma)} = O\left(\frac{1}{n^{3/2}\log^2 n }\right),
\end{gather*}
and therefore,
\begin{gather*}
-\int_\Sigma \mu (v_\Sigma - \tilde{v}_\Sigma) \tilde{\mu}^{-1}\dbar s = -\int_\Sigma \tilde{\mu} (v_\Sigma - \tilde{v}_\Sigma) \tilde{\mu}^{-1}\dbar s -\int_\Sigma\left(\mu - \tilde{\mu}\right) (v_\Sigma - \tilde{v}_\Sigma)\tilde{\mu}^{-1}\dbar s,
\end{gather*}
and by (\ref{Mu-MuL2Pre}) and (\ref{Mu-MuL2}),
\begin{gather*}
\left\vert\int_\Sigma \tilde{\mu} (v_\Sigma - \tilde{v}_\Sigma) \big(\mu^{-1} - \tilde{\mu}^{-1}\big)\dbar s\right\vert \le \left\Vert \tilde{\mu} (v_\Sigma - \tilde{v}_\Sigma) \right\Vert_{L^2(\Sigma)} \left\Vert \mu^{-1} - \tilde{\mu}^{-1}\right\Vert_{L^2(\Sigma)} = O\left(\frac{1}{n\log^4n}\right).
\end{gather*}
So
\begin{gather*}
-\int_\Sigma \tilde{\mu} (v_\Sigma - \tilde{v}_\Sigma) \mu^{-1}\dbar s = -\int_\Sigma \tilde{\mu} (v_\Sigma - \tilde{v}_\Sigma) \tilde{\mu}^{-1}\dbar s + O\left(\frac{1}{n\log^4n}\right).
\end{gather*}
Somewhat similarly,
\begin{gather*}
-\int_\Sigma \mu^{(n)} \big(v^{(n)}_\Sigma - \tilde{v}^{(n)}_\Sigma\big) \big(\tilde{\mu}^{(n)}\big)^{-1}\dbar s + \int_\Sigma \mu^{(n+1)} \big(v^{(n+1)}_\Sigma - \tilde{v}^{(n+1)}_\Sigma\big) \big(\tilde{\mu}^{(n+1)}\big)^{-1}\dbar s\nonumber\\
\qquad{} =-\int_\Sigma \tilde{\mu}^{(n)} \big(v^{(n)}_\Sigma - \tilde{v}^{(n)}_\Sigma\big) \big(\tilde{\mu}^{(n)}\big)^{-1}\dbar s + \int_\Sigma \tilde{\mu}^{(n+1)} \big(v^{(n+1)}_\Sigma - \tilde{v}^{(n+1)}_\Sigma\big) \big(\tilde{\mu}^{(n+1)}\big)^{-1}\dbar s \nonumber\\
\qquad\quad{} - \int_\Sigma \big(\mu^{(n)} - \tilde{\mu}^{(n)}\big) \big( \big(v^{(n)}_\Sigma - \tilde{v}^{(n)}_\Sigma\big)\big(\tilde{\mu}^{(n)}\big)^{-1} - \big(v_\Sigma^{(n+1)} - \tilde{v}_\Sigma^{(n+1)} \big) \big(\tilde{\mu}^{(n+1)}\big)^{-1}\big)\dbar s\nonumber\\
\qquad \quad{}- \int_\Sigma \big(\mu^{(n)} - \tilde{\mu}^{(n)} - \big(\mu^{(n+1)} - \tilde{\mu}^{(n+1)}\big)\big) \big(v^{(n+1)}_\Sigma - \tilde{v}^{(n+1)}_\Sigma\big) \big(\tilde{\mu}^{(n+1)}\big)^{-1}\dbar s.
\end{gather*}
Now, by (\ref{Mu-MuL2}) and (\ref{Mu-MuL2HardPre}),
\begin{gather*}
\left\vert \int_\Sigma \big(\mu^{(n)} - \tilde{\mu}^{(n)}\big) \big( \big(v^{(n)}_\Sigma - \tilde{v}^{(n)}_\Sigma\big)\big(\tilde{\mu}^{(n)}\big)^{-1} - \big(v_\Sigma^{(n+1)} - \tilde{v}_\Sigma^{(n+1)} \big) \big(\tilde{\mu}^{(n+1)}\big)^{-1}\big)\dbar s\right\vert \nonumber\\
 \qquad{} \le \big\vert \big\vert \mu^{(n)} - \tilde{\mu}^{(n)} \big\vert \big\vert_{L^2(\Sigma)}\big\vert \big\vert \big(v^{(n)}_\Sigma - \tilde{v}^{(n)}_\Sigma\big) \big(\tilde{\mu}^{(n)}\big)^{-1} - \big(v^{(n+1)}_\Sigma - \tilde{v}^{(n+1)}_\Sigma\big) \big(\tilde{\mu}^{(n+1)}\big)^{-1} \big\vert \big\vert_{L^2(\Sigma)} \nonumber\\
\qquad{} = O\left(\frac{1}{n^2\log^4n}\right),
\end{gather*}
and by (\ref{Mu-MuL2Pre}) and (\ref{Mu-MuL2Hard}),
\begin{gather*}
\left\vert\int_\Sigma \big(\mu^{(n)} - \tilde{\mu}^{(n)} - \big(\mu^{(n+1)} - \tilde{\mu}^{(n+1)}\big)\big) (v^{(n+1)}_\Sigma - \tilde{v}^{(n+1)}_\Sigma) \big(\tilde{\mu}^{(n+1)}\big)^{-1}\dbar s\right\vert \nonumber\\
 \qquad {}\le \big\vert \big\vert \big(\mu^{(n)} - \tilde{\mu}^{(n)}\big) - \big(\mu^{(n+1)} - \tilde{\mu}^{(n+1)}\big) \big\vert \big\vert_{L^2(\Sigma)}\big\vert \big\vert \tilde{\mu}^{(n+1)}\big( v_{\Sigma}^{(n+1)} - \tilde{v}_\Sigma^{(n+1)} \big) \big\vert \big\vert_{L^2(\Sigma)} \nonumber\\
 \qquad{} = O\left(\frac{1}{n^2\log^4n}\right).
 \end{gather*}
Therefore, we see that
\begin{gather*}
-\int_\Sigma \mu^{(n)} \big(v^{(n)}_\Sigma - \tilde{v}^{(n)}_\Sigma\big) \big(\tilde{\mu}^{(n)}\big)^{-1}\dbar s + \int_\Sigma \mu^{(n+1)} \big(v^{(n+1)}_\Sigma - \tilde{v}^{(n+1)}_\Sigma\big) \big(\tilde{\mu}^{(n+1)}\big)^{-1}\dbar s \nonumber\\
\qquad{} = -\int_\Sigma \tilde{\mu}^{(n)} \big(v^{(n)}_\Sigma - \tilde{v}^{(n)}_\Sigma\big) \big(\tilde{\mu}^{(n)}\big)^{-1}\dbar s + \int_\Sigma \tilde{\mu}^{(n+1)} \big(v^{(n+1)}_\Sigma - \tilde{v}^{(n+1)}_\Sigma\big) \big(\tilde{\mu}^{(n+1)}\big)^{-1}\dbar s \nonumber\\
 \qquad\quad{} + O\left(\frac{1}{n^2\log^4n}\right),
\end{gather*}
which completes the proof of Proposition \ref{Replacement}.
\end{proof}

The integrals in Proposition \ref{Replacement} can be computed asymptotically using the known asymptotics of $\tilde{\mu}$, yielding:
\begin{prop}\label{EIntegralProp}
\begin{gather*}
\int_\Sigma \tilde{\mu}(z) (v_\Sigma(z) - \tilde{v}_\Sigma(z)) \tilde{\mu}^{-1}(z) \dbar z = \frac{3}{16 n\log^2n} \left(\begin{matrix}
1 & -i\\
-i& -1
\end{matrix}\right) + O\left(\frac{1}{n\log^3n}\right). %\label{EIntegral}
\end{gather*}
\end{prop}
\begin{proof}
See Proposition \ref{EIntegralPropApp}.
\end{proof}

\begin{prop}\label{HIntegralProp}
\begin{gather*}
\int_\Sigma \tilde{\mu}^{(n)}(z) \big(v_\Sigma^{(n)}(z) - \tilde{v}^{(n)}_\Sigma(z)\big) \big(\tilde{\mu}^{(n)}\big)^{-1}(z) \dbar z \nonumber\\
\qquad\quad{} - \int_\Sigma \tilde{\mu}^{(n+1)}(z) \big(v_\Sigma^{(n+1)}(z) - \tilde{v}^{(n+1)}_\Sigma(z)\big) \big(\tilde{\mu}^{(n+1)}\big)^{-1}(z) \dbar z \nonumber\\
\qquad{} = \frac{3}{16 n^2\log^2n} \left(\begin{matrix}
1 & -i\\
-i& -1
\end{matrix}\right) + O\left(\frac{1}{n^2\log^3n}\right).\label{HIntegral}
\end{gather*}
\end{prop}
\begin{proof}
See Proposition \ref{HIntegralPropApp}.
\end{proof}

Combining the results of this section, we can compute the difference $Q_1^{(n)}- \tilde{Q}_1^{(n)}$ asymptotically,

\begin{prop}\label{Q-TildeQProp}The difference $Q_1- \tilde{Q}_1 := Q_1^{(n)}- \tilde{Q}_1^{(n)}$ has the following asymptotic expansions
\begin{gather}
Q_1 - \tilde{Q}_1 = \frac{3}{16 n \log^2n} \left(\begin{matrix}
-1 & i \\
i & 1
\end{matrix}\right) + O\left(\frac{1}{n\log^3n}\right),\label{Q-TildeQ}\\
Q_1^{(n)} - \tilde{Q}_1^{(n)} - \big(Q_1^{(n+1)} - \tilde{Q}_1^{(n+1)}\big) = \frac{3}{16 n^2 \log^2n} \left(\begin{matrix}
-1 & i \\
i & 1
\end{matrix}\right) + O\left(\frac{1}{n^2\log^3n}\right).\label{Q-TildeQHard}
\end{gather}
\end{prop}
\begin{proof}
Combining Propositions \ref{QQProp}, \ref{Replacement}, and~\ref{EIntegralProp},
\begin{gather*}
Q_1 - \tilde{Q}_1 = -\int_\Sigma \tilde{\mu}(v_\Sigma - \tilde{v}_\Sigma) \mu^{-1} \dbar z
= -\int_\Sigma \tilde{\mu}(v_\Sigma - \tilde{v}_\Sigma) \tilde{\mu}^{-1} \dbar z + O\left(\frac{1}{n\log^3n}\right)\nonumber\\
\hphantom{Q_1 - \tilde{Q}_1}{}= \frac{3}{16 n \log^2n} \left(\begin{matrix}
-1 & i \\
i & 1
\end{matrix}\right) + O\left(\frac{1}{n\log^3n}\right)
\end{gather*}
demonstrating (\ref{Q-TildeQ}). Similarly, combining Propositions \ref{QQProp}, \ref{Replacement}, and~\ref{HIntegralProp},
\begin{gather*}
Q_1^{(n)} - \tilde{Q}_1^{(n)}- \big(Q_1^{(n+1)} - \tilde{Q}_1^{(n+1)}\big) \nonumber\\
\qquad{}= -\int_\Sigma \tilde{\mu}^{(n)}\big(v_\Sigma^{(n)} - \tilde{v}_\Sigma^{(n)}\big) \big(\mu^{(n)}\big)^{-1} \dbar z +\int_\Sigma \tilde{\mu}^{(n+1)}\big(v_\Sigma^{(n+1)} - \tilde{v}_\Sigma^{(n+1)}\big) \big(\mu^{(n+1)}\big)^{-1} \dbar z \nonumber\\
\qquad{} = -\int_\Sigma \tilde{\mu}^{(n)}\big(v_\Sigma^{(n)} - \tilde{v}_\Sigma^{(n)}\big) \big(\tilde{\mu}^{(n)}\big)^{-1} \dbar z \\
\qquad\quad{} +\int_\Sigma \tilde{\mu}^{(n+1)}\big(v_\Sigma^{(n+1)} - \tilde{v}_\Sigma^{(n+1)}\big) \big(\tilde{\mu}^{(n+1)}\big)^{-1} \dbar z + O\left(\frac{1}{n^2\log^3n}\right) \nonumber\\
\qquad{} = \frac{3}{16 n^2 \log^2n} \left(\begin{matrix}
-1 & i \\
i & 1
\end{matrix}\right) + O\left(\frac{1}{n^2\log^3n}\right)
\end{gather*}
demonstrating (\ref{Q-TildeQHard}).
\end{proof}

\subsection[Computing the asymptotic behavior of $a_n$, $b_n$]{Computing the asymptotic behavior of $\boldsymbol{a_n}$, $\boldsymbol{b_n}$}
\begin{proof}[Proof of Theorem \ref{MainResult}] Recall that by Proposition \ref{RecCoeffFormula},
\begin{gather*}
a_n - \tilde{a}_n = \big(Q_1^{(n)}\big)_{11} - \big(\tilde{Q}_1^{(n)}\big)_{11}- \big(\big(Q_1^{(n+1)}\big)_{11} - \big(\tilde{Q}_1^{(n+1)}\big)_{11}\big),\\
b_{n-1}^2 - \tilde{b}_{n-1}^2 = \big(\big(Q_1^{(n)}\big)_{12} - \big(\tilde{Q}_1^{(n)}\big)_{12}\big) \big(\big(Q_1^{(n)}\big)_{21} - \big(Q_1^{(n+1)}\big)_{21}\big)\nonumber \\
\hphantom{b_{n-1}^2 - \tilde{b}_{n-1}^2 =}{} + \big(\tilde{Q}_1^{(n)}\big)_{12} \big[ \big(Q_1^{(n)}\big)_{21} - \big(\tilde{Q}_1^{(n)}\big)_{21} - \big(\big(Q_1^{(n+1)}\big)_{21} - \big(\tilde{Q}_1^{(n+1)}\big)_{21}\big)\big].
\end{gather*}
The calculation for $a_{n} - \tilde{a}_n$ is read directly from Proposition \ref{Q-TildeQProp}
\begin{gather*}
a_n - \tilde{a}_n= \left(\frac{3}{16 n^2 \log^2n} \left(\begin{matrix}
-1 & i \\
i & 1
\end{matrix}\right) + O\left(\frac{1}{n^2\log^3n}\right)\right)_{11} \\
\hphantom{a_n - \tilde{a}_n}{} = - \frac{3}{16 n^2 \log^2 n} + O\left(\frac{1}{n^2 \log^3 n} \right).
\end{gather*}
To calculate $b_{n-1}^2 - \tilde{b}_{n-1}^2$, note from (\ref{StoTildeQ}) and (\ref{SAsymptotics1}) that as $z \to \infty$,
\begin{gather}
\tilde{Q}^{(n)} = S^{(n)}(z) = R^{(n)}(z) N(z) \nonumber\\
\hphantom{\tilde{Q}^{(n)}}{} = \left( 1 + \frac{R_1(z)}{n} + O\left(\frac{1}{n^2}\right)\right) \left(\begin{matrix}
\dfrac{a + a^{-1}}{2} & \dfrac{a - a^{-1}}{2i} \vspace{1mm}\\
\dfrac{a - a^{-1}}{-2i} & \dfrac{a + a^{-1}}{2}
\end{matrix}\right).\label{TildeQStuff}
\end{gather}
From (\ref{ThingyDelta}),
\begin{gather*}
R^{(n)}(z) = I + \frac{R_1(z)}{n} + \Delta(z),
\end{gather*}
where $R_1$, $\Delta(z)$ are analytic in for $z \in \mathbb{C}\backslash \Sigma_S$ and
\begin{gather*}
\vert R_1(z) \vert \le \frac{c}{\vert z \vert }\text{ and } \vert \Delta(z) \vert \le \frac{c}{\vert z \vert n^2} \qquad \text{as $z \to \infty$}.
\end{gather*}
We can write $N(z) = I + \frac{N_1}{z} + O\big(z^{-2}\big)$ and $N(z) = \left(\begin{smallmatrix}
N_{11}(z) & N_{12}(z)\\
N_{21}(z) & N_{22}(z)
\end{smallmatrix}\right)$. So, in particular, $\tilde{Q}_1 = N_1 + O\left(\frac{1}{n}\right)$ and
\begin{gather*}
N_{12}(z) = \frac{a(z) - a^{-1}(z)}{2i} = \frac{1}{2i}\left(\left(\frac{z-1}{z+1}\right)^{1/4} - \left(\frac{z+1}{z-1}\right)^{1/4}\right) \nonumber\\
\hphantom{N_{12}(z)}{} = \frac{1}{2i} \frac{(z-1)^{1/2} - (z+1)^{1/2}}{\big(z^2-1\big)^{1/4}}= \frac{1}{2i} \frac{\big(1-\frac{1}{z}\big)^{1/2} - \big(1 + \frac{1}{z}\big)^{1/2}}{\big(1-\frac{1}{z^2}\big)^{1/4}} \nonumber\\
\hphantom{N_{12}(z)}{}= \frac{1}{2i} \left( 1 - \frac{1}{2z} - 1 - \frac{1}{2z} + O\big(z^{-2}\big) \right)\big(1 + O\big(z^{-2}\big)\big)
 = \left(- \frac{1}{2i z} + O\big(z^{-2}\big)\right)
\end{gather*}
from which it follows that
\begin{gather}
\big(\tilde{Q}_1\big)_{12} = -\frac{1}{2i} + O\left(\frac{1}{n}\right).\label{TildeQ1}
\end{gather}
Combining Proposition \ref{Q-TildeQProp} with (\ref{TildeQStuff}) we see that
\begin{gather}
Q_1^{(n)} - Q_1^{(n+1)} = \tilde{Q}_1^{(n)} - \tilde{Q}_1^{(n+1)} + O\left(\frac{1}{n^2\log^2 n}\right) \nonumber\\
\hphantom{Q_1^{(n)} - Q_1^{(n+1)}}{} = \lim_{z \to \infty}z \big(R^{(n)}(z) - R^{(n+1)}(z) \big) N(z)+ O\left(\frac{1}{n^2\log^2 n}\right) = O\left(\frac{1}{n^2}\right).\label{TildeQ2}
\end{gather}
Using Propositions \ref{RecCoeffFormula} and~\ref{Q-TildeQProp}, and substituting in (\ref{TildeQ1}) and (\ref{TildeQ2}), we see that
\begin{gather}
b_{n-1}^2 - \tilde{b}_{n-1}^2 = \big(\big(Q_1^{(n)}\big)_{12} - \big(\tilde{Q}_1^{(n)}\big)_{12}\big) \big(\big(Q_1^{(n)}\big)_{21} - \big(Q_1^{(n+1)}\big)_{21}\big)\nonumber\\
\hphantom{b_{n-1}^2 - \tilde{b}_{n-1}^2 =}{} + \big(\tilde{Q}_1^{(n)}\big)_{12} \big[ \big(Q_1^{(n)}\big)_{21} - \big(\tilde{Q}_1^{(n)}\big)_{21} - \big(\big(Q_1^{(n+1)}\big)_{21} - \big(\tilde{Q}_1^{(n+1)}\big)_{21}\big)\big] \nonumber\\
\hphantom{b_{n-1}^2 - \tilde{b}_{n-1}^2 }{}= O\left(\frac{1}{n\log^2 n }\right) \cdot O\left(\frac{1}{n^2}\right) + \left(-\frac{1}{2i} + O\left(\frac{1}{n}\right)\right) \nonumber\\
\hphantom{b_{n-1}^2 - \tilde{b}_{n-1}^2 =}{} \times \left(\frac{3}{16 n^2 \log^2n} \left(\begin{matrix}
-1 & i \\
i & 1
\end{matrix}\right) + O\left(\frac{1}{n^2\log^3n}\right)\right)_{21} \nonumber\\
\hphantom{b_{n-1}^2 - \tilde{b}_{n-1}^2 }{}= \frac{3}{32 n^2 \log^2n} + O\left(\frac{1}{n^2\log^3n}\right).\label{NearlyThere}
\end{gather}

Next, note that
\begin{gather*}
b_{n-1}^2 - \tilde{b}_{n-1}^2 = \big(b_{n-1} - \tilde{b}_{n-1}\big) \big(b_{n-1} + \tilde{b}_{n-1}\big) = \big(b_{n-1} - \tilde{b}_{n-1}\big) \big(2\tilde{b}_{n-1} + b_{n-1} - \tilde{b}_{n-1}\big).
\end{gather*}
Recall that $\tilde{b}_{n-1} = \frac{1}{2} +O\big(\frac{1}{n^2}\big)$. Therefore
\begin{gather}
b_{n-1}^2 - \tilde{b}_{n-1}^2 = \big(b_{n-1} - \tilde{b}_{n-1}\big) \left(1 + \big(b_{n-1} - \tilde{b}_{n-1}\big) + O\left(\frac{1}{n^2}\right)\right).
\label{Silliness}
\end{gather}
By definition, $b_{n-1} > 0$, so $ \big(1 + \big(b_{n-1} - \tilde{b}_{n-1}\big) + O\big(\frac{1}{n^2}\big)\big) \ge \frac{1}{2} -\epsilon$, which combined with~(\ref{NearlyThere}) and~(\ref{Silliness}), first demonstrates
\begin{gather*}
b_{n-1} - \tilde{b}_{n-1} = O\left(\frac{1}{n^2 \log^2 n}\right),
\end{gather*}
which, substituted into (\ref{Silliness}) and using (\ref{NearlyThere}) a second time, gives
\begin{gather*}
b_{n-1} - \tilde{b}_{n-1} = \frac{3}{32 n^2 \log^2n} + O\left(\frac{1}{n^2\log^3n}\right)
\end{gather*}
concluding the proof of Theorem \ref{MainResult}.
\end{proof}

\appendix
\section{Asymptotics of the Szeg\H{o} function}\label{FProofsSection}
\subsection[Asymptotics of $F$]{Asymptotics of $\boldsymbol{F}$}\label{AppendixA}
The following proves assertion~\eqref{prop2-2.3} of Proposition~\ref{FProps}
\begin{prop}\label{FAsymptoticsZApp}For $z \in \mathbb{C}_\pm$,
\begin{gather*}
\frac{F^2}{w}(z) = \left(1 \mp \frac{i\pi}{w(z)} - \frac{\pi^2}{2w^2(z)} + O\left(\frac{1}{w^3(z)}\right)\right)
\end{gather*}
as $z \to 1$.
\end{prop}
\begin{proof}
Let $z = 1 + re^{i\theta}$ where $0 < \vert \theta \vert < \pi$, $r \to 0$. We will first prove that
\begin{gather*}
\int_{-1}^1 \frac{\log w(s)}{\big(s^2-1\big)^{1/2}_+} \frac{\dbar s}{s -z} = \frac{1}{\sqrt{2}}r^{-1/2}e^{-i\theta/2}\left(\frac{1}{2}\log \log \frac{k}{re^{i\theta}} - \frac{\pi^2}{4} \frac{1}{\log^2\frac{k}{re^{i\theta}}} + O\left(\frac{1}{\log^3 r}\right)\right),
\end{gather*}
where $O\big(\frac{1}{\log^3r}\big)$ is uniform in $0 <\vert \theta \vert < \pi$. Note that
\begin{gather}
\int_{-1}^1 \frac{\log w(s)}{\big(s^2-1\big)^{1/2}_+} \frac{\dbar s}{s- z} = -i \int_{-1}^1 \frac{\log w(s)}{\sqrt{1-s^2}} \frac{\dbar s}{s-z} \label{Referee2}\\
\qquad{} = -i \left( \int_{-1}^1 \frac{\log w(s)}{\sqrt{2}\sqrt{1-s}} \frac{\dbar s}{s- z} + \int_{-1}^1 \log w(s)\left[\frac{1}{\sqrt{1-s^2}} - \frac{1}{\sqrt{2}\sqrt{1-s}}\right] \frac{\dbar s}{s-z}\right).\nonumber
\end{gather}
A calculation shows that
\begin{gather}
\int_{-1}^1 \log w(s)\left[\frac{1}{\sqrt{1-s^2}} - \frac{1}{\sqrt{2}\sqrt{1-s}}\right] \frac{\dbar s}{s- z} \label{LemmaControlledInt}\\
 = \int_{-1}^0 \log w(s)\left[ \frac{\sqrt{1-s}}{2\sqrt{1+ s} + \sqrt{2}(1+s) }\right] \frac{\dbar s}{s-z} + \int_{0}^1 \log w(s)\left[ \frac{\sqrt{1-s}}{2\sqrt{1+ s} + \sqrt{2}(1+s) }\right] \frac{\dbar s}{s-z}. \nonumber
\end{gather}
The first integral is clearly bounded as $z \to 1$. The second integral is easily seen to be controlled via the following lemma:
\begin{Lemma}\label{WAnalytic}
Let $f(x)$ be an absolutely continuous function on the interval $[a,1]$ with $f(1) = 0$ and with a derivative $f'(x) \in L^p([a,1])$ for some $p > 1$. Then
\begin{gather*}
\int_{a}^1 f(s) \frac{\dbar s}{s-z} = O(1), \qquad \text{as $z \to +1$}
\end{gather*}
for $z \in \mathbb{C}\backslash[a,1]$.
\end{Lemma}
\begin{proof}
\begin{gather*}
\int_{a}^1 \frac{f(s)\dbar s}{s-z} = \int_{a}^1 \frac{f(s) - f(1)\dbar s}{s-z} = -\int_{a}^1\int_s^1 \frac{ f'(t){\rm d}t \dbar s}{s-z} \nonumber\\
\hphantom{\int_{a}^1 \frac{f(s)\dbar s}{s-z}}{} =\int_{a}^1\int_{a}^t \frac{ \dbar s}{z-s}f'(t){\rm d}t = -\frac{1}{2\pi i}\int_{a}^1 f'(t) \left(\log (z-t) - \log(z-a)\right){\rm d}t.
\end{gather*}
Therefore
\begin{gather*}
\left\vert \int_{a}^1 \frac{f(s)\dbar s}{s-z}\right\vert\le \Vert f' \Vert_{L^p([a,1])} \times \left\Vert \log(z-t) - \log(z-a)\right\Vert_{L^q([a,1])} = O(1),
\end{gather*}
where $\frac{1}{q} = 1- \frac{1}{p}$. Here we have used the fact that $\log(z-t) - \log(z -a)$ is in $L^q([a,1])$ uniformly as $z \to +1$, $z \in \mathbb{C}\backslash[a,1]$, for all $q < \infty$.
\end{proof}

Applying Lemma \ref{WAnalytic} to (\ref{LemmaControlledInt}) and combining with (\ref{Referee2}),
\begin{gather*}
\int_{-1}^1 \frac{\log w(s)}{\big(s^2-1\big)^{1/2}_+} \frac{\dbar s}{s-z}= -\frac{i}{\sqrt{2}} \int_{-1}^1 \frac{\log w(s)}{\sqrt{1-s}} \frac{\dbar s}{s-z}+ O(1) \label{CauchyFirstStep}
\end{gather*}
as $z \to 1$ uniformly for $z \in \mathbb{C}_\pm$. Recall $z = 1 + re^{i\theta}$ and make the change of variables $t = 1-s$, then
\begin{gather*}
 -\frac{i}{\sqrt{2}} \int_{-1}^1 \frac{\log w(s)}{\sqrt{1-s}} \frac{\dbar s}{s-z}= \frac{i}{\sqrt{2}} \int_{0}^2 \frac{\log w(1-t)}{\sqrt{t}} \frac{\dbar t}{t + re^{i\theta}}.
 \end{gather*}
From Lemma \ref{WAnalytic}, $w(1-t)$ is analytic for $s \in \mathbb{C}\backslash (-\infty, 0]$ and $\log w(1-t)$ is analytic in the region $\mathbb{C}\backslash((-\infty, 0] \cup [2k, \infty))$, so
 \begin{gather*}
\frac{i}{\sqrt{2}} \int_{0}^2 \frac{\log w(1-t)}{\sqrt{t}} \frac{\dbar t}{t + re^{i\theta}} = \frac{i}{\sqrt{2}} \left(\int_{C_1} + \int_{C_2}\right)\frac{\log w(1-t)}{t^{1/2}} \frac{\dbar t}{t + re^{i\theta}},
 \end{gather*}
where the contours $C_1$, $C_2$ are depicted in Fig.~\ref{CFigure} for $0 < \theta < \pi$; there is an analogous picture for $-\pi < \theta < 0$. Here $t^{1/2} = \vert t \vert^{1/2} e^{i\theta /2}$ where $-\pi < \theta < \pi$ is the analytic continuation of $\sqrt{t}$ from $(0, \infty)$ to $\mathbb{C} \backslash (-\infty, 0]$.
 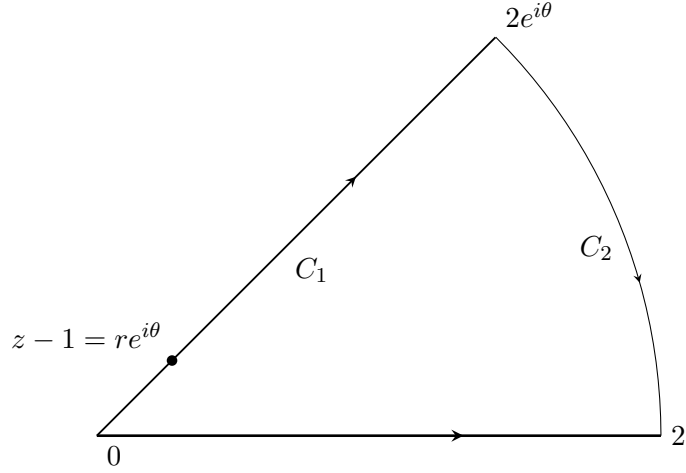
\begin{figure}[t]
\centering
\begin{tikzpicture}[scale=5.]
\coordinate (up) at (.5,.2);
\coordinate (dn) at (.5,-.2);
\draw[line width = 1,directed] (-1,0) to (.5,0) ;
\draw[line width = .7,directed] (-1,0) to (.06,1.06);
\draw [black, domain= 45:0, directed] plot ({1.5 * cos(\x) -1}, {1.5 * sin(\x)});
\fill (-.8, .2) circle (.4pt);
\node[below right] at (-1,0) {$0$};
\node[above left] at (-.8, .2) {$z-1= re^{i\theta}$};
\node[below right] at (-.5, .5){$C_1$};
\node[left] at (.4, .5){$C_2$};
\node[above right] at (.06, 1.06) {$2e^{i\theta}$};
\node[right] at (.5, 0){$2$};
\end{tikzpicture}
\caption{Definition of $C_1$, $C_2$.}\label{CFigure}
\end{figure}
The integral over the contour $C_2$ is bounded away from all singularities of the integrand, and thus, by dominated convergence, tends uniformly in $0 < \vert \theta\vert < \pi$ as $r \to 0$ to
\begin{gather*}
 \int_{C_2}\log w(1-t) \frac{\dbar t}{t^{3/2}},
\end{gather*}
which is finite. Therefore,
\begin{gather*}
\int_{-1}^1 \frac{\log w(s)}{\big(s^2-1\big)^{1/2}_+} \frac{\dbar s}{s -z} = \frac{i}{\sqrt{2}} \int_{C_1} \frac{\log w(1-t)}{t^{1/2}} \frac{\dbar t}{t + re^{i\theta}} + O(1).
\end{gather*}
Making the change of variables $t = re^{i\theta}u$, $u > 0$, $\gamma = \sqrt{u}$,
 \begin{gather}
 \frac{i}{\sqrt{2}} \int_{C_1} \frac{\log w(1-t)}{t^{1/2}} \frac{\dbar t}{t + re^{i\theta}} = \frac{ir^{-1/2}e^{-i\theta/2}}{\sqrt{2}} \int_0^{2/r}\log w(1 - re^{i\theta}u) \frac{\dbar u}{u^{3/2} + u^{1/2}}\nonumber\\
\qquad {} =\sqrt{2}ir^{-1/2}e^{-i\theta/2} \int_0^{\sqrt{2}r^{-1/2}}\log w\big(1-re^{i\theta}\gamma^2\big) \frac{\dbar\gamma}{\gamma^2 + 1} \nonumber\\
\qquad{} = \sqrt{2}ir^{-1/2}e^{-i\theta/2} \int_0^{\sqrt{2}r^{-1/2}}\log \left(\log \frac{2k}{r} - 2\log \gamma - i\theta\right)\frac{\dbar\gamma}{\gamma^2 + 1}\nonumber\\
\qquad{} = \sqrt{2}ir^{-1/2}e^{-i\theta/2}\left( \int_0^{\left(\frac{r}{2k}\right)^{1/4}} + \int_{\left(\frac{r}{2k}\right)^{1/4}}^{\left(\frac{r}{2k}\right)^{-1/4}} + \int_{\left(\frac{r}{2k}\right)^{-1/4}}^{\sqrt{2}r^{-1/2}}\right)\nonumber\\
\qquad\quad{}\times \log \left(\log \frac{2k}{r} - 2\log \gamma - i\theta\right)\frac{\dbar\gamma}{\gamma^2 + 1}\nonumber\\
\qquad {}= \sqrt{2}ir^{-1/2}e^{-i\theta/2} (N_1(r) + N_2(r) + N_3(r)). \label{ApproxWithN}
\end{gather}
When $\gamma \le \left(\frac{r}{2k}\right)^{1/4}$, we have that $\log \frac{2k}{r} \le - 4 \log \gamma$. So as $-2 \log \gamma > 0$,
\begin{gather*}
\vert N_1 (r) \vert = \left\vert \int_0^{\left(\frac{r}{2k}\right)^{1/4}}\log \left(\log \frac{2k}{r} - 2\log \gamma - i\theta\right)\frac{\dbar \gamma}{\gamma^2 + 1} \right\vert \\
\hphantom{\vert N_1 (r) \vert}{}
 \le \left\vert \int_0^{\left(\frac{r}{2k}\right)^{1/4}}\log \left(\log \frac{2k}{r} - 2\log \gamma\right)\frac{\dbar \gamma}{\gamma^2 + 1} \right\vert \\
\hphantom{\vert N_1 (r) \vert=}{}
+ \left\vert \int_0^{\left(\frac{r}{2k}\right)^{1/4}}\log \left(1 - \frac{i\theta}{\log \frac{2k}{r} - 2\log \gamma }\right)\frac{\dbar \gamma}{\gamma^2 + 1} \right\vert.
 \end{gather*}
As $\log\frac{2k}{r} -2\log \gamma \ge \log\frac{2k}{r} \to \infty$ as $r \to \infty$, the integrand in the second integral is bounded above and therefore
\begin{gather*}
\vert N_1(r)\vert\le \frac{1}{2\pi} \int_0^{\left(\frac{r}{2k}\right)^{1/4}}\left \vert \log \left(\log \frac{2k}{r} - 2\log \gamma \right)\right\vert\frac{{\rm d}\gamma}{\gamma^2 + 1} + O\big(r^{1/4}\big).
\end{gather*}
$\vert \log x \vert$ is an increasing function when $x > 1$, therefore since $\log\frac{2k}{r} \le -4\log \gamma$ and $\log\frac{2k}{r} - 2\log \gamma >1$ as $r \to 0$,
\begin{gather*}
\vert N_1(r)\vert\le \frac{1}{2\pi} \int_0^{\left(\frac{r}{2k}\right)^{1/4}} \log \left(- 6\log \gamma \right) \frac{{\rm d}\gamma}{\gamma^2 + 1} +O\big(r^{1/4}\big)\nonumber\\
\hphantom{\vert N_1(r)\vert}{} \le \frac{1}{2\pi} \int_0^{\left(\frac{r}{2k}\right)^{1/4}} \log \left(- 6\log \gamma \right) {\rm d}\gamma + O\big(r^{1/4}\big)
\end{gather*}
Integrating by parts,
\begin{gather*}
 \int_0^{\left(\frac{r}{2k}\right)^{1/4}} \log (- 6\log \gamma ) {\rm d}\gamma = \log(-6\log \gamma) \gamma \bigg\vert_0^{r^{1/4}/(2k)^{1/4}} - \int_0^{r^{1/4}/(2k)^{1/4}} \frac{1}{\log \gamma }{\rm d}\gamma\\
 \hphantom{\int_0^{\left(\frac{r}{2k}\right)^{1/4}} \log (- 6\log \gamma ) {\rm d}\gamma}{}
 = O\big(r^{1/4} \log (-\log r) \big),
\end{gather*}
and so
\begin{gather}
\vert N_1(r)\vert = O\big(r^{1/4} \log (-\log r) \big).\label{N1Estimate}
\end{gather}
Now,
\begin{gather*}
\vert N_3 (r) \vert = \left\vert \int_{\left(\frac{r}{2k}\right)^{-1/4}}^{\sqrt{2}r^{-1/2}}\log \left(\log \frac{2k}{r} - 2\log \gamma - i\theta \right)\frac{\dbar\gamma}{\gamma^2 + 1} \right\vert,
\end{gather*}
which, after the change of variables $y = \frac{1}{\gamma}$, becomes
\begin{gather*}
\left\vert \int^{\left(\frac{r}{2k}\right)^{1/4}}_{r^{1/2}/\sqrt{2}}\log \left(\log \frac{2k}{r} + 2\log y - i\theta\right)\frac{\dbar y}{y^2 + 1}\right\vert\nonumber\\
\qquad{}\le \frac{1}{2\pi}\int^{\left(\frac{r}{2k}\right)^{1/4}}_{r^{1/2}/\sqrt{2}}\left\vert \log \left(\log \frac{2k}{r} + 2\log y - i\theta\right)\right\vert\frac{{\rm d} y}{y^2 + 1}\nonumber\\
\qquad{} \le \frac{1}{2\pi}\int^{\left(\frac{r}{2k}\right)^{1/4}}_{r^{1/2}/\sqrt{2}}\left\vert \frac{1}{2}\log\left(\left(\log \frac{2k}{r} + 2 \log y\right)^2 + \theta^2\right)\right. \nonumber\\
\left. \qquad\quad{} + i\arg\left( \left(\log \frac{2k}{r} + 2\log y - i\theta\right)\right)\right\vert\frac{{\rm d} y}{y^2 + 1}\nonumber\\
\qquad{} \le \frac{1}{4\pi}\int^{\left(\frac{r}{2k}\right)^{1/4}}_{r^{1/2}/\sqrt{2}}\left\vert \log\left(\left(\log\frac{2k}{r} + 2 \log y\right)^2 + \theta^2\right)\right\vert\frac{{\rm d} y}{y^2 + 1} + O\big(r^{1/4}\big).
\end{gather*}
When $\frac{r^{1/2}}{\sqrt{2}} \le y \le \big(\frac{r}{2k}\big)^{1/4}$, we have that $-2 \log y + \log\frac{2k}{2}\le \log \frac{2k}{r} \le - 4 \log y$. So
\begin{gather}
\vert N_3(r)\vert \le \frac{1}{4\pi}\int^{\left(\frac{r}{2k}\right)^{1/4}}_{r^{1/2}/\sqrt{2}}\left(\big\vert \log \big((- 2\log y)^2 + \theta^2\big)\big\vert +\left\vert \log\left(\left(\log\frac{2k}{2}\right)^2 +\theta^2\right)\right\vert\right)\frac{{\rm d} y}{y^2 + 1}\nonumber\\
\qquad{} \le \frac{1}{4\pi}\int^{\left(\frac{r}{2k}\right)^{1/4}}_{r^{1/2}/\sqrt{2}}\left\vert \log \left(4\log^2y + \theta^2\right)\right\vert\frac{{\rm d} y}{y^2 + 1} + O\big(r^{1/4}\big) \nonumber\\
\qquad{} \le \frac{1}{4\pi}\int^{\left(\frac{r}{2k}\right)^{1/4}}_{r^{1/2}/\sqrt{2}}\left\vert \log \left(4\log^2y\right)\right\vert\frac{{\rm d} y}{y^2 + 1} \nonumber\\
\qquad\quad{} + \frac{1}{2\pi}\int^{\left(\frac{r}{2k}\right)^{1/4}}_{r^{1/2}/\sqrt{2}}\left\vert \log \left(1 + \frac{\theta^2}{4\log^2y}\right)\right\vert\frac{{\rm d} y}{y^2 + 1} + O\big(r^{1/4}\big)\nonumber\\
\qquad{} \le \frac{1}{2\pi} \int^{\left(\frac{r}{2k}\right)^{1/4}}_{r^{1/2}/\sqrt{2}}\log \left(- 2\log y\right)\frac{{\rm d} y}{y^2 + 1}+ O\big(r^{1/4}\big)\nonumber\\
\qquad{} = O\big(r^{1/4} \log (-\log r)\big),\label{N3Estimate}
\end{gather}
where, in the final step, we use the same argument as in the proof of (\ref{N1Estimate}).

When $\left(\frac{r}{2k}\right)^{1/4} \le \gamma \le \left(\frac{r}{2k}\right)^{-1/4}$, we have $\vert \log \gamma\vert \le \frac{1}{4} \left\vert\log \frac{2k}{re^{i\theta}}\right\vert$. Therefore for $\gamma$ in this region, we can use the power series expansion for the logarithm
\begin{gather*}
\log \left(\log \frac{2k}{r} - 2\log \gamma - i\theta \right)= \log \left(\log \frac{2k}{re^{i\theta}} - 2\log \gamma\right) = \log \log\frac{2k}{re^{i\theta}} +\log \left(1 - \frac{2 \log \gamma}{\log\frac{2k}{re^{i\theta}}} \right) \nonumber\\
\hphantom{\log \left(\log \frac{2k}{r} - 2\log \gamma - i\theta \right)}{} = \log \log \frac{2k}{re^{i\theta}} - \frac{2\log \gamma}{\log \frac{2k}{re^{i\theta}}} - \frac{2 \log^2 \gamma}{\log^2 \frac{2k}{re^{i\theta}}} + O\left(\frac{\log^3 \gamma}{\log^3r}\right),
\end{gather*}
where again the term $O\left(\frac{\log^3\gamma}{\log^3r}\right)$ is uniform in $\theta$. Therefore
\begin{gather}
N_2(r) = \int_{\left(\frac{r}{2k}\right)^{1/4}}^{\left(\frac{r}{2k}\right)^{-1/4}} \left[\log \log \frac{2k}{re^{i\theta}} - \frac{2\log \gamma}{\log \frac{2k}{re^{i\theta}}} -\frac{2 \log^2 \gamma}{\log^2 \frac{2k}{re^{i\theta}}} + O\left(\frac{\log^3 \gamma}{\log^3r}\right)\right] \frac{\dbar \gamma}{\gamma^2 + 1} \nonumber\\
\qquad{} =\log \log\frac{2k}{re^{i\theta}}\int_{\left(\frac{r}{2k}\right)^{1/4}}^{\left(\frac{r}{2k}\right)^{-1/4}} \frac{\dbar \gamma}{\gamma^2 + 1} - \frac{1}{\log\frac{2k}{re^{i\theta}}} \int_{\left(\frac{r}{2k}\right)^{1/4}}^{\left(\frac{r}{2k}\right)^{-1/4}} \frac{2\log \gamma \dbar \gamma}{\gamma^2 + 1}\nonumber\\
\qquad\quad{} - \frac{1}{\log^2\frac{2k}{re^{i\theta}}}\int_{\left(\frac{r}{2k}\right)^{1/4}}^{\left(\frac{r}{2k}\right)^{-1/4}} \frac{2\log^2\gamma \dbar \gamma}{\gamma^2 + 1} + \frac{1}{\log^3 r}\int_{\left(\frac{r}{2k}\right)^{1/4}}^{\left(\frac{r}{2k}\right)^{-1/4}}\frac{O\left(\log^3 \gamma\right)}{\gamma^2 + 1}\nonumber\\
\qquad{} =\log \log\frac{2k}{re^{i\theta}}\int_0^\infty \frac{\dbar \gamma}{\gamma^2 + 1} - \frac{1}{\log\frac{2k}{re^{i\theta}}}\int_0^\infty \frac{2\log \gamma \dbar \gamma}{\gamma^2 + 1}\nonumber\\
\qquad\quad{} - \frac{1}{\log^2\frac{2k}{re^{i\theta}}}\int_0^\infty \frac{2\log^2\gamma \dbar \gamma}{\gamma^2 + 1} + O\left(\frac{1}{\log^3r}\right).\label{N2Expansion}
\end{gather}
The first integral in (\ref{N2Expansion}) is easily calculated
\begin{gather}
\int_0^\infty \frac{\dbar \gamma}{\gamma^2+1} = \frac{1}{2\pi i}\arctan \gamma \bigg\vert_0^\infty = \frac{1}{4 i}.\label{N2FirstTerm}
\end{gather}
The second integral, after a change of variables $\eta = \frac{1}{\gamma}$, satisfies
\begin{gather*}
\int_0^\infty \frac{\log \gamma \dbar \gamma}{\gamma^2 + 1}= -\int_0^\infty \frac{\log \eta \dbar\eta}{\eta^2 + 1},
\end{gather*}
and so
\begin{gather}
\int_0^\infty \frac{\log \gamma \dbar \gamma}{\gamma^2 + 1}= 0.\label{N2SecondTerm}
\end{gather}
The final integral can be calculated by the method of residues, which yields:
\begin{gather}
2\int_0^\infty \frac{\log^2 \gamma \dbar \gamma}{\gamma^2+1} = \frac{\pi^2}{8i}.\label{N2ThirdTerm}
\end{gather}
Combining (\ref{N2Expansion}) with (\ref{N2FirstTerm}), (\ref{N2SecondTerm}), and (\ref{N2ThirdTerm}), we see that
\begin{gather}
N_2(r) = \frac{1}{4i} \log \log\frac{2k}{re^{i\theta}} - \frac{\pi^2}{8 i} \frac{1}{\log^2\frac{2k}{re^{i\theta}}} + O\left(\frac{1}{\log^3 r}\right), \label{N2Asymptotic}
\end{gather}
and combining (\ref{ApproxWithN}) with (\ref{N1Estimate}), (\ref{N3Estimate}), and (\ref{N2Asymptotic}), we see that
\begin{gather}
\int_{-1}^1 \frac{\log w(s)}{\big(s^2-1\big)^{1/2}_+} \frac{\dbar s}{s -z}\nonumber\\
\qquad{} = \frac{1}{\sqrt{2}}r^{-1/2}e^{-i\theta/2}\left(\frac{1}{2}\log \log \frac{2k}{re^{i\theta}} - \frac{\pi^2}{4} \frac{1}{\log^2\frac{2k}{r}} + O\left(\frac{1}{\log^3 r}\right)\right) + O(1).\label{CauchyEstimate}
\end{gather}
Now note that for $z = 1 + re^{i\theta}$,
\begin{gather}
\big( z^2-1\big)^{1/2} = \sqrt{2} r^{1/2}e^{i\theta/2} + O(r). \label{FrontEstimate}
\end{gather}
So combining (\ref{CauchyEstimate}) and (\ref{FrontEstimate}) with the definition of $F$, we see that
\begin{gather*}
\log F(z) = \frac{1}{2}\log \log \frac{2k}{re^{i\theta}} - \frac{\pi^2}{4} \frac{1}{\log^2\frac{2k}{re^{i\theta}}} + O\left(\frac{1}{\log^3 r}\right),
\end{gather*}
and so
\begin{gather*}
F^2(z) = \log \frac{2k}{re^{i\theta}} \left( 1 - \frac{\pi^2}{2} \frac{1}{\log^2 \frac{2k}{re^{i\theta}}} + O\left(\frac{1}{\log^3 r}\right)\right).
\end{gather*}
Now for $z = 1+re^{i\theta} \in \mathbb{C}_\pm$, and recalling that $\log$ always refers to the principle branch,
\begin{gather*}
\log \frac{2k}{re^{i\theta}} = \log\frac{2k}{re^{i(\theta \mp \pi)}} \mp i\pi= \log\frac{2k}{-re^{i\theta}} \mp i\pi = w(z) \mp i\pi
\end{gather*}
from which it follows that
\begin{gather*}
\frac{F^2}{w}(z) = 1 \mp \frac{i\pi}{w(z)} - \frac{\pi^2}{2 w^2(z)} + O\left(\frac{1}{w^3(z)}\right)
\end{gather*}
uniformly for $z \in \mathbb{C}_\pm$ as $z \to 1$, which completes the proof of Proposition \ref{FAsymptoticsZApp}.
\end{proof}

The following proves assertion~\eqref{prop3-2.3} of Proposition \ref{FProps}
\begin{prop}\label{-1FAsymptotics}
\begin{gather*}
\frac{F^2}{w}(z) = 1 + O\big(\vert z+1\vert^{1/2}\big) \label{Something237}
\end{gather*}
as $z \to -1$, $z \in \mathbb{C} \backslash [-1, 1]$.
\end{prop}
\begin{proof}
Note that as $z\to -1$, uniformly for $z \in \mathbb{C} \backslash [-1,1]$,
\begin{gather*}
\int_{-1}^1 \frac{\log w(s)}{\big(s^2-1\big)^{1/2}_+} \frac{\dbar s}{s-z} = \int_{-1}^1 \frac{\log w(-1)}{\big(s^2-1\big)^{1/2}_+} \frac{\dbar s}{s-z} + \int_{-1}^1 \frac{\log w(s) - \log (w(-1))}{\big(s^2-1\big)^{1/2}_+} \frac{\dbar s}{s-z}.
\end{gather*}
By the proof of Lemma \ref{WAnalytic}, the second integral is easily seen to be $O(1)$. Therefore
\begin{gather*}
\int_{-1}^1 \frac{\log w(s)}{\big(s^2-1\big)^{1/2}_+} \frac{\dbar s}{s-z} = \int_{-1}^1 \frac{\log w(-1)}{\big(s^2-1\big)^{1/2}_+} \frac{\dbar s}{s-z} + O(1)\\
\hphantom{\int_{-1}^1 \frac{\log w(s)}{\big(s^2-1\big)^{1/2}_+} \frac{\dbar s}{s-z}}{}
=\frac{1}{2}\log w(-1) \int_C \frac{1}{(s^2 - 1)^{1/2}}\frac{\dbar s}{s-z} + O(1),
\end{gather*}
where $C$ refers to any clockwise contour containing $[-1,1]$ but not the point $z$. Letting $C$ go to infinity, a residue calculation gives
\begin{gather*}
\frac{1}{2}\log w(-1) \int_C \frac{1}{\big(s^2 - 1\big)^{1/2}}\frac{\dbar s}{s-z} = \frac{1}{2} \frac{\log w(-1)}{\big(z^2-1\big)^{1/2}}.
\end{gather*}
Thus
\begin{gather*}
F^2(z) = \exp \left(2\big(z^2-1\big)^{1/2} \int_{-1}^1 \frac{\log w(s)}{\big(s^2-1\big)^{1/2}_+} \frac{\dbar s}{s-z}\right) = \exp \big(\log w(-1) + O\big( \vert z + 1\vert^{1/2}\big)\big) \nonumber\\
\hphantom{F^2(z)}{} = w(-1) + O\big( \vert z + 1\vert^{1/2}\big),
\end{gather*}
and so
\begin{gather*}
\frac{F^2}{w}(z) = 1 + O\big(\vert z+1\vert^{1/2}\big)
\end{gather*}
as $z \to -1$, $z \in \mathbb{C} \backslash [-1,1]$, which proves Proposition \ref{-1FAsymptotics}.
\end{proof}

\subsection[Estimate on $\frac{F^2}{w}(1+r) - \frac{F^2}{w}(1+\tilde{r})$, $r, \tilde{r} > 0$]{Estimate on $\boldsymbol{\frac{F^2}{w}(1+r) - \frac{F^2}{w}(1+\tilde{r})}$, $\boldsymbol{r, \tilde{r} > 0}$}
\begin{prop} \label{F-FAsymptoticsApp}
Fix $R > 0$ and suppose $r, \tilde{r} > 0$ obey
\begin{gather*}
r, \tilde{r} = O\left(\frac{1}{n^2}\right), \qquad n\left(\frac{r}{\tilde{r}} - 1\right) \in [-R, R].
\end{gather*}
Then
\begin{gather*}
\frac{F^2}{w_+}(1 + r) - \frac{F^2}{w_+} (1 + \tilde{r}) + \frac{F^2}{w_-}(1 + r) - \frac{F^2}{w_-} (1 + \tilde{r}) = O_R\left(\frac{1}{n\log^3n}\right).
\end{gather*}
\end{prop}

\begin{Remark}
Notice that if $r = O\big(\frac{1}{n^2}\big)$, then Proposition~\ref{FCancellation} implies $\frac{F^2}{w_+}(1 + r) + \frac{F^2}{w_-}(1 + r) - 2 = O\big(\frac{1}{\log^2n}\big)$. Also notice that $\frac{1}{\log^2 n} - \frac{1}{\log^2(n + c)} = O\big(\frac{1}{n\log^3n}\big)$. Therefore we would expect $\frac{F^2}{w_+}(1 + r) - \frac{F^2}{w_+} (1 + \tilde{r}) + \frac{F^2}{w_-}(1 + r) - \frac{F^2}{w_-} (1 + \tilde{r}) = O\big(\frac{1}{n\log^3n}\big)$ for sufficiently close $r$, $\tilde{r}$. Proposition~\ref{F-FAsymptoticsApp} makes this rigorous.
\end{Remark}

\begin{proof}
We define $a = a(r,\tilde{r}): = n \big(\frac{r}{\tilde{r}} - 1\big)$. Clearly $\frac{r}{\tilde{r}} = 1 + \frac{a}{n}$. In the calculations that follow, we assume without loss of generality that $r > \tilde{r}$, and hence $a > 0$. First, note that
\begin{gather}
\log F(1+r)= \big((1+r)^2 - 1\big)^{1/2} \int_{-1}^1 \frac{\log w(s)}{\big(s^2 - 1\big)^{1/2}_+} \frac{\dbar s}{s - (1 + r)} \nonumber\\
\hphantom{\log F(1+r)}{} = \big(\sqrt{2} r^{1/2} + O\big(r^{3/2}\big)\big) \int_{-1}^1 \frac{\log w(s)}{\big(s^2 - 1\big)^{1/2}_+} \frac{\dbar s}{s - (1 + r)}. \label{yoyo}
\end{gather}
By (\ref{CauchyEstimate}), which holds uniformly for $z \in \mathbb{C}_\pm$, we have
\begin{gather*}
 \int_{-1}^1 \frac{\log w(s)}{(s^2 - 1)^{1/2}_+} \frac{\dbar s}{s - (1 + r)} = O\big(r^{-1/2} \log \log r\big),
\end{gather*}
and so, since $r = O\big(n^{-2}\big)$,
\begin{gather*}
\log F(1+r) = \sqrt{2}r^{1/2} \int_{-1}^1 \frac{\log w(s)}{(s^2 - 1)^{1/2}_+} \frac{\dbar s}{s - (1 + r)} + O\left(\frac{\log \log n}{n^2}\right).
\end{gather*}
Therefore
\begin{gather}
\log\frac{F(1 + r)}{F(1 + \tilde{r})} =\sqrt{2}r^{1/2}\int_{-1}^1 \frac{\log w(s)}{\big(s^2-1\big)_+^{1/2}} \frac{\dbar s}{s-(1 + r)} \nonumber\\
\hphantom{\log\frac{F(1 + r)}{F(1 + \tilde{r})} =}{} - \sqrt{2}\tilde{r}^{1/2}\int_{-1}^1 \frac{\log w(s)}{\big(s^2-1\big)_+^{1/2}} \frac{\dbar s}{s-(1 + \tilde{r})} + O\left(\frac{\log \log n}{n^2}\right).\label{3.6}
\end{gather}
In (\ref{3.6}),
\begin{gather*}
\sqrt{2}r^{1/2}\int_{-1}^1 \frac{\log w(s)}{\big(s^2-1\big)_+^{1/2}} \frac{\dbar s}{s-(1 + r)} - \sqrt{2}\tilde{r}^{1/2}\int_{-1}^1 \frac{\log w(s)}{\big(s^2-1\big)_+^{1/2}} \frac{\dbar s}{s-(1 + \tilde{r})} \nonumber\\
\qquad{} = -\sqrt{2}ir^{1/2}\int_{-1}^1 \frac{\log w(s)}{\sqrt{1-s^2}} \frac{\dbar s}{s-(1 + r)} + \sqrt{2}i\tilde{r}^{1/2}\int_{-1}^1 \frac{\log w(s)}{\sqrt{1-s^2}} \frac{\dbar s}{s-(1 + \tilde{r})}.
\end{gather*}
Making the change of variables $t = 1-s$,
\begin{gather}
 -\sqrt{2}ir^{1/2}\int_{-1}^1 \frac{\log w(s)}{\sqrt{1-s^2}} \frac{\dbar s}{s-(1 + r)} + \sqrt{2}i\tilde{r}^{1/2}\int_{-1}^1 \frac{\log w(s)}{\sqrt{1-s^2}} \frac{\dbar s}{s-(1 + \tilde{r})}\nonumber\\
\qquad{}= -\sqrt{2}ir^{1/2}\int_{2}^0 \frac{\log w(1-t)}{\sqrt{t (2-t)}} \frac{-\dbar t}{-t -r} +\sqrt{2}i\tilde{r}^{1/2}\int_{2}^0 \frac{\log w(1-t)}{\sqrt{t (2-t)}} \frac{-\dbar t}{-t -\tilde{r}} \nonumber \\
\qquad{}= \sqrt{2}ir^{1/2}\int_0^2 \frac{\log w(1-t)}{\sqrt{t (2-t)}} \frac{\dbar t}{t +r} -\sqrt{2}i\tilde{r}^{1/2}\int_0^2 \frac{\log w(1-t)}{\sqrt{t (2-t)}} \frac{\dbar t}{t +\tilde{r}}\nonumber \\
\qquad{} = ir^{1/2}\int_0^2 \frac{\log w(1-t)}{\sqrt{t}} \frac{\dbar t}{t +r} -i\tilde{r}^{1/2}\int_0^2 \frac{\log w(1-t)}{\sqrt{t}} \frac{\dbar t}{t +\tilde{r}} + iH(r,\tilde{r}),\label{35}
\end{gather}
where
\begin{gather*}
H(r,\tilde{r}) :=
 \sqrt{2}\left[r^{1/2}\int_0^{2} \frac{\log w(1-t)}{\sqrt{t (2-t)}} \frac{\dbar t}{t +r} -\tilde{r}^{1/2}\int_0^{2} \frac{\log w(1-t)}{\sqrt{t (2-t)}} \frac{\dbar t}{t +\tilde{r}}\right]\nonumber\\
\hphantom{H(r,\tilde{r}) :=}{} -\left[r^{1/2}\int_0^2 \frac{\log w(1-t)}{\sqrt{t}} \frac{\dbar t}{t +r} -\tilde{r}^{1/2}\int_0^2 \frac{\log w(1-t)}{\sqrt{t}} \frac{\dbar t}{t +\tilde{r}}\right].
\end{gather*}
The same calculations as in the proof of Proposition \ref{FAsymptoticsZApp} show that
\begin{gather*}
H(r,\tilde{r})
 =r^{1/2} \int_0^{2} \log w(1-t)\left[ \frac{\sqrt{t}}{\sqrt{2}\sqrt{2-t} + (2-t) }\right] \frac{\dbar t}{t+r} \nonumber\\
\hphantom{H(r,\tilde{r})=}{} -\tilde{r}^{1/2}\int_0^{2} \log w(1-t)\left[ \frac{\sqrt{t}}{\sqrt{2}\sqrt{2-t} + (2-t) }\right] \frac{\dbar t}{t+\tilde{r}} \nonumber\\
\hphantom{H(r,\tilde{r})}{} = \big(r^{1/2} - \tilde{r}^{1/2}\big) \int_0^{2} \log w(1-t)\left[ \frac{\sqrt{t}}{\sqrt{2}\sqrt{2-t} + (2-t) }\right] \frac{\dbar t}{t+r}\nonumber\\
\hphantom{H(r,\tilde{r})=}{} + \tilde{r}^{1/2} \int_0^2 \log w(1-t)\left[ \frac{\sqrt{t}}{\sqrt{2}\sqrt{2-t} + (2-t) }\right] \left( \frac{\dbar t}{t+r} - \frac{\dbar t}{t+\tilde{r}}\right).
 \end{gather*}
 So
 \begin{gather*}
\vert H(r, \tilde{r})\vert \le \big\vert r^{1/2}- \tilde{r}^{1/2}\big\vert\int_0^{2} \vert \log w(1-t)\vert \left[ \frac{\sqrt{t}}{\sqrt{2}\sqrt{2-t} + (2-t) }\right] \frac{{\rm d} t}{t}\nonumber\\
 \hphantom{\vert H(r, \tilde{r})\vert \le}{} + \tilde{r}^{1/2}\int_0^2 \vert \log w(1-t)\vert\left[ \frac{\sqrt{t}}{\sqrt{2}\sqrt{2-t} +(2-t) }\right] \frac{\vert \tilde{r}-r\vert {\rm d} t}{\vert t + r \vert \cdot \vert t + \tilde{r}\vert }\nonumber\\
 \hphantom{\vert H(r, \tilde{r})\vert}{} \le c_1 \tilde{r}^{1/2} \left\vert \frac{r^{1/2}}{\tilde{r}^{1/2}} - 1\right\vert + \tilde{r}^{1/2} \frac{\vert \tilde{r}-r\vert }{ \tilde{r} }\int_0^2 \vert \log w(1-t)\vert\left[ \frac{\sqrt{t}}{\sqrt{2}\sqrt{2-t} +(2-t) }\right] \frac{{\rm d}t}{t} \nonumber\\
 \hphantom{\vert H(r, \tilde{r})\vert }{} \le c_1 \tilde{r}^{1/2} \left\vert \frac{r^{1/2}}{\tilde{r}^{1/2}} - 1\right\vert + c_2 \tilde{r}^{1/2} \left\vert\frac{r}{\tilde{r}} - 1\right\vert = O\big(n^{-2}\big).
\end{gather*}

 Making the changes of variables $t \to rt$ and $t \to \tilde{r} t$ in (\ref{35}) within the respective integrals, we have
\begin{gather}
 ir^{1/2}\int_0^2 \frac{\log w(1-t)}{\sqrt{t}} \frac{\dbar t}{t +r} -i\tilde{r}^{1/2}\int_0^2 \frac{\log w(1-t)}{\sqrt{t}} \frac{\dbar t}{t +\tilde{r}} \nonumber\\
\quad{} = ir^{1/2}\int_0^{2/r} \frac{\log w(1-rt)}{\sqrt{rt}} \frac{\dbar t}{t +1} -i\tilde{r}^{1/2}\int_0^{2/\tilde{r}} \frac{\log w(1-\tilde{r}t)}{\sqrt{\tilde{r}t}} \frac{\dbar t}{t +1}\nonumber\\
\quad{}= i\int_0^{2/r} \frac{\log w(1-rt)}{\sqrt{t}} \frac{\dbar t}{t +1} -i\int_0^{2/r} \frac{\log w(1-\tilde{r}t)}{\sqrt{t}} \frac{\dbar t}{t +1} -i \int_{2/r}^{2/\tilde{r}} \frac{\log w(1-\tilde{r}t)}{\sqrt{t}} \frac{\dbar t}{t +1}\nonumber\\
\quad{}= i \int_0^{2/r} \left(\log w(1-rt) - \log w(1-\tilde{r}t)\right) \frac{\dbar t}{t^{3/2}+ t^{1/2}} -i \int_{2/r}^{2/\tilde{r}} \frac{\log w(1-\tilde{r}t)}{\sqrt{t}} \frac{\dbar t}{t +1}.
\label{r-rtildeReduction}
\end{gather}
Note that, as $\log k \le \log\frac{2k}{\tilde{r}t} \le \log\frac{k}{1 + O\left(\frac{1}{n}\right)}$ for $t \in [\frac{2}{r}, \frac{2}{\tilde{r}}]$, we have
\begin{gather*}
\left\vert \int_{2/r}^{2/\tilde{r}} \frac{\log w(1-\tilde{r}t)}{\sqrt{t}} \frac{\dbar t}{t +1}\right\vert\le \left\Vert \frac{1}{\sqrt{t}} \frac{1}{t +1} \right\Vert_{L^\infty(2/r, 2/\tilde{r})} \int_{2/r}^{2/\tilde{r}} \left\vert \log w(1 - \tilde{r}t)\right\vert \dbar t \nonumber\\
\hphantom{\left\vert \int_{2/r}^{2/\tilde{r}} \frac{\log w(1-\tilde{r}t)}{\sqrt{t}} \frac{\dbar t}{t +1}\right\vert}{} \le c r^{3/2}\left\vert \frac{2}{\tilde{r}} - \frac{2}{r}\right\vert = c r^{1/2} \left\vert \frac{r}{\tilde{r}} - 1\right\vert = O\big(n^{-2}\big).
 \end{gather*}
To estimate the remaining term in (\ref{r-rtildeReduction}),
\begin{gather}
i \int_0^{2/r} \left(\log w(1-rt) - \log w(1-\tilde{r}t)\right) \frac{\dbar t}{t^{3/2}+ t^{1/2}} \nonumber\\
\qquad{} =i \int_0^{2/r} \log \left(\frac{w(1 - rt)}{w(1-\tilde{r}t)}\right) \frac{\dbar t}{t^{3/2}+ t^{1/2}}. \label{FurtherReduction}
\end{gather}
Now, for $t \in [0, \frac{2}{r}]$,
\begin{gather*}
\frac{w(1 - rt)}{w(1-\tilde{r}t)} = \frac{\log\big(\frac{2k}{rt}\big)}{\log\big(\frac{2k}{\tilde{r}t}\big)} = 1 + \frac{\log\big(\frac{2k}{rt}\big) -\log\big(\frac{2k}{\tilde{r}t}\big) }{\log\big(\frac{2k}{\tilde{r}t}\big)}\nonumber\\
\hphantom{\frac{w(1 - rt)}{w(1-\tilde{r}t)}}{} = 1 + \frac{\log \frac{\tilde{r}}{r}}{\log\big(\frac{2k}{\tilde{r}t}\big)} = 1 + \frac{\log \big(1 + \frac{a}{n}\big)}{\log\big(\frac{2k}{\tilde{r}t}\big)} = 1 + \frac{a}{n\log\big(\frac{2k}{\tilde{r}t}\big)} + O\big(n^{-2}\big).
\end{gather*}
So
\begin{gather}
\log\left(\frac{w(1 - rt)}{w(1-\tilde{r}t)}\right) = \frac{a}{n\log\big(\frac{2k}{\tilde{r}t}\big)} + O\big(n^{-2}\big). \label{3.16}
\end{gather}
Substituting (\ref{3.16}) into (\ref{FurtherReduction}) yields
\begin{gather*}
i\int_0^{2/r} \left( \frac{a}{n\log\big(\frac{2k}{\tilde{r}t}\big)} + O\big(n^{-2}\big)\right)\frac{\dbar t}{t^{3/2}+ t^{1/2}} \nonumber\\
\qquad{}= i \int_{r^{1/2}}^{r^{-1/2}}\frac{a}{n\log\big(\frac{2k}{\tilde{r}t}\big)}\frac{\dbar t}{t^{3/2}+ t^{1/2}} + E(r) + O\big(n^{-2}\big),
\end{gather*}
where
\begin{gather*}
E(r) = i \int_0^{r^{1/2}}\frac{a}{n\log\big(\frac{2k}{\tilde{r}t}\big)}\frac{\dbar t}{t^{3/2}+ t^{1/2}} + i \int_{r^{-1/2}}^{2/r}\frac{a}{n\log\big(\frac{2k}{\tilde{r}t}\big)}\frac{\dbar t}{t^{3/2}+ t^{1/2}}.%\label{EDef}
\end{gather*}
Now
\begin{gather*}
\left\vert \int_0^{r^{1/2}}\frac{a}{n\log\big(\frac{2k}{\tilde{r}t}\big)}\frac{\dbar t}{t^{3/2}+ t^{1/2}}\right\vert \le \int_0^{r^{1/2}}\frac{c}{n\vert\log r\vert}\frac{d t}{t^{1/2}} \le \frac{cr^{1/4}}{n\log r} = O\big(n^{-3/2}\big).
\end{gather*}
Similarly
\begin{gather*}
\left\vert \int_{r^{-1/2}}^{2/r}\frac{a}{n\log\big(\frac{2k}{\tilde{r}t}\big)}\frac{\dbar t}{t^{3/2}+ t^{1/2}}\right\vert \le \int_{r^{-1/2}}^{2/r}\frac{c}{n\vert\log r\vert}\frac{d t}{t^{3/2}} \le \frac{cr^{1/4}}{n\log r} = O\big(n^{-3/2}\big).
\end{gather*}
Therefore,
\begin{gather*}
\vert E(r)\vert = O\big(n^{-3/2}\big),
\end{gather*}
and so
\begin{gather}
i\int_0^{2/r} \left( \frac{a}{n\log\big(\frac{2k}{\tilde{r}t}\big)} + O\big(n^{-2}\big)\right)\frac{\dbar t}{t^{3/2}+ t^{1/2}}\nonumber\\
\qquad{}= i \int_{r^{1/2}}^{r^{-1/2}}\frac{a}{n\log\big(\frac{2k}{\tilde{r}t}\big)}\frac{\dbar t}{t^{3/2}+ t^{1/2}} + O\big(n^{-3/2}\big)\nonumber\\
\qquad{} = i \int_{r^{1/2}}^{r^{-1/2}}\frac{a}{n\log\big(\frac{2k}{\tilde{r}}\big) - n\log t}\frac{\dbar t}{t^{3/2}+ t^{1/2}}+ O\big(n^{-3/2}\big) \nonumber\\
\qquad{} = i\int_{r^{1/2}}^{r^{-1/2}}\frac{a}{n\log\big(\frac{2k}{\tilde{r}}\big)} \frac{1}{1 - \frac{\log t}{\log\big(\frac{2k}{\tilde{r}}\big)}}\frac{\dbar t}{t^{3/2}+ t^{1/2}}+ O\big(n^{-3/2}\big),\label{3.22}
\end{gather}
but $t \in [r^{1/2}, r^{-1/2}]$ and so $\vert \log t \vert \le \frac{1}{2}\log \frac{1}{r} \le \frac{1}{2} \log \frac{2k}{r}$. Therefore, (\ref{3.22}) is
\begin{gather*}
\frac{ ia}{n\log\big(\frac{2k}{\tilde{r}}\big)} \int_{r^{1/2}}^{r^{-1/2}}\left(1 + \frac{\log t}{\log\big(\frac{2k}{\tilde{r}}\big)} + O\left(\frac{\log^2 t}{\log^2\big(\frac{2k}{\tilde{r}}\big)}\right)\right)\frac{\dbar t}{t^{3/2}+ t^{1/2}}+ O\big(n^{-3/2}\big).
\end{gather*}
Making the change of variables $\gamma = t^{1/2}$, this becomes
\begin{gather*}
 \frac{2i a}{n\log\big(\frac{2k}{\tilde{r}}\big)} \int_{r^{1/4}}^{r^{-1/4}}\left(1 + \frac{2\log \gamma}{\log\big(\frac{2k}{\tilde{r}}\big)} + O\left(\frac{\log^2 \gamma}{\log^2\big(\frac{2k}{\tilde{r}}\big)}\right)\right)\frac{\dbar \gamma}{\gamma^2 + 1}+ O\big(n^{-3/2}\big)\nonumber\\
\qquad{} = \frac{2i a}{n\log\big(\frac{2k}{\tilde{r}}\big)} \left(\int_{r^{1/4}}^{r^{-1/4}}\frac{\dbar \gamma}{\gamma^2 + 1} + \int_{r^{1/4}}^{r^{-1/4}}\frac{2\log \gamma}{\log\big(\frac{2k}{\tilde{r}}\big)}\frac{\dbar \gamma}{\gamma^2 + 1}\right) + O\left(\frac{1}{n\log^3 n }\right)\nonumber\\
\qquad{} = \frac{2i a}{n\log\big(\frac{2k}{\tilde{r}}\big)} \left(\int_{0}^{\infty}\frac{\dbar \gamma}{\gamma^2 + 1} + \int_{0}^{\infty}\frac{2\log \gamma}{\log\big(\frac{2k}{\tilde{r}}\big)}\frac{\dbar \gamma}{\gamma^2 + 1} \right)+ O\left(\frac{1}{n\log^3 n }\right).
\end{gather*}
So, using the same calculations of the above integrals, see (\ref{N2FirstTerm}) and (\ref{N2SecondTerm}), as in the proof of Proposition \ref{FAsymptoticsZApp},
\begin{gather}
i \int_0^{2/r} \big(\log w(1-rt) - \log w(1-\tilde{r}t)\big) \frac{\dbar t}{t^{3/2}+ t^{1/2}} = \frac{a}{2n\log \frac{2k}{\tilde{r}}} + O\left(\frac{1}{n\log^3 n }\right) . \label{AThing}
\end{gather}
So, combining (\ref{yoyo}), (\ref{35}), (\ref{r-rtildeReduction}), and (\ref{AThing}), we obtain
\begin{gather*}
\log \left(\frac{F^2(1+r)}{F^2(1 +\tilde{r})}\right) = 2 \log \frac{F(1+r)}{F(1 + \tilde{r})}= \frac{a}{n\log\frac{2k}{\tilde{r}}} + O\left(\frac{1}{n\log^3 n}\right),
\end{gather*}
which implies
\begin{gather*}
\frac{F^2(1+r)}{F^2(1 +\tilde{r})} = \exp \left( \frac{a}{n\log\frac{2k}{\tilde{r}}} + O\left(\frac{1}{n\log^3 n}\right)\right) = 1 + \frac{a}{n\log\frac{2k}{\tilde{r}}} + O\left(\frac{1}{n\log^3n}\right).
\end{gather*}
Now,
\begin{gather}
\frac{F^2}{w_\pm}(1 + r) - \frac{F^2}{w_\pm} (1 + \tilde{r}) = \frac{F^2(1 + r) - F^2(1 + \tilde{r})}{w_\pm(1 + r)} + F^2(1 + \tilde{r}) \left(\frac{1}{w_\pm(1 + r)} - \frac{1}{w_\pm(1 + \tilde{r})}\right) \nonumber\\
\hphantom{\frac{F^2}{w_\pm}(1 + r) - \frac{F^2}{w_\pm} (1 + \tilde{r}) }{} = \frac{F^2(1 + \tilde{r})}{w_\pm(1+r) }\left( \frac{F^2(1 + r)}{F^2(1 + \tilde{r})} -1\right) + \frac{F^2(1 + \tilde{r})}{w_\pm(1 + r)} \left(1- \frac{w_\pm(1 + r)}{w_\pm(1 + \tilde{r})}\right) \nonumber\\
\hphantom{\frac{F^2}{w_\pm}(1 + r) - \frac{F^2}{w_\pm} (1 + \tilde{r}) }{} = \frac{F^2(1 + \tilde{r})}{w_\pm(1 + r)} \left( \frac{a}{n\log \frac{2k}{r}} + O\!\left(\frac{1}{n\log^3n}\right) + 1 - \frac{w_\pm(1 +r)}{w_\pm(1 + \tilde{r})} \right)\!,\!\!\!\!\!\!\!\label{StupidWHandling}
\end{gather}
where
\begin{gather*}
 \frac{w_\pm(1 +r)}{w_\pm(1 + \tilde{r})} = \frac{\log \frac{2k}{r} \pm i\pi}{\log \frac{2k}{\tilde{r}} \pm i \pi} = 1 + \frac{\log \frac{2k}{r} - \log \frac{2k}{\tilde{r}}}{\log \frac{2k}{\tilde{r}} \pm i \pi} \nonumber\\
\hphantom{\frac{w_\pm(1 +r)}{w_\pm(1 + \tilde{r})}}{} = 1 + \frac{\log \frac{\tilde{r}}{r}}{\log \frac{2k}{\tilde{r}} \pm i\pi} = 1 + \frac{a}{n \left(\log \frac{2k}{\tilde{r}} \pm i\pi\right)} + O\big(n^{-2}\big),%\label{3.27}
\end{gather*}
so
\begin{gather}
\frac{a}{n\log \frac{2k}{r}} + 1 - \frac{w_\pm(1 +r)}{w_\pm(1 + \tilde{r})} = \frac{a}{n\log\frac{2k}{\tilde{r}}} - \frac{a}{n \big(\log \frac{2k}{\tilde{r}} \pm i\pi\big)} \nonumber\\
\hphantom{\frac{a}{n\log \frac{2k}{r}} + 1 - \frac{w_\pm(1 +r)}{w_\pm(1 + \tilde{r})}}{}
= \frac{\pm i \pi a}{n \log\frac{2k}{\tilde{r}} \big( \log\frac{2k}{\tilde{r}} \pm i\pi\big)} = \pm \frac{i \pi a}{n \log^2 \frac{2k}{\tilde{r}}} + O\left(\frac{1}{n\log^3n}\right).\label{3.28}
\end{gather}
Substituting (\ref{3.28}) into (\ref{StupidWHandling}),
\begin{gather*}
\frac{F^2}{w_\pm}(1 + r) - \frac{F^2}{w_\pm} (1 + \tilde{r}) = \frac{F^2(1 + \tilde{r})}{w_\pm(1+r)} \left(\pm \frac{i \pi a}{n \log^2 \frac{2k}{\tilde{r}}} + O\left(\frac{1}{n\log^3n}\right) \right)\nonumber\\
\hphantom{\frac{F^2}{w_\pm}(1 + r) - \frac{F^2}{w_\pm} (1 + \tilde{r})}{} =\left( \pm \frac{i \pi a}{n \log^2 \frac{2k}{\tilde{r}}} + O\left(\frac{1}{n\log^3n}\right)\right)\left(1 + O\left(\frac{1}{\log n}\right)\right),
\end{gather*}
where we have used Proposition \ref{FAsymptoticsZApp}. So, in particular
\begin{gather*}
\frac{F^2}{w_+}(1 + r) - \frac{F^2}{w_+} (1 + \tilde{r}) + \frac{F^2}{w_-}(1 + r) - \frac{F^2}{w_-} (1 + \tilde{r}) = O\left(\frac{1}{n\log^3n}\right),
\end{gather*}
which concludes the proof of Proposition~\ref{F-FAsymptotics}.
\end{proof}

\section{A summary of Legendre asymptotics}\label{NormsSection}
In this appendix we prove a series of $L^2$ estimates, ultimately allowing us to replace the integral in Proposition \ref{QQProp} with an integral only involving Legendre quantities, as in Proposition~\ref{Replacement}.

\subsection{A summary of Legendre asymptotics}\label{LegendreSummary}
In \cite{Kuijlaars} the authors derive the asymptotics of the solution $S(z) = S^{(n)}(z)$ to the RHP $(\Sigma_S, v_S)$:
\begin{enumerate}\itemsep=0pt
\item[(a)]
\begin{gather}
 S(z) \text{ is analytic for $z \in \mathbb{C}\backslash \Sigma_S$},\label{SRHP1}
 \end{gather}
\item[(b)] $S$ satisfies the following jump relation on $s \in \Sigma_S$
\begin{gather}
S_+(s) = S_-(s)v_S(s),\label{SRHP2}
\end{gather}
\item[(c)] $S(z)$ has the following behavior at infinity
\begin{gather} S(z) = \left(I + O\left(\frac{1}{z}\right)\right), \qquad\text{as $z \to \infty$}, \label{SRHP3}
\end{gather}
\item[(d)] $S(z)$ has the following behavior as $z\to 1$
\begin{gather}
 S(z) = O\left(\begin{matrix}
\log \vert z-1\vert & \log \vert z-1\vert \\
\log\vert z-1\vert & \log\vert z-1\vert
\end{matrix}\right), \label{SRHP4}
\end{gather}
\item[(e)] $S(z)$ has the following behavior as $z \to -1$
\begin{gather}
 S(z) = O\left(\begin{matrix}
\log \vert z+1\vert & \log \vert z+1\vert \\
\log \vert z+1\vert & \log\vert z+1\vert
\end{matrix}\right), \label{SRHP5}
\end{gather}
\end{enumerate}
where $\Sigma_S$ is depicted in Fig.~\ref{SigmaSDef},
\begin{gather*}
v_S(s) = \begin{cases}
\left(\begin{matrix}
1 & 0\\
\phi(s)^{-2n} & 1
\end{matrix}\right) & \text{for $s$ in the upper and lower lips}, \\
\left(\begin{matrix}
0 & 1\\
-1 & 0
\end{matrix}\right) & \text{for $s \in [-1,1]$}.
\end{cases}
\end{gather*}
The asymptotics are derived for $z$ in the regions $A$, $B$, $C$, $D$, and $E$, see Fig.~\ref{RegionsDef}. We will summarize the necessary results in this section.

Recall that $\tilde{Q}$ is the solution to (\ref{TildeQRHP1})--(\ref{TildeQRHP5}) $(\Sigma, \tilde{v}_\Sigma)$. $S$ is the solution to (\ref{SRHP1})--(\ref{SRHP5}) $(\Sigma_S, v_S$). See Figs.~\ref{SigmaDef2} and~\ref{SigmaSDef} for depictions of the contours $\Sigma$ and $\Sigma_S$. The RHP for $S$ differs from the RHP for $\tilde{Q}$ near the point~1, but the problems are related via the deformation:
\begin{gather}
\tilde{Q}(z) = \begin{cases}
S(z) \left(\begin{matrix}
1 & 0\\
\phi^{-2n}(z) & 1
\end{matrix}\right) & \text{for $z \in \Omega'_1$},\\
S(z) \left(\begin{matrix}
1 & 0\\
-\phi^{-2n}(z) & 1
\end{matrix}\right) & \text{for $z \in \Omega'_2$},\\
 S(z) & \text{otherwise},
\end{cases} \label{StoTildeQ}
\end{gather}
where the regions $\Omega_1'$ and $\Omega_2'$ are depicted in Fig.~\ref{OmegaPrime}.
\begin{figure}[t]
\centering
\begin{tikzpicture}[scale=10.]
\fill (1,0) circle (.006cm);
\draw[line width = 1,directed] (1.1,0) to (1, 0);
\draw[line width = 1, directed] (0,0) to (1, 0);
\coordinate (up) at (.5,.2);
\coordinate (dn) at (.5,-.2);
\draw[line width = 1, domain=135:0 , directed] plot ({1 + .1* cos(\x) }, {.1 * sin(\x)});
\draw[line width = 1, domain=225:360 , directed] plot ({1+.1* cos(\x) }, {.1 * sin(\x)});
\draw[line width = 1,directed] (0,0) to [out=45,in=180] (up) to [out=0, in=180-45] (.932, .068);
\draw[line width = 1,directed] (0,0) to [out=-45,in=180] (dn) to [out=0, in=-180+45] (.932,-.068);
\node[below left] at (0,0) {$-1$};
\node[below left] at (1,0) {$1$};
\node at (.2,.2){$\Sigma$};
\node[right] at (1.1, 0){$1+\delta$};
\end{tikzpicture}
\caption{Depiction of $\Sigma$.} \label{SigmaDef2}
\end{figure}
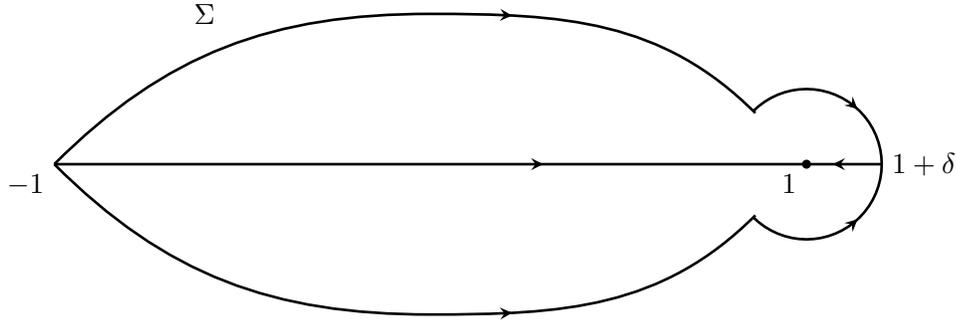
\begin{figure}[t]
\centering
\begin{tikzpicture}[scale=10.]
\draw[line width = 1, directed] (0,0) to (1, 0);
\coordinate (up) at (.5,.2);
\coordinate (dn) at (.5,-.2);
%\draw[line width = 1, domain=135:0 , directed] plot ({1 + .1* cos(\x) }, {.1 * sin(\x)});
\draw[line width = 1, directed] (0,0) to [out=45,in=180] (up) to [out=0, in=180-45] (1, 0);
\draw[line width = 1, directed] (0,0) to [out=-45,in=180] (dn) to [out=0, in=-180+45] (1,0);
\node[below left] at (0,0) {$-1$};
\node[below right] at (1,0) {$1$};
\node at (.2, .2){$\Sigma_S$};
\end{tikzpicture}
\caption{Depiction of $\Sigma_S$.} \label{SigmaSDef}
\end{figure}
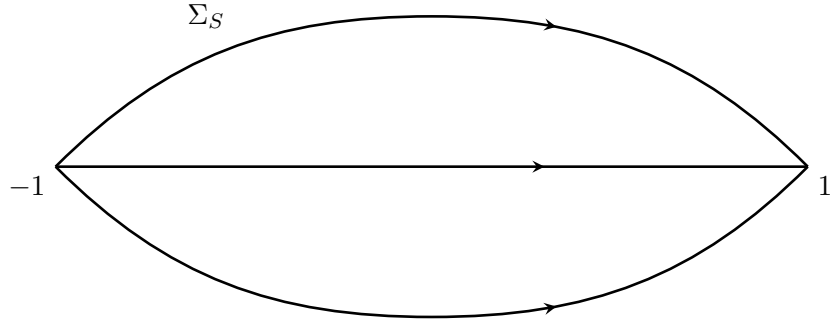
\begin{figure}[t]
\centering
\begin{tikzpicture}[scale=10.]
\draw[line width = 1] (0,0) to (1, 0);
\coordinate (up) at (.5,.2);
\coordinate (dn) at (.5,-.2);
%\draw[line width = 1, domain=135:0 , directed] plot ({1 + .1* cos(\x) }, {.1 * sin(\x)});
\draw[line width = 1, domain=0:360 ] plot ({1+.1* cos(\x) }, {.1 * sin(\x)});
\draw[line width = 1, domain=0:360 ] plot ({.1* cos(\x) }, {.1 * sin(\x)});
\draw[line width = 1] (0,0) to [out=45,in=180] (up) to [out=0, in=180-45] (1, 0);
\draw[line width = 1] (0,0) to [out=-45,in=180] (dn) to [out=0, in=-180+45] (1,0);
\node[below left] at (0,0) {$-1$};
\node[below right] at (1,0) {$1$};
\node [above] at (.2, .2) {$A$};
\node[above left] at (0,0) {$C$};
\node[above right] at (1,0) {$B$};
\node [below right] at (.2, .1){$D$};
\node [above right] at (.2, -.1){$E$};
\node [above right] at (1.1, 0){$U_\delta(1)$};
\node [above left] at (-.1, 0){$U_\delta(-1)$};
\end{tikzpicture}
\caption{Definition of the regions $A$, $B$, $C$, $D$, and $E$.} \label{RegionsDef}
\end{figure}
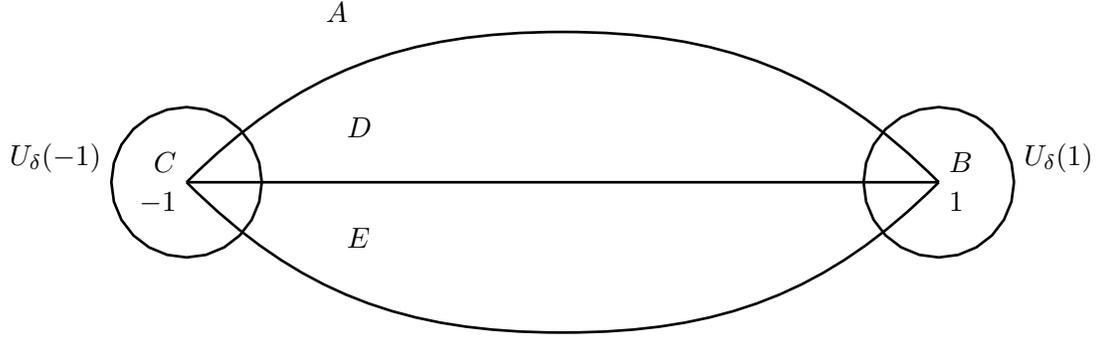

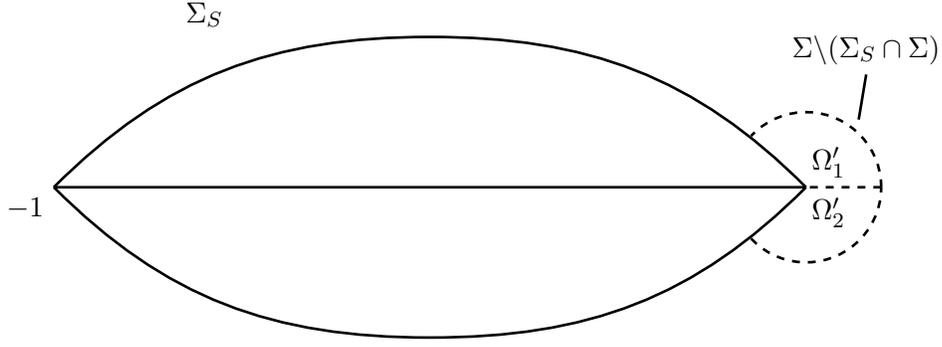
\begin{figure}[t]
\centering
\begin{tikzpicture}[scale=10.]
\draw[line width = 1] (0,0) to (1, 0);
\coordinate (up) at (.5,.2);
\coordinate (dn) at (.5,-.2);
%\draw[line width = 1, domain=135:0 , directed] plot ({1 + .1* cos(\x) }, {.1 * sin(\x)});
\draw[line width = 1, domain=0:137, dashed ] plot ({1+.1* cos(\x) }, {.1 * sin(\x)});
\draw[line width = 1] (1.08, .15) to (1.07, .09);
\draw[line width = 1, domain=223:360, dashed ] plot ({1+.1* cos(\x) }, {.1 * sin(\x)});
\draw[line width = 1, dashed] (1.1, 0) to (1, 0);
\draw[line width = 1] (0,0) to [out=45,in=180] (up) to [out=0, in=180-45] (1, 0);
\draw[line width = 1] (0,0) to [out=-45,in=180] (dn) to [out=0, in=-180+45] (1,0);
\node[below left] at (0,0) {$-1$};
\node [above] at (.2, .2) {$\Sigma_S$};
\node [above] at (1.08, .15) {$\Sigma \backslash (\Sigma_S \cap \Sigma)$};
\node [above] at (1.03, 0) {$\Omega'_1$};
\node [below] at (1.03, 0) {$\Omega'_2$};
\end{tikzpicture}
\caption{Definition of the regions $\Omega_1'$, $\Omega_2'$. \label{OmegaPrime}}
\end{figure}
For $z \in \mathbb{C} \backslash [-1,1]$, we define, as in \cite{Kuijlaars},
\begin{gather*}
N(z) = \left(\begin{matrix}
\dfrac{a+a^{-1}}{2} & \dfrac{a - a^{-1}}{2i} \vspace{1mm}\\
\dfrac{a - a^{-1}}{-2i} & \dfrac{a + a^{-1}}{2}
\end{matrix}\right),
\end{gather*}
where
\begin{gather*}
a := a(z) = \left(\frac{z - 1}{z+1}\right)^{1/4},
\end{gather*}
where $a(z)$ is analytic for $z \in \mathbb{C}\backslash[-1,1]$ and $a(z) >0$ when $z > 1$. Let $U_\delta(s)$ refer to the closed ball of radius $\delta$ centered as $s = \pm 1$. The regions $B$ and $C$ are given by $B = U_\delta(1)\backslash \Sigma_S$ and $C = U_\delta(-1)\backslash \Sigma_S$ respectively. We define
\begin{gather*}
f(z) = \frac{\log^2 \phi(z)}{4}.
\end{gather*}
$f(z)$ is analytic in a neighborhood of $1$ and
\begin{gather}
f(1) = 0,\nonumber\\
f(z) = \frac{1}{2}(z-1) + O\big(\vert z-1\vert^2\big) \qquad \text{as $z \to 1$},\nonumber \\
f(-z) = -\frac{1}{2}(z+1) + O\big(\vert z+1\vert^2\big) \qquad \text{as $z \to -1$}. \label{fFact}
\end{gather}
So, in particular, we can pick $\delta$ sufficiently small so that $f(z)$ is locally conformal for $z \in B$, which clearly also implies $f(-z)$ will be locally conformal in $C$. We define
\begin{gather}
\Psi(\zeta) = \begin{cases}
\left(\begin{matrix}
I_0\big(2\zeta^{1/2}\big) & \frac{i}{\pi} K_0\big(2\zeta^{1/2}\big) \\
2 \pi i \zeta^{1/2} I_0' \big(2\zeta^{1/2}\big) & -2\zeta^{1/2} K_0'\big(2\zeta^{1/2}\big)
\end{matrix}\right) \\
\hspace*{80mm} \text{for $\vert \arg (\zeta) \vert< 2\pi/3$},\\
\left(\begin{matrix}
\frac{1}{2}H_0^{(1)}\big(2(-\zeta)^{1/2}\big) & \frac{1}{2} H_0^{(2)}(2(-\zeta^{1/2})) \\
\pi\zeta^{1/2} \big(H_0^{(1)}\big)' \big(2(-\zeta)^{1/2}\big) & \pi\zeta^{1/2} \big(H_0^{(2)}\big)'\big(2(-\zeta)^{1/2}\big)
\end{matrix}\right) \\
\hspace*{80mm} \text{for $ 2\pi/3 < \arg \zeta < \pi$},\\
\left(\begin{matrix}
\frac{1}{2}H_0^{(2)}\big(2(-\zeta)^{1/2}\big) &- \frac{1}{2} H_0^{(1)}(2(-\zeta^{1/2})) \\
-\pi\zeta^{1/2} \big(H_0^{(2)}\big)' \big(2(-\zeta)^{1/2}\big) & \pi\zeta^{1/2} \big(H_0^{(1)}\big)'\big(2(-\zeta)^{1/2}\big)
\end{matrix}\right) \\
\hspace*{80mm} \text{for $ -\pi < \arg \zeta < -2\pi/3$},
\end{cases} \!\!\!\!\!\label{PsiDef}
\end{gather}
for $\zeta \in \mathbb{C}\backslash \Gamma$, and where $\Gamma$ is defined in Fig.~\ref{GammaDef}. The contour $\Sigma_S$ is chosen within $U_\delta(1)$, (resp.~$-1$) to be the preimage of $\Gamma$ under the map $n^2f(z)$ (resp.~$n^2f(-z)$). Therefore, around the point~$-1$, $\Sigma$ is the preimage of $\Gamma$ under the map $n^2f(-z)$ as well. In other words
\begin{gather}
n^2 f(- (\Sigma \cap U_{\delta}(-1))) \subset \Gamma \qquad \text{for all $n$}. \label{SigmaToGamma}
\end{gather}

In (\ref{PsiDef}), $I_0$, $K_0$ refer to the modified Bessel functions of the first and second kinds with parameter $0$, and $H^{(1)}_0$, $H^{(2)}_0$ refer to the Hankel functions of the first and second kinds, see, for instance, \cite{Stegun}. As before, the square root in the above definition is given by its principle value.

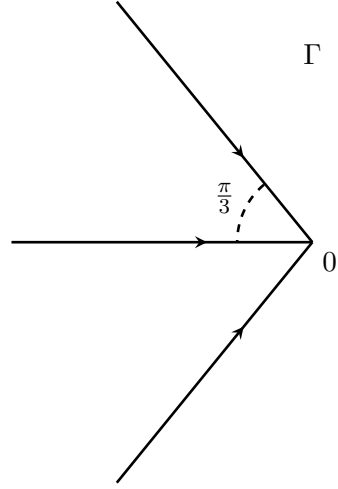
\begin{figure}[t]
\centering
\begin{tikzpicture}[scale=5.]
\draw[line width = 1, directed] (.2,0) to (1, 0);
\draw[line width = 1, directed] (.48, .64) to (1, 0);
\draw[line width = 1, directed] (.48,-.64) to (1, 0);
\draw[line width = 1, domain = 130:180, dashed] plot ({1+.2* cos(\x) }, {.2 * sin(\x)});
\node [below right] at (1,0) {$0$};
\node at (1,.5){$\Gamma$};
\node[above left] at (.82, .05){$\frac{\pi}{3}$};
\end{tikzpicture}
\caption{Definition of $\Gamma$.} \label{GammaDef}
\end{figure}
For any fixed $\delta$ sufficiently small, the following holds uniformly,
\begin{gather}
S^{(n)}(z) = \begin{cases}
R^{(n)}(z)N(z) & \text{for $z \in A \cup D \cup E$}, \\
R^{(n)}(z)E(z) (2 \pi n)^{\sigma_3/2} \Psi\big(n^2f(z)\big) \phi^{-n\sigma_3}(z)& \text{for $z \in B$},\\
\sigma_3 R^{(n)}(-z)E(-z) (2 \pi n)^{\sigma_3/2} \Psi\big(n^2f(-z)\big) \phi^{-n\sigma_3}(-z) \sigma_3 & \text{for $z \in C$},
\end{cases}\hspace{-10mm} \label{SAsymptotics1}
\end{gather}
where $R^{(n)}(z)$ is bounded, analytic in $\mathbb{C} \backslash \Sigma_R$, with $\Sigma_R$ depicted in Fig.~\ref{SigmaRDef}, and
\begin{gather*}
R^{(n)}(z) = I + \frac{R_1(z)}{n} + O\left(\frac{1}{n^2}\right) \qquad \text{uniformly in $z \in \mathbb{C}\backslash \Sigma_R$},
\end{gather*}
where $R_1(z)$ is an $n$-independent, bounded, analytic function for $z \in \mathbb{C} \backslash \Sigma_R$. Additionally
\begin{gather}
\left \vert R^{(n)}(z) - I - \frac{R_1(z)}{n} \right\vert \le \frac{1}{n^2 \vert z \vert}.\label{ThingyDelta}
\end{gather}
\begin{figure}[t]\centering
\begin{tikzpicture}[scale=10.]
\coordinate (up) at (.5,.2);
\coordinate (dn) at (.5,-.2);
%\draw[line width = 1, domain=135:0 , directed] plot ({1 + .1* cos(\x) }, {.1 * sin(\x)});
\draw[line width = 1, domain=0:360, directed ] plot ({1+.1* cos(\x) }, {.1 * sin(\x)});
\draw[line width = 1, domain=0:360, directed ] plot ({.1* cos(\x) }, {.1 * sin(\x)});
\draw[line width = 1, domain=223:360, dashed ] plot ({1+.1* cos(\x) }, {.1 * sin(\x)});
\draw[line width = 1, directed] (.08,.06) to [out=45,in=180] (up) to [out=0, in=180-45] (.92, .06);
\draw[line width = 1, directed] (.08,-.06) to [out=-45,in=180] (dn) to [out=0, in=-180+45] (.92,-.06);
\node[below] at (0,0) {$-1$};
\node[below] at (1,0) {$1$};
\node[circle,fill,inner sep=1pt] at (0,0){};
\node[circle,fill,inner sep=1pt] at (1,0){};
\node [above] at (.2, .2) {$\Sigma_R$};
\end{tikzpicture}
\caption{Depiction of $\Sigma_R$.} \label{SigmaRDef}
\end{figure}
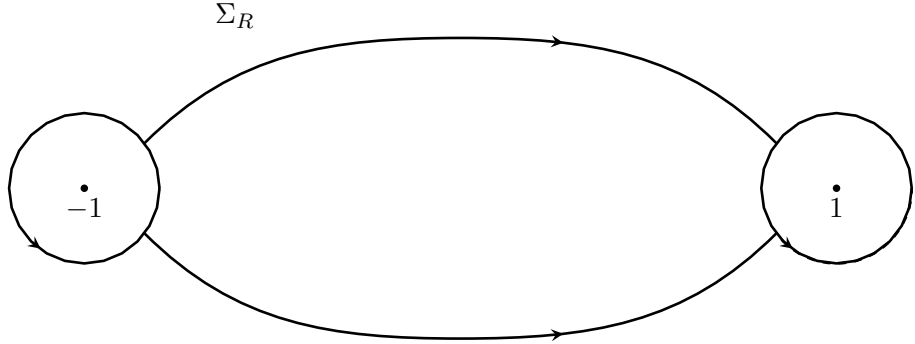
Finally, $E(z)$ is defined by (see Remark~\ref{NotationalDifferences})
\begin{gather*}
E(z) = N(z) \frac{1}{\sqrt{2}}\left(\begin{matrix}
1 & -i\\
-i & 1
\end{matrix}\right) f(z)^{\sigma_3/4}.
\end{gather*}
$E(z)$ is analytic and bounded for $z \in B$ and has determinant 1. It follows, in particular, that $E^{-1}(z)$ is also analytic and bounded for $z\in B$.
\begin{Remark}\label{NotationalDifferences}
We use slightly different notation than in \cite{Kuijlaars} in a few ways,
\begin{itemize}\itemsep=0pt
\item We refer to $R(z)$ as $R^{(n)}(z)$ to emphasize its dependence on $n$.
\item The neighborhoods $U_\delta$ and $\tilde{U}_\delta$ we refer to as $U_\delta(1)$ and $U_\delta(-1)$ respectively.
\item The matching factor $E_n(z)$ in \cite{Kuijlaars} we write $E_n(z) = E(z) (2\pi n)^{\sigma_3/2}$, where $E(z)$ is clearly independent of $n$.
\end{itemize}
\end{Remark}

In estimating the integral
\begin{gather*}
- \int_\Sigma \mu(s)(v_\Sigma(s) - \tilde{v}_\Sigma(s)) \tilde{\mu}^{-1}(s) \dbar s
\end{gather*}
we need asymptotic information about the function $\tilde{\mu}$. However, since $v_\Sigma(s) - \tilde{v}_\Sigma(s)$ is identically~0 for $s \in (-1,1)$ and is strictly lower triangular otherwise, we actually only need information about $\tilde{\mu}$'s second column. To that end, let $\tilde{\mu}_2, \tilde{Q}_2$, and $S_2$ refer to the second columns of the matrices $\tilde{\mu}$, $\tilde{Q}$, and~$S$.

It is clear from the RHP for $\tilde{Q}$, (\ref{TildeQRHP1})--(\ref{TildeQRHP5}), that $\tilde{Q}_2$ is actually analytic in $\mathbb{C}\backslash [-1,1]$, and it is clear from the deformation~(\ref{StoTildeQ}) that $\tilde{Q}_2 = S_2$. Similarly, $S_2$ is analytic in $\mathbb{C}\backslash[-1,1]$. Therefore, for $z \in \Sigma \backslash [-1,1]$,
\begin{gather*}
\tilde{\mu}_2(z) = \big(\tilde{Q}_2(z)\big)_- = \tilde{Q}_2(z) = S_2(z).
\end{gather*}
For $z \in (\Sigma \cap \Sigma_2) \backslash [-1,1]$, $\tilde{\mu}_2(z) = S_+(z) = S_-(z)$ and hence, in evaluating $\tilde{\mu}_2(z)$ we can, when convenient, use either applicable expression from (\ref{PsiDef}) or (\ref{SAsymptotics1}).

We will always pick the ``outside'' interpretation for $S_2(z)$ in such cases. When $z \in \Sigma \backslash (\Sigma_S \cap \Sigma)$ (see again Fig.~\ref{OmegaPrime}), we are outside of the contour $\Sigma_S$. It follows from (\ref{SAsymptotics1}), uniformly as $n\to \infty$ for $z \in \Sigma \backslash[-1,1]$,
\begin{gather}
\tilde{\mu}_2(z) = \begin{cases}R^{(n)}(z)\left(\begin{matrix}
\dfrac{a-a^{-1}}{2i} \\
\dfrac{a + a^{-1}}{2}
\end{matrix}\right) \\ \hspace*{50mm} \text{for $z \in \Sigma \backslash ([-1,1] \cup U_\delta(1) \cup U_\delta(-1))$}, \\
R^{(n)}(z)E(z) (2 \pi n)^{\sigma_3/2} \Psi_2\big(n^2f(z)\big) \phi^{n}(z) \\ \hspace*{50mm} \text{for $z \in (\Sigma\backslash [-1,1]) \cap U_\delta(1)$}, \\
-\sigma_3 R^{(n)}(-z)E(-z) (2 \pi n)^{\sigma_3/2} \Psi_2\big(n^2f(-z)\big) \phi^{n}(-z) \\ \hspace*{50mm} \text{for $z \in (\Sigma\backslash [-1,1]) \cap U_\delta(-1)$},
\end{cases}\label{TildeMu2Def}
\end{gather}
where
\begin{gather*}
\Psi_2(\zeta) = \left(\begin{matrix}
\frac{i}{\pi} K_0\big(2\zeta^{1/2}\big) \\
 -2\zeta^{1/2} K_0'\big(2\zeta^{1/2}\big)
\end{matrix}\right).
\end{gather*}
Note, from the properties of the modified Bessel function $K_0$, that $\Psi_2(\zeta)$ has an analytic extension to $\mathbb{C}\backslash (-\infty, 0]$. From (\ref{SAsymptotics1}) and (\ref{TildeMu2Def}), we have
\begin{gather}
 \tilde{\mu}(z) = \begin{cases}
R^{(n)}(z)E(z) (2 \pi n)^{\sigma_3/2} \Psi\big(n^2f(z)\big) \phi^{-n\sigma_3}(z) \\
\hspace*{50mm} \text{for $(\Sigma\backslash [-1,1]) \cap U_\delta(1)$}, \\
-\sigma_3 R^{(n)}(-z)E(-z) (2 \pi n)^{\sigma_3/2} \Psi\big(n^2f(-z)\big) \phi^{-n\sigma_3}(-z)\sigma_3 \\
\hspace*{50mm} \text{for $z \in (\Sigma\backslash [-1,1]) \cap U_\delta(-1)$},
\end{cases} \label{TildeMuExpansion}
\end{gather}
where
\begin{gather*}
\Psi = \left(\begin{matrix}
\Psi_{11} & \Psi_{12}\\
\Psi_{21} & \Psi_{22}
\end{matrix}\right),\qquad
\Psi_{12}(\zeta) = \frac{i}{\pi} K_0\big(2\zeta^{1/2}\big), \qquad \Psi_{22}(\zeta) = -2\zeta^{1/2} K_0'\big(2\zeta^{1/2}\big).
\end{gather*}
We will repeatedly make use of the following asymptotic estimates for the functions $\Psi_{12}$, $\Psi_{22}$:
\begin{gather}
\Psi_{12}(\zeta), \Psi_{22}(\zeta),\Psi_{12}'(\zeta), \Psi_{22}'(\zeta)= O\big(\zeta^{1/4} e^{-2\zeta^{1/2}}\big) \text{ as $\zeta\to \infty$, uniformly for $\vert \arg \zeta \vert < \pi$},\nonumber\\
\Psi_{12}(\zeta) = O\left(\log \zeta\right) \text{ as $\zeta \to 0$}, \qquad \Psi_{22}(\zeta) = O(1) \text{ as $\zeta \to 0$}, \nonumber\\
\Psi_{12}'(\zeta) = O\left(\frac{1}{\zeta}\right) \text{ as $\zeta \to 0$}, \qquad \Psi_{22}'(\zeta) = O\left(\frac{1}{\zeta}\right) \text{ as $\zeta \to 0$}.\label{PsiAsymptotics}
\end{gather}
These estimates follow from \cite{Stegun} formulas (9.6.11), (9.6.13), (9.6.27), (9.7.2), (9.7.4), and the fact that $K_0$ is a solution to the modified Bessel equation
\begin{gather*}
x^2K_0''(x) + xK_0'(x) - x^2 = 0.
\end{gather*}
For clarity, we depict the contours in which $\tilde{\mu}$ is given by different expansions, see (\ref{TildeMu2Def}), in Figs.~\ref{SigmaFar},~\ref{SigmaNear1}, and~\ref{SigmaNear-1}. Note that $(\Sigma\backslash [-1,1]) \cap U_\delta(1) = [1+\delta, 1]$. In all instances where we integrate over the contour $\Sigma$, the integral over $[-1,1]$ will not contribute.

It will be routine in the calculations that follow to break up an integral over $\Sigma$ into an integral over $ (\Sigma\backslash [-1,1]) \cap U_{1/n}(1)$, an integral over $(\Sigma\backslash [-1,1]) \cap U_{1/n}(-1)$, and an integral over $ \Sigma \backslash ([-1,1] \cup U_{1/n}(1) \cup U_{1/n}(-1))$, where $U_{1/n}(s)$ denotes a ball of radius $1/n$ centered at $s = \pm 1$. This decomposition is shown in Fig.~\ref{SigmaDecomposition}.

\begin{figure}[t]
\centering
\begin{tikzpicture}[scale=10.]
\fill (1,0) circle (.006cm);
\draw[line width = 1, dashed] (1.1,0) to (1, 0);
\draw[line width = 1, dashed] (0,0) to (.068, .068);
\draw[line width = 1, dashed] (0,0) to (.068, -.068);
\coordinate (up) at (.5,.2);
\coordinate (dn) at (.5,-.2);
\draw[line width = 1, domain=135:0 , directed] plot ({1 + .1* cos(\x) }, {.1 * sin(\x)});
\draw[line width = 1, domain=225:360 , directed] plot ({1+.1* cos(\x) }, {.1 * sin(\x)});
\draw[line width = 1,directed] (.068, .068) to [out=45,in=180] (up) to [out=0, in=180-45] (.932, .068);
\draw[line width = 1,directed] (.068, -.068) to [out=-45,in=180] (dn) to [out=0, in=-180+45] (.932,-.068);
\node[below left] at (0,0) {$-1$};
\node[below left] at (1,0) {$1$};
\node[right] at (1.1, 0){$1+\delta$};
\end{tikzpicture}
\caption{Depiction of $\Sigma \backslash ([-1,1] \cup U_\delta(1) \cup U_\delta(-1))$.\label{SigmaFar}}
\text{ }\\ \text{ }
\centering
\begin{tikzpicture}[scale=10.]
\fill (1,0) circle (.006cm);
\draw[line width = 1, directed] (1.1,0) to (1, 0);
\draw[line width = 1, dashed] (0,0) to (.068, .068);
\draw[line width = 1, dashed] (0,0) to (.068, -.068);
\coordinate (up) at (.5,.2);
\coordinate (dn) at (.5,-.2);
\draw[line width = 1, domain=135:0 , dashed] plot ({1 + .1* cos(\x) }, {.1 * sin(\x)});
\draw[line width = 1, domain=225:360 , dashed] plot ({1+.1* cos(\x) }, {.1 * sin(\x)});
\draw[line width = 1,dashed] (.068, .068) to [out=45,in=180] (up) to [out=0, in=180-45] (.932, .068);
\draw[line width = 1,dashed] (.068, -.068) to [out=-45,in=180] (dn) to [out=0, in=-180+45] (.932,-.068);
\node[below left] at (0,0) {$-1$};
\node[below left] at (1,0) {$1$};
\node[right] at (1.1, 0){$1+\delta$};
\end{tikzpicture}
\caption{Depiction of $(\Sigma\backslash [-1,1]) \cap U_\delta(1)$.\label{SigmaNear1}}
\text{ }\\ \text{ }
\centering
\begin{tikzpicture}[scale=10.]
\fill (1,0) circle (.006cm);
\draw[line width = 1, dashed] (1.1,0) to (1, 0);
\draw[line width = 1, directed] (0,0) to (.068, .068);
\draw[line width = 1, directed] (0,0) to (.068, -.068);
\coordinate (up) at (.5,.2);
\coordinate (dn) at (.5,-.2);
\draw[line width = 1, domain=135:0 , dashed] plot ({1 + .1* cos(\x) }, {.1 * sin(\x)});
\draw[line width = 1, domain=225:360 , dashed] plot ({1+.1* cos(\x) }, {.1 * sin(\x)});
\draw[line width = 1,dashed] (.068, .068) to [out=45,in=180] (up) to [out=0, in=180-45] (.932, .068);
\draw[line width = 1,dashed] (.068, -.068) to [out=-45,in=180] (dn) to [out=0, in=-180+45] (.932,-.068);
\node[below left] at (0,0) {$-1$};
\node[below left] at (1,0) {$1$};
\node[right] at (1.1, 0){$1+\delta$};
\end{tikzpicture}
\caption{Depiction of $(\Sigma\backslash [-1,1]) \cap U_\delta(-1)$.\label{SigmaNear-1}}
\end{figure}
\begin{figure}[t]
\centering
\begin{tikzpicture}[scale=10.]
\fill (1,0) circle (.006cm);
\draw[line width = 1, dashed] (1.05,0) to (1, 0);
\draw[line width = 1, directed] (1.1, 0) to (1.05, 0);
\draw[line width = 1, dashed] (0,0) to (.034, .034);
\draw[line width = 1, dashed] (0,0) to (.034, -.034);
\coordinate (up) at (.5,.2);
\coordinate (dn) at (.5,-.2);
\draw[line width = 1, domain=135:0 , directed] plot ({1 + .1* cos(\x) }, {.1 * sin(\x)});
\draw[line width = 1, domain=225:360 , directed] plot ({1+.1* cos(\x) }, {.1 * sin(\x)});
\draw[line width = 1,directed] (.034, .034) to [out=45,in=180] (up) to [out=0, in=180-45] (.932, .068);
\draw[line width = 1,directed] (.034, -.034) to [out=-45,in=180] (dn) to [out=0, in=-180+45] (.932,-.068);
\node[right] at (1.1, 0){$1+\delta$};
\draw[line width = 1, domain=360:0 , loosely dotted] plot ({1 + .05* cos(\x) }, {.05 * sin(\x)});
\draw[line width = 1, domain=360:0 , loosely dotted] plot ({ .05* cos(\x) }, {.05 * sin(\x)});
\node[left] at (-.05, 0){$U_{1/n}(-1)$};
\node[left] at (.95, 0){$U_{1/n}(1)$};
\end{tikzpicture}
\caption{Decomposition of $\Sigma\backslash[-1,1]$ into $\Sigma \backslash ([-1,1] \cup U_{1/n}(1) \cup U_{1/n}(-1))$ and $(\Sigma\backslash [-1,1]) \cap U_{1/n}(1)$ and $(\Sigma\backslash [-1,1]) \cap U_{1/n}(-1)$.}\label{SigmaDecomposition}
\end{figure}

\subsection{Some important norms}
\begin{prop}\label{NormsApp} The functions $\mu$, $\tilde{\mu}$ obey the following estimates:
\begin{gather}
\big\vert \big\vert \tilde{\mu}^{(n)}\big( v_{\Sigma}^{(n)} - \tilde{v}_\Sigma^{(n)} \big) \big\vert \big\vert_{L^2(\Sigma)} = O\left(\frac{1}{n^{1/2}\log^2 n }\right),\label{Mu-MuL2PreApp}
\\
\big\vert \big\vert \mu^{(n)} - \tilde{\mu}^{(n)} \big\vert \big\vert_{L^2(\Sigma)} = O\left(\frac{1}{n^{1/2}\log^2 n }\right),\label{Mu-MuL2App}
\\
\big\vert \big\vert \tilde{\mu}^{(n)}\big( v_{\Sigma}^{(n)} - \tilde{v}_\Sigma^{(n)} \big) - \tilde{\mu}^{(n+1)}\big( v_{\Sigma}^{(n+1)} - \tilde{v}_\Sigma^{(n+1)} \big) \big\vert \big\vert_{L^2(\Sigma)} = O\left(\frac{1}{n^{3/2}\log^2 n }\right),\label{Mu-MuL2HardPreApp}
\\
\big\vert \big\vert \big(\mu^{(n)} - \tilde{\mu}^{(n)}\big) - \big(\mu^{(n+1)} - \tilde{\mu}^{(n+1)}\big) \big\vert \big\vert_{L^2(\Sigma)} = O\left(\frac{1}{n^{3/2}\log^2 n }\right).\label{Mu-MuL2HardApp}
\end{gather}
\end{prop}

\begin{proof}
Recall that $\delta > 0$ is fixed as $n \to \infty$ and that $v_\Sigma(s) - \tilde{v}_\Sigma(s) \equiv 0$ for $s\in [-1,1]$ and $U_{1/n}(s)$ refers to the ball of radius $1/n$ centered at $s = \pm 1$. We have that
\begin{gather}
\big\vert \big\vert \tilde{\mu}\left( v_{\Sigma} - \tilde{v}_\Sigma \right) \big\vert \big\vert_{L^2(\Sigma)}^2 = \int_{\Sigma\backslash[-1,1]} \left\vert \tilde{\mu}\left( v_{\Sigma} - \tilde{v}_\Sigma \right)\right\vert^2 \vert {\rm d}s\vert \nonumber\\
=\left( \int_{(\Sigma\backslash[-1,1]) \cap U_{1/n}(1)}\!+ \int_{(\Sigma\backslash[-1,1])\cap U_{1/n}(-1) }\! + \int_{\Sigma \backslash ([-1,1]\cup U_{1/n}(-1) \cup U_{1/n}(1))} \right) \! \left\vert\tilde{\mu}\left( v_{\Sigma} - \tilde{v}_\Sigma \right)\right\vert^2 \vert {\rm d}s\vert \nonumber\\
= I_1 + I_2 + I_3.\label{IBreakdown}
\end{gather}
First we will bound
\begin{gather*}
I_3 = \int_{\Sigma \backslash ([-1,1]\cup U_{1/n}(-1) \cup U_{1/n}(1))} \left\vert\tilde{\mu}\left( v_{\Sigma} - \tilde{v}_\Sigma \right)\right\vert^2 \vert {\rm d}s\vert,
\end{gather*}
which will be done by first proving $L^\infty$ bounds on $\tilde{\mu}_2$ and $\phi^{-2n}$. From (\ref{TildeMu2Def}),
\begin{gather*}
\sup_{z \in [1+ 1/n,1+\delta]}\left \vert \tilde{\mu}_2(z) \right\vert = \sup_{\frac{1}{n} \le z -1 \le \delta}\big\vert R^{(n)}(z) E(z) (2\pi n)^{\sigma_3/2} \Psi_2(n^2 f(z))\phi^n(z) \big\vert \nonumber\\
\le c n^{1/2}\sup_{\frac{1}{n} \le z -1 \le \delta}\big\vert\Psi_2\big(n^2 f(z)\big)\phi^n(z)\big\vert \nonumber\\
\hphantom{\sup_{z \in [1+ 1/n,1+\delta]}\left \vert \tilde{\mu}_2(z) \right\vert}{} = c n^{1/2}\sup_{\frac{1}{n} \le z-1 \le \delta}\big\vert\Psi_2\big(n^2 f(z)\big)e^{2 (n^2 f(z))^{1/2}} \big\vert\nonumber\\
\hphantom{\sup_{z \in [1+ 1/n,1+\delta]}\left \vert \tilde{\mu}_2(z) \right\vert}{}= c n^{1/2}\sup_{O(n) \le y \le O(n^2)} \left\vert\begin{matrix}
\Psi_{12}(y)e^{2y^{1/2}} \\
\Psi_{22}(y)e^{2 y^{1/2}}
\end{matrix}\right\vert,
\end{gather*}
where $y = n^2f(z)$. Substituting in the asymptotics of $\Psi$ as in $(\ref{PsiAsymptotics})$, we see that
\begin{gather}
\sup_{z \in [1+1/n, 1+\delta]}\left \vert \tilde{\mu}_2(z) \right\vert = O(n)\label{MuInfty1}
\end{gather}
Similarly
\begin{gather*}
\sup_{z \in (\Sigma \backslash [-1,1]) \cap (U_\delta(-1) \backslash U_{1/n}(-1))}\left \vert \tilde{\mu}_2(z) \right\vert \nonumber\\
 \qquad{} \le c n^{1/2}\sup_{z \in (\Sigma \backslash [-1,1]) \cap (U_\delta(-1) \backslash U_{1/n}(-1))}\big\vert\Psi_2(n^2 f(z))e^{2 (n^2 f(z))^{1/2}} \big\vert\nonumber\\
 \qquad{} = c n^{1/2}\sup_{\substack{
 O(n) \le \vert y\vert \le O(n^2) \\
 \arg y = \pm 2\pi /3
 }}\left\vert\begin{matrix}
\Psi_{12}(y)e^{2y^{1/2}} \\
\Psi_{22}(y)e^{2 y^{1/2}}
\end{matrix}\right\vert,
\end{gather*}
where in the last step we have used the fact that $n^2f(-\Sigma \cap U_\delta(-1)) \subset \Gamma$ for any choice of $n$, see~(\ref{SigmaToGamma}). Once again, substituting in the asymptotics for $\Psi$ from~(\ref{PsiAsymptotics}),
\begin{gather}
\sup_{z \in (\Sigma \cap U_\delta(-1)) \backslash U_{1/n}(-1)}\left \vert \tilde{\mu}_2(z) \right\vert = O(n).\label{MuInfty2}
\end{gather}
 Lastly, it is clear from (\ref{TildeMu2Def}) that $\tilde{\mu}_2$ is uniformly bounded for $z \in \Sigma \backslash ([-1,1] \cup U_\delta(-1) \cup U_\delta(1))$. Combining this with (\ref{MuInfty1}) and (\ref{MuInfty2}), if we are at least a distance $1/n$ from $\pm1$ on $\Sigma\backslash [-1,1]$, then $\vert \tilde \mu \vert \le cn$, that is,
\begin{gather}
\left\Vert \tilde{\mu}_2 \right\Vert_{L^\infty (\Sigma \backslash ([-1,1]\cup U_{1/n}(-1) \cup U_{1/n}(1))} = O(n).\label{MuInfty}
\end{gather}

 Similarly, we prove an $L^\infty$ bound on the function $\phi(z)$: A straightforward calculation using properties \eqref{1branch} and~\eqref{2branch} of Proposition \ref{PhiProps} shows that
\begin{gather}
\inf_{z \in (\Sigma \backslash [-1,1]) \cap (U_\delta(-1) \backslash U_{1/n}(-1))} \vert \phi(z)\vert = 1 + \frac{c}{\sqrt{n}}, \qquad \text{ for some $c > 0$}\label{PhiBound1}
\end{gather}
and, since the contour $\Sigma \backslash ([-1,1] \cup U_\delta(-1) \cup U_\delta(1))$ is bounded away from the interval $[-1,1]$, property~\eqref{prop5} of Proposition \ref{PhiProps} implies
\begin{gather}
\inf_{z \in \Sigma \backslash ([-1,1] \cup U_\delta(-1) \cup U_\delta(1))} \vert \phi(z)\vert = 1 + c, \qquad \text{for some fixed $c > 0$}.\label{PhiBound2}
\end{gather}
Equations (\ref{PhiBound1}) and (\ref{PhiBound2}) together imply
\begin{gather}
\big\Vert \phi^{-1}(z) \big\Vert_{L^\infty(\Sigma \backslash ([-1,1] \cup U_{1/n}(1) \cup U_{1/n}(-1)))} \le 1 - \frac{c}{\sqrt{n}}\qquad \text{ for some $c > 0$},\nonumber\\
\big\Vert \phi^{-2n}(z) \big\Vert_{L^\infty(\Sigma \backslash ([-1,1] \cup U_{1/n}(1) \cup U_{1/n}(-1)))} \le c_1 e^{-c_2 n^{1/2}}.\label{PhiBound}
\end{gather}
So, using the fact that $\frac{F^2}{w}$ is bounded, (\ref{MuInfty}), and (\ref{PhiBound}),
\begin{gather*}
 \left\vert\tilde{\mu}\left( v_{\Sigma} - \tilde{v}_\Sigma \right)\right\vert^2 = \left\vert \left(\mathbf{0}, \tilde{\mu}_2\right)(v_\Sigma - \tilde{v}_\Sigma)\right\vert^2= O\big(e^{-cn^{1/2}}\big)
\end{gather*}
uniformly for $s \in \Sigma \backslash ([-1,1] \cup U_{1/n}(-1) \cup U_{1/n}(1))$ and so
\begin{gather}
I_3 = O\big(e^{-cn^{1/2}}\big). \label{I3}
\end{gather}
Turning to $I_1$ and using (\ref{TildeMuExpansion}),
\begin{gather*}
I_1 = \int_{U_{1/n}(1) \cap \Sigma} \left\vert\tilde{\mu}\left( v_{\Sigma} - \tilde{v}_\Sigma \right)\right\vert^2 \vert {\rm d}s\vert = \int_1^{1+\frac{1}{n}}\left\vert\tilde{\mu}\left( v_{\Sigma} - \tilde{v}_\Sigma \right)\right\vert^2 {\rm d}s \nonumber\\
 = \int_1^{1+\frac{1}{n}}\!\left\vert R^{(n)}(s) E(s) (2\pi n)^{\sigma_3/2}\Psi\big(n^2f(s)\big) \phi^{-n\sigma_3}(s) \!\left( \!\begin{matrix}
0 & 0\\
\left(\frac{F^2}{w_+}(s) + \frac{F^2}{w_-}(s) - 2\right) \phi^{-2n}\! & 0
\end{matrix}\right)\right\vert^2\! {\rm d}s \nonumber\\
 = \int_1^{1+\frac{1}{n}}\left\vert \left(I + O\left(\frac{1}{n}\right)\right) E(s) (2\pi n)^{\sigma_3/2}\Psi\big(n^2f(s)\big)\left( \begin{matrix}
0 & 0\\
1 & 0
\end{matrix}\right)\right\vert^2\\
\quad{}\times \left\vert \left(\frac{F^2}{w_+}(s) + \frac{F^2}{w_-}(s) - 2\right)\phi^{-n}(s)\right\vert^2 {\rm d}s \nonumber\\
\le cn\int_1^{1+\frac{1}{n}}\left\vert \left( \begin{matrix}
\Psi_{12}\big(n^2f(s)\big) & 0\\
\Psi_{22}\big(n^2f(s)\big) & 0
\end{matrix}\right)\right\vert^2 \left\vert \frac{F^2}{w_+}(s) + \frac{F^2}{w_-}(s) - 2\right\vert^2 {\rm d}s.
\end{gather*}
By Proposition \ref{FCancellation},
\begin{gather*}
\frac{F^2}{w_+}(s) + \frac{F^2}{w_-}(s) - 2 = O\left(\frac{1}{\log^2\vert s-1\vert}\right) \qquad \text{as $s \to 1$},
\end{gather*}
so for $\vert s-1 \vert \le \frac{1}{n}$,
\begin{gather}
\frac{F^2}{w_+}(s) + \frac{F^2}{w_-}(s) - 2 = O\left(\frac{1}{\log^2n}\right). \label{FTonRight}
\end{gather}
So
\begin{gather*}
cn\int_1^{1+\frac{1}{n}}\left\vert \left( \begin{matrix}
\Psi_{12}\big(n^2f(s)\big) & 0\\
\Psi_{22}\big(n^2f(s)\big) & 0
\end{matrix}\right)\right\vert^2 \left\vert \frac{F^2}{w_+}(s) + \frac{F^2}{w_-}(s) - 2\right\vert^2 {\rm d}s \nonumber\\
\qquad {} \le c \frac{n}{\log^4 n}\int_1^{1+\frac{1}{n}}\left\vert \left( \begin{matrix}
\Psi_{12}\big(n^2f(s)\big) & 0\\
\Psi_{22}\big(n^2f(s)\big) & 0
\end{matrix}\right)\right\vert^2 {\rm d}s.
\end{gather*}
Making the substitution $y =n^2f(s)$,
\begin{gather*}
 c \frac{n}{\log^4 n}\int_1^{1+\frac{1}{n}}\left\vert \left( \begin{matrix}
\Psi_{12}\big(n^2f(s)\big) & 0\\
\Psi_{22}\big(n^2f(s)\big) & 0
\end{matrix}\right)\right\vert^2 {\rm d}s \nonumber\\
\qquad{} = c \frac{n}{\log^4 n}\int_0^{n^2f(1+\frac{1}{n})}\left\vert \left( \begin{matrix}
\Psi_{12}(y) & 0\\
\Psi_{22}(y) & 0
\end{matrix}\right)\right\vert^2 {\rm d}f^{-1} \left(\frac{y}{n^2}\right).
\end{gather*}
Using (\ref{fFact}), we see that $n^2f(1 + \frac{1}{n}) = O(n)$ and ${\rm d}f^{-1}\big(\frac{y}{n^2}\big) = 2\frac{{\rm d}y}{n^2}\big(1 + O\big(\frac{1}{n}\big)\big)$. Therefore,
\begin{gather*}
c \frac{n}{\log^4 n}\int_0^{n^2f(1+\frac{1}{n})}\left\vert \left( \begin{matrix}
\Psi_{12}(y) & 0\\
\Psi_{22}(y) & 0
\end{matrix}\right)\right\vert^2 {\rm d}f^{-1} \left(\frac{y}{n^2}\right) \nonumber\\
\qquad{} \le \frac{c}{n\log^4 n} \int_0^{n^2f(1+\frac{1}{n})}\left\vert \left( \begin{matrix}
\Psi_{12}(y) & 0\\
\Psi_{22}(y) & 0
\end{matrix}\right)\right\vert^2{\rm d}y = \frac{c}{n\log^4n},
\end{gather*}
where we have used that $\Psi_{12}, \Psi_{22} \in L^2(0,\infty)$, which is clear from (\ref{PsiAsymptotics}). So we have proved
\begin{gather}
I_1 = O\left(\frac{1}{n\log^4 n}\right).\label{I1}
\end{gather}
Similarly
\begin{gather*}
I_2 = \int_{(\Sigma \cap U_{1/n}(-1)) \backslash [-1,1]} \left\vert\tilde{\mu}\left( v_{\Sigma} - \tilde{v}_\Sigma \right)\right\vert^2 \vert {\rm d}s\vert = \int_{(\Sigma \cap U_{1/n}(-1)) \backslash [-1,1]}\left\vert\tilde{\mu}\left( v_{\Sigma} - \tilde{v}_\Sigma \right)\right\vert^2 \vert {\rm d}s\vert \nonumber\\
\hphantom{I_2 }{}= \int_{(\Sigma \cap U_{1/n}(-1)) \backslash [-1,1]}\left\vert \sigma_3R^{(n)}(-s) E(-s) (2\pi n)^{\sigma_3/2}\Psi\big(n^2f(-s)\big) \phi^{-n\sigma_3}(-s)\sigma_3\vphantom{\left( \begin{matrix}
0 & 0\\
\left(\dfrac{F^2}{w}(s) -1\right) \phi^{-2n}(s) & 0
\end{matrix}\right)} \right. \nonumber\\
\left. \hphantom{I_2 =}{} \times\left( \begin{matrix}
0 & 0\\
\left(\dfrac{F^2}{w}(s) -1\right) \phi^{-2n}(s) & 0
\end{matrix}\right)\right\vert^2 \vert {\rm d}s\vert\nonumber\\
\hphantom{I_2 }{}
\le cn \int_{(\Sigma \cap U_{1/n}(-1)) \backslash [-1,1]}\left\vert \left( \begin{matrix}
\Psi_{12}\big(n^2f(-s)\big) & 0\\
\Psi_{22}\big(n^2f(-s)\big) & 0
\end{matrix}\right)\right\vert^2 \left\vert \left(\frac{F^2}{w}(s)-1\right)\phi^{-n}(-s)\right\vert^2 \vert {\rm d}s \vert \nonumber\\
\hphantom{I_2 }{}\le cn\int_{(\Sigma \cap U_{1/n}(-1)) \backslash [-1,1]}\left\vert \left( \begin{matrix}
\Psi_{12}\big(n^2f(-s)\big) & 0\\
\Psi_{22}\big(n^2f(-s)\big) & 0
\end{matrix}\right)\right\vert^2 \left\vert \frac{F^2}{w}(s)-1\right\vert^2 \vert {\rm d}s\vert.
\end{gather*}
By Proposition \ref{FProps},
\begin{gather*}
\frac{F^2}{w}(s) - 1 = O\big(\vert s + 1\vert^{1/2}\big), \qquad \text{as $s \to -1$},
\end{gather*}
and so for $\vert s + 1 \vert \le \frac{1}{n}$,
\begin{gather}
\left\vert \frac{F^2}{w}(s) - 1 \right\vert = O\big(n^{-1/2}\big),\label{FTonLeft}
\end{gather}
and so
\begin{gather*}
cn\int_{(\Sigma \cap U_{1/n}(-1)) \backslash [-1,1]}\left\vert \left( \begin{matrix}
\Psi_{12}\big(n^2f(-s)\big) & 0\\
\Psi_{22}\big(n^2f(-s)\big) & 0
\end{matrix}\right)\right\vert^2 \left\vert \frac{F^2}{w}(s)-1\right\vert^2 \vert {\rm d}s\vert \nonumber\\
\qquad{} \le c \int_{(\Sigma \cap U_{1/n}(-1)) \backslash [-1,1]}\left\vert \left( \begin{matrix}
\Psi_{12}\big(n^2f(-s)\big) & 0\\
\Psi_{22}\big(n^2f(-s)\big) & 0
\end{matrix}\right)\right\vert^2 \vert {\rm d}s\vert.
\end{gather*}
Making the change of variables $y = n^2 f(-s)$, and using (\ref{fFact}),
\begin{gather*}
c \int_{(\Sigma \cap U_{1/n}(-1)) \backslash [-1,1]}\left\vert \left( \begin{matrix}
\Psi_{12}\big(n^2f(-s)\big) & 0\\
\Psi_{22}\big(n^2f(-s)\big) & 0
\end{matrix}\right)\right\vert^2 \vert {\rm d}s\vert \nonumber\\
\qquad{} =\frac{c}{n^2}\left(1 + O\left(\frac{1}{n}\right)\right) \int_{n^2f(-(\Sigma \cap U_{1/n}(-1)) \backslash [-1,1])}\left\vert \left( \begin{matrix}
\Psi_{12}(y) & 0\\
\Psi_{22}(y) & 0
\end{matrix}\right)\right\vert^2\vert {\rm d}y\vert.
\end{gather*}
Recall that $n^2f(- (\Sigma \cap U_{\delta}(-1))) \subset \Gamma$, see (\ref{SigmaToGamma}), therefore
\begin{gather*}
\frac{c}{n^2}\left(1 + O\left(\frac{1}{n}\right)\right) \int_{n^2f(-(\Sigma \cap U_{1/n}(-1)) \backslash [-1,1])}\left\vert \left( \begin{matrix}
\Psi_{12}(y) & 0\\
\Psi_{22}(y) & 0
\end{matrix}\right)\right\vert^2\vert {\rm d}y\vert \nonumber\\
\qquad{} =\frac{c}{n^2}\int_{\substack{0 \le \vert y \vert \le O(n) \\ \vert \arg y \vert =2\pi/3}}\left\vert \left( \begin{matrix}
\Psi_{12}(y) & 0\\
\Psi_{22}(y) & 0
\end{matrix}\right)\right\vert^2\vert {\rm d}y\vert \le \frac{c}{n^2},
\end{gather*}
where the asymptotics of $\Psi_{12}$, $\Psi_{22}$ in (\ref{PsiAsymptotics}) imply $\Psi_{12}(y), \Psi_{22}(y) \in L^2(C_0)$, where $C_0$ is the contour $\vert \arg y \vert = 2\pi/3$. This proves
\begin{gather}
I_2 = O\left(\frac{1}{n^2}\right).\label{I2}
\end{gather}
Thus combining (\ref{IBreakdown}), (\ref{I1}), (\ref{I2}), and (\ref{I3}), we see that
\begin{gather*}
\big\vert \big\vert \tilde{\mu}^{(n)}\big( v_{\Sigma}^{(n)} - \tilde{v}_\Sigma^{(n)} \big) \big\vert \big\vert_{L^2(\Sigma)}^2 = O\left(\frac{1}{n\log^4n}\right),\\
\big\vert \big\vert \tilde{\mu}^{(n)}\big( v_{\Sigma}^{(n)} - \tilde{v}_\Sigma^{(n)} \big) \big\vert \big\vert_{L^2(\Sigma)} = O\left(\frac{1}{n^{1/2}\log^2n}\right),
\end{gather*}
thus proving (\ref{Mu-MuL2PreApp}).

The equation (\ref{Mu-MuL2App}) is a relatively straightforward consequence of (\ref{Mu-MuL2PreApp}) as follows
\begin{gather*}
\mu - \tilde{\mu} = (1 - C_{v_\Sigma})^{-1}I - (1 - C_{\tilde{v}_\Sigma})^{-1} I = (1 - C_{v_\Sigma})^{-1} (C_{v_\Sigma} - C_{\tilde{v}_\Sigma}) (1-C_{\tilde{v}_\Sigma})^{-1}I \nonumber\\
\hphantom{\mu - \tilde{\mu}}{} = (1 - C_{v_\Sigma})^{-1} (C_{v_\Sigma} - C_{\tilde{v}_\Sigma}) \tilde{\mu} = (1-C_{v_\Sigma})^{-1} C_\Sigma^- (\tilde{\mu} (v_\Sigma - \tilde{v}_\Sigma)).
\end{gather*}
Therefore, using the fact that $C_\Sigma^-$ and $(1 - C_{v_\Sigma})^{-1}$ are both bounded as $L^2 \to L^2$ operators uniformly in $n$, we see that
\begin{gather*}
\big\Vert \mu - \tilde{\mu}\big\Vert_{L^2} \le \big\Vert (1 - C_{v_\Sigma})^{-1}\big\Vert_{L^2 \to L^2} \big\Vert C_{\Sigma}^-\big\Vert_{L^2 \to L^2} \big\Vert \tilde{\mu}(v_\Sigma - \tilde{v}_\Sigma)\big\Vert_{L^2} = O\left(\frac{1}{n^{1/2}\log^2n}\right),
\end{gather*}
thus proving (\ref{Mu-MuL2App}).

As in the proof of (\ref{Mu-MuL2PreApp}), we have that
\begin{gather}
\big\vert \big\vert \tilde{\mu}^{(n)}\big( v_{\Sigma}^{(n)} - \tilde{v}_\Sigma^{(n)} \big) - \tilde{\mu}^{(n+1)}\big( v_{\Sigma}^{(n+1)} - \tilde{v}_\Sigma^{(n+1)} \big) \big\vert \big\vert_{L^2(\Sigma)}^2 \nonumber\\
\qquad{} =\left( \int_{\Sigma \cap U_{1/n}(1)}+ \int_{\Sigma\cap U_{\frac{1}{n}}(-1)} + \int_{\Sigma \backslash (U_{\frac{1}{n}}(-1) \cup U_{\frac{1}{n}}(1))} \right) \nonumber\\
\qquad\quad{}\times \big\vert\tilde{\mu}^{(n)}\big( v_{\Sigma}^{(n)} - \tilde{v}_\Sigma^{(n)} \big) - \tilde{\mu}^{(n+1)}\big( v_{\Sigma}^{(n+1)} - \tilde{v}_\Sigma^{(n+1)} \big) \big\vert^2 \vert {\rm d}s\vert = I_1 + I_2 + I_3,\label{IBreakdownHard}
\end{gather}
and the same estimates as before imply
\begin{gather}
I_3 = O\big(e^{-cn^{1/2}}\big). \label{I3Hard}
\end{gather}
For $I_1$, we see that
\begin{gather}
I_1 = \int_1^{1+\frac{1}{n}}\big\vert\tilde{\mu}^{(n)}\big( v_{\Sigma}^{(n)} - \tilde{v}^{(n)}_\Sigma \big)-\tilde{\mu}^{(n+1)}\big( v_{\Sigma}^{(n+1)} - \tilde{v}^{(n+1)}_\Sigma \big)\big\vert^2 {\rm d}s \nonumber\\
\hphantom{I_1 }{}= \int_1^{1+\frac{1}{n}}\Bigg\vert \left(I + \frac{R_1(s)}{n} + O\left(\frac{1}{n^2}\right)\right) E(s) (2\pi n)^{\sigma_3/2}\Psi_2\big(n^2f(s)\big) \phi^{n}(s) \nonumber\\
\hphantom{I_1 =}{} \times\left( \begin{matrix}
0 & 0\\
\left(\dfrac{F^2}{w_+}(s) + \dfrac{F^2}{w_-}(s) - 2\right) \phi^{-2n} & 0
\end{matrix}\right)
- \big( (n) \leftrightarrow (n+1)\big) \Bigg \vert^2 {\rm d}s \nonumber\\
\hphantom{I_1 }{}= \int_1^{1+\frac{1}{n}}\Bigg\vert \left(I + \frac{R_1(s)}{n} + O\left(\frac{1}{n^2}\right)\right) E(s) (2\pi n)^{\sigma_3/2}\left( \begin{matrix}
\Psi_{12}\big(n^2f(s)\big) & 0\\
\Psi_{22}\big(n^2f(s)\big) & 0
\end{matrix}\right)\nonumber\\
\hphantom{I_1 =}{} \times\left(\frac{F^2}{w_+}(s) + \frac{F^2}{w_-}(s) - 2\right) \phi^{-n}
- \big( (n) \leftrightarrow (n+1)\big) \Bigg \vert^2 {\rm d}s. \label{Something12}
\end{gather}
Making the substitution $y = n^2f(s)$, and defining $s = s_n(y) = f^{-1}_1\big(\frac{y}{n^2}\big)$, where $f_1^{-1}$ denotes the inverse of $f_1$ in the neighborhood~$B$ where it is conformal, (\ref{Something12}) becomes
\begin{gather}
\int_0^{n^2f\left(1+\frac{1}{n}\right)}\left\vert \left(I + \frac{R_1(s)}{n} + O\left(\frac{1}{n^2}\right)\right) E(s) (2\pi n)^{\sigma_3/2} \left( \begin{matrix}
\Psi_{12}(y) & 0\\
\Psi_{22}(y) & 0
\end{matrix}\right) \right.\nonumber\\
\qquad{} \times\left(\frac{F^2}{w_+}(s) + \frac{F^2}{w_-}(s) - 2\right) e^{-2y^{1/2}}
- \left. \left(I + \frac{R_1(s)}{n+1} + O\left(\frac{1}{n^2}\right)\right) E(s) (2\pi (n+1))^{\sigma_3/2}\right.\nonumber\\
\left.\qquad{}\times \left( \begin{matrix}
\Psi_{12}\left(\dfrac{(n+1)^2}{n^2} y\right) & 0\vspace{1mm}\\
\Psi_{22}\left(\dfrac{(n+1)^2}{n^2} y\right) & 0
\end{matrix}\right)\left(\frac{F^2}{w_+}(s) + \frac{F^2}{w_-}(s) - 2\right) e^{-2\frac{n+1}{n}y^{1/2}}\right\vert^2 {\rm d} f_1^{-1}\left(\frac{y}{n^2}\right)\nonumber\\
= 2\pi n \int_0^{n^2f\left(1+\frac{1}{n}\right)}\left\vert \left(I + \frac{R_1(s)}{n} + O\left(\frac{1}{n^2}\right)\right) E(s) \left(\begin{matrix}
1 & 0\\
0 & \dfrac{1}{2\pi n}
\end{matrix}\right)\left( \begin{matrix}
\Psi_{12}(y) & 0\\
\Psi_{22}(y) & 0
\end{matrix}\right) \right.\nonumber\\
 \qquad{} \times\left(\frac{F^2}{w_+}(s) + \frac{F^2}{w_-}(s) - 2\right) e^{-2y^{1/2}}\nonumber\\
 \left. \qquad{} - \left(I + \frac{R_1(s)}{n+1} + O\left(\frac{1}{n^2}\right)\right) E(s) \left(\begin{matrix}
\sqrt{\dfrac{n+1}{n}} & 0\\
0 & \dfrac{1}{2\pi \sqrt{n(n+1)}}
\end{matrix}\right)\right. \nonumber\\
\left.\qquad{}\times \left( \begin{matrix}
\Psi_{12}\left(\dfrac{(n+1)^2}{n^2} y\right) & 0\vspace{1mm}\\
\Psi_{22}\left(\dfrac{(n+1)^2}{n^2} y\right) & 0
\end{matrix}\right)\left(\frac{F^2}{w_+}(s) + \frac{F^2}{w_-}(s) - 2\right) e^{-2\frac{n+1}{n}y^{1/2}}\right\vert^2 {\rm d} f_1^{-1}\left(\frac{y}{n^2}\right).\label{Something13}
\end{gather}
First note that $n^2 f(1 + \frac{1}{n}) = O(n)$ and for $\vert y \vert \le cn$, ${\rm d} f_1^{-1}\big(\frac{y}{n^2}\big) =\frac{{\rm d}y}{n^2} \big(1 + O\big(\frac{1}{n}\big)\big)$. As in (\ref{FTonRight}), $\big\vert\frac{F^2}{w_+}(s) + \frac{F^2}{w_-}(s) - 2\big\vert \le \frac{c}{\log^2n}$. Therefore (\ref{Something13}) is
\begin{gather}
\int_0^{O(n)}\left\vert \left(I + \frac{R_1(s)}{n} + O\left(\frac{1}{n^2}\right)\right) E(s) \left(\begin{matrix}
1 & 0 \\
0 & \dfrac{1}{2\pi n}
\end{matrix}\right) \left( \begin{matrix}
\Psi_{12}(y) & 0\\
\Psi_{22}(y) & 0
\end{matrix}\right)e^{-2y^{1/2}}\right. \nonumber\\
\left. \qquad{}- \left(I + \frac{R_1(s)}{n+1} + O\left(\frac{1}{n^2}\right)\right) E(s) \left(\begin{matrix}
2\pi \sqrt{\dfrac{n+1}{n}} & 0\\
0 & \dfrac{1}{2\pi \sqrt{n(n+1)}}
\end{matrix}\right)\right. \nonumber\\
 \left. \qquad{} \times \left( \begin{matrix}
\Psi_{12}\left(\dfrac{(n+1)^2}{n^2} y\right) & 0\vspace{1mm}\\
\Psi_{22}\left(\dfrac{(n+1)^2}{n^2} y\right) & 0
\end{matrix}\right) e^{-2\frac{n+1}{n}y^{1/2}}\right\vert^2 {\rm d}y \, O\left(\frac{1}{n\log^4n}\right).\label{Something14}
\end{gather}
Now $R_1$ and $E$ are uniformly bounded functions, and $\Psi_{12}$, $\Psi_{22}$ are $L^2(0,\infty)$ functions. Therefore, performing the subtraction in (\ref{Something14}) term by term, (\ref{Something14}) is
\begin{gather}
\int_0^{O(n)}\left\vert O\left(\frac{1}{n}\right)\left( \begin{matrix}
\Psi_{12}(y) & 0\\
\Psi_{22}(y) & 0
\end{matrix}\right)e^{-2y^{1/2}}\right\vert^2 {\rm d}y \, O\left(\frac{1}{n\log^4n}\right)\nonumber\\
+ \int_0^{O(n)}\left\vert\left( \begin{matrix}
\Psi_{12}(y) & 0\\
\Psi_{22}(y) & 0
\end{matrix}\right)e^{-2y^{1/2}} - \left( \begin{matrix}
\Psi_{12}\left(\dfrac{(n+1)^2}{n^2}y\right) & 0\vspace{1mm}\\
\Psi_{22}\left(\dfrac{(n+1)^2}{n^2}y\right) & 0
\end{matrix}\right)e^{-2\frac{n+1}{n}y^{1/2}}\right\vert^2 {\rm d}y \, O\left(\frac{1}{n\log^4n}\right) \nonumber\\
=O\left(\frac{1}{n^3\log^4n}\right) + \int_0^{O(n)}\left\vert \int_y^{\frac{(n+1)^2}{n^2}y}\frac{{\rm d}}{{\rm d}x}\left( \left( \begin{matrix}
\Psi_{12}(x) & 0\\
\Psi_{22}(x) & 0
\end{matrix}\right)e^{-2x^{1/2}}\right){\rm d}x\right\vert^2 {\rm d}y \, O\left(\frac{1}{n\log^4n}\right)\nonumber\\
=O\left(\frac{1}{n^3\log^4n}\right) + \int_0^{O(n)}\left\vert \frac{(n+1)^2-n^2}{n^2}y\sup_{\left(y,\frac{(n+1)^2}{n^2}y\right)}\frac{{\rm d}}{{\rm d}x} \left( \left( \begin{matrix}
\Psi_{12}(x) & 0\\
\Psi_{22}(x) & 0
\end{matrix}\right)e^{-2x^{1/2}}\right)\right\vert^2 {\rm d}y \nonumber\\
\qquad{} \times O\left(\frac{1}{n\log^4n}\right).\label{n3log4n}
\end{gather}
Now, $\Psi_2(x) = O(\log x)$ and $\Psi_2'(x) = O\big(\frac{1}{x}\big)$ as $x \to 0$ and $\Psi_2, \Psi_2'(x) = O\big(x^{1/4}e^{-2x^{1/2}}\big)$ as $x \to \infty$, see (\ref{PsiAsymptotics}). As a consequence,
\begin{gather*}
y\sup_{\left(y,\frac{(n+1)^2}{n^2}y\right)}\frac{{\rm d}}{{\rm d}x}\left( \left( \begin{matrix}
\Psi_{12}(x) & 0\\
\Psi_{22}(x) & 0
\end{matrix}\right)e^{-2x^{1/2}}\right) = \begin{cases}
O(1) & \text{for all $y$}, \\
O\big(y^{1/4}e^{-4y^{1/2}}\big) & \text{as $y \to \infty$},
\end{cases}
\end{gather*}
and thus
\begin{gather*}
\left\Vert y\sup_{\left(y,\frac{(n+1)^2}{n^2}y\right)}\frac{{\rm d}}{{\rm d}x}\left( \left( \begin{matrix}
\Psi_{12}(x) & 0\\
\Psi_{22}(x) & 0
\end{matrix}\right)e^{-2x^{1/2}}\right)\right\Vert_{L^2(0,\infty)} = O(1),
\end{gather*}
and so
\begin{gather*}
\int_0^{O(n)}\left\vert \frac{(n+1)^2-n^2}{n^2}y\sup_{\left(y,\frac{(n+1)^2}{n^2}y\right)}\frac{{\rm d}}{{\rm d}x}\left( \left( \begin{matrix}
\Psi_{12}(x) & 0\\
\Psi_{22}(x) & 0
\end{matrix}\right)e^{-2x^{1/2}}\right){\rm d}x\right\vert^2 {\rm d}y\nonumber\\
\qquad{} \le \frac{c}{n^2} \int_0^{\infty}\left\vert y\sup_{\left(y,\frac{(n+1)^2}{n^2}y\right)}\frac{{\rm d}}{{\rm d}x}\left( \left( \begin{matrix}
\Psi_{12}(x) & 0\\
\Psi_{22}(x) & 0
\end{matrix}\right)e^{-2x^{1/2}}\right){\rm d}x\right\vert^2 {\rm d}y=O\left(\frac{1}{n^2}\right),
\end{gather*}
and hence, by (\ref{n3log4n}),
\begin{gather}
I_1 = O\left(\frac{1}{n^3\log^4 n}\right).\label{I1Hard}
\end{gather}
A similar calculation shows
\begin{gather}
I_2 = O\left(\frac{1}{n^4}\right).\label{I2Hard}
\end{gather}
Thus combining (\ref{IBreakdownHard}), (\ref{I1Hard}), (\ref{I2Hard}), and (\ref{I3Hard}), we see that
\begin{gather*}
\big\vert \big\vert \tilde{\mu}^{(n)}\big( v_{\Sigma}^{(n)} - \tilde{v}_\Sigma^{(n)} \big)- \tilde{\mu}^{(n+1)}\big( v_{\Sigma}^{(n+1)} - \tilde{v}_\Sigma^{(n+1)} \big) \big\vert \big\vert_{L^2(\Sigma)}^2 = O\left(\frac{1}{n^3\log^4n}\right), \nonumber\\
\big\vert \big\vert \tilde{\mu}^{(n)}\big( v_{\Sigma}^{(n)} - \tilde{v}_\Sigma^{(n)} \big)- \tilde{\mu}^{(n+1)}\big( v_{\Sigma}^{(n+1)} - \tilde{v}_\Sigma^{(n+1)} \big) \big\vert \big\vert_{L^2(\Sigma)} = O\left(\frac{1}{n^{3/2}\log^2n}\right),
\end{gather*}
thus proving (\ref{Mu-MuL2HardPreApp}).

To move from (\ref{Mu-MuL2HardPreApp}) to (\ref{Mu-MuL2HardApp}), we proceed as follows
\begin{gather}
 \big(\mu^{(n)} - \tilde{\mu}^{(n)}\big) - \big(\mu^{(n+1)} - \tilde{\mu}^{(n+1)}\big) \nonumber\\
 \qquad{}= \big(\big(1-C_{v^{(n)}_\Sigma}\big)^{-1}I - \big(1-C_{\tilde{v}^{(n)}_\Sigma}\big)^{-1}I\big) - \big(\big(1-C_{v^{(n+1)}_\Sigma}\big)^{-1}I - \big(1-C_{v^{(n+1)}_\Sigma}\big)^{-1}I\big) \nonumber\\
 \qquad{} = \big(1-C_{v^{(n)}_\Sigma}\big)^{-1} C_\Sigma^- \tilde{\mu}^{(n)}\big(v_\Sigma^{(n)} - \tilde{v}_\Sigma^{(n)}\big) - \big(1-C_{v^{(n+1)}_\Sigma}\big)^{-1} C_\Sigma^- \tilde{\mu}^{(n+1)}\big(v_\Sigma^{(n+1)} - \tilde{v}_\Sigma^{(n+1)}\big)\nonumber\\
 \qquad{} = \big( \big(1-C_{v^{(n)}_\Sigma}\big)^{-1} - \big(1-C_{v^{(n+1)}_\Sigma}\big)^{-1}\big)C_\Sigma^- \tilde{\mu}^{(n)}\big(v_\Sigma^{(n)} - \tilde{v}_\Sigma^{(n)}\big) \nonumber\\
\qquad\quad{}+\big(1-C_{v^{(n+1)}_\Sigma}\big)^{-1} C_\Sigma^- \big(\tilde{\mu}^n \big(v_\Sigma^{(n)} - \tilde{v}_\Sigma^{(n)}\big) - \tilde{\mu}^{(n+1)}\big(v_\Sigma^{(n+1)} - \tilde{v}_\Sigma^{(n+1)}\big) \big).\label{122}
\end{gather}
Now, by (\ref{Mu-MuL2HardPreApp}) and the boundedness of the operators $C_\Sigma^-$ and $(1-C_{v^{(n+1)}_\Sigma})^{-1}$, we see that
\begin{gather}
\big\Vert \big(1-C_{v^{(n+1)}_\Sigma}\big)^{-1} C_\Sigma^- \big(\tilde{\mu}^n \big(v_\Sigma^{(n)} - \tilde{v}_\Sigma^{(n)}\big) - \tilde{\mu}^{(n+1)}\big(v_\Sigma^{(n+1)} - \tilde{v}_\Sigma^{(n+1)}\big) \big)\big\Vert_{L^2}\nonumber\\
\qquad{} = O\left(\frac{1}{n^{3/2}\log^2 n}\right)\label{122First}
\end{gather}
and
\begin{gather*}
 \big( \big(1-C_{v^{(n)}_\Sigma}\big)^{-1} - \big(1-C_{v^{(n+1)}_\Sigma}\big)^{-1}\big)C_\Sigma^- \tilde{\mu}^{(n)}\big(v_\Sigma^{(n)} - \tilde{v}_\Sigma^{(n)}\big) \nonumber\\
\qquad{} = \big(\big(1 - C_{v^{(n)}_\Sigma}\big)^{-1}\big(C_{v^{(n)}_\Sigma} - C_{v^{(n+1)}_\Sigma}\big)\big(1 - C_{v^{(n+1)}_\Sigma}\big)^{-1}\big)C_\Sigma^- \tilde{\mu}^{(n)}\big(v_\Sigma^{(n)} - \tilde{v}_\Sigma^{(n)}\big).
\end{gather*}
Now
\begin{gather*}
\big\Vert C_{v^{(n)}_\Sigma} - C_{v^{(n+1)}_\Sigma}\big\Vert_{L^2 \to L^2} = \Vert C_\Sigma^- \Vert_{L^2\to L^2} \big\Vert v_\Sigma^{(n)} - v_{\Sigma}^{(n+1)}\big\Vert_{L^\infty(\Sigma)},
\end{gather*}
where
\begin{gather*}
v_\Sigma^{(n)} - v_\Sigma^{(n+1)} = \begin{cases}\left(\begin{matrix}
0 & 0\\
\dfrac{F^2}{w}(z) \phi^{-2n}(z) \big(1 - \phi^{ - 2}(z)\big) & 0
\end{matrix}\right) & \text{for $z \in \Sigma\backslash [-1,1]$}, \\
0 & \text{for $z \in [-1,1]$}.
\end{cases}
\end{gather*}
Therefore
\begin{gather}
\big\Vert v_\Sigma^{(n)} - v_{\Sigma}^{(n+1)}\big\Vert_{L^\infty(\Sigma)} \le c \big\Vert \phi^{-2n} \big(1 - \phi^{-2}\big)\big\Vert_{L^\infty(\Sigma\backslash [-1,1])}.\label{Something20}
\end{gather}
When $z \in (\Sigma\backslash[-1,1]) \cap U_\delta(1)$, $\phi(z) = 1 + \sqrt{2}(z-1)^{1/2} + O(z-1)$, letting $y = n^2 (z-1)$,
\begin{gather*}
\phi^{-2n} (z)\big(1 - \phi^{-2}\big)(z) = \left(1 + \sqrt{2}\frac{y^{1/2}}{n} + O\left(\frac{y}{n^2}\right)\right)^{-2n}\left( 2\sqrt{2}\frac{y^{1/2}}{n}+ O\left(\frac{y}{n^2}\right)\right) \nonumber\\
\hphantom{\phi^{-2n} (z)\big(1 - \phi^{-2}\big)(z)}{} = \frac{2\sqrt{2}y^{1/2}}{n} e^{- 2\sqrt{2} y^{1/2}}\left(1 + O\left(\frac{y^{1/2}}{n}\right)\right).
\end{gather*}
Since $z \in (\Sigma\backslash[-1,1]) \cap U_\delta(1)$, $\operatorname{Re}(y^{1/2}) \ge \epsilon > 0$. Therefore $y^{1/2} e^{-2\sqrt{2} y^{1/2}} \le c < \infty$, so
\begin{gather*}
\big\vert \phi^{-2n} (z)\big(1 - \phi^{-2}\big)(z) \big\vert \le \frac{c}{n}, \qquad \text{for $z \in (\Sigma\backslash[-1,1]) \cap U_\delta(1)$}.
\end{gather*}
The same argument shows the same bound around the point $-1$. Since $\vert \phi(z)\vert \ge 1 + \delta > 1$ when $z \in \Sigma \backslash ([-1,1] \cup U_\delta(1) \cup U_\delta(-1))$,
\begin{gather*}
\left\vert \phi^{-2n} (z)\big(1 - \phi^{-2}\big)(z) \right\vert \le ce^{-n\delta}, \qquad \text{for $z \in \Sigma \backslash ([-1,1] \cup U_\delta(1) \cup U_\delta(-1))$}
\end{gather*}
Therefore,
\begin{gather*}
 \big\Vert \phi^{-2n} \big(1 - \phi^{-2}\big)\big\Vert_{L^\infty(\Sigma\backslash [-1,1])} = O\left(\frac{1}{n}\right),
\end{gather*}
which, combined with (\ref{Something20}), implies
\begin{gather*}
\big\Vert v_\Sigma^{(n)} - v_{\Sigma}^{(n+1)}\big\Vert_{L^\infty(\Sigma)} = O\left(\frac{1}{n}\right),
\qquad
\big\Vert C_{v^{(n)}_\Sigma} - C_{v^{(n+1)}_\Sigma}\big\Vert_{L^2 \to L^2} = O\left(\frac{1}{n}\right).
\end{gather*}
Therefore
\begin{gather}
\big\Vert \big( \big(1-C_{v^{(n)}_\Sigma}\big)^{-1} - \big(1-C_{v^{(n+1)}_\Sigma}\big)^{-1}\big)C_\Sigma^- \tilde{\mu}^{(n)}\big(v_\Sigma^{(n)} - \tilde{v}_\Sigma^{(n)}\big) \big\Vert_{L^2(\Sigma)}\nonumber\\
\qquad{} = \big\Vert\big(\big(1 - C_{v^{(n)}_\Sigma}\big)^{-1}\big(C_{v^{(n)}_\Sigma} - C_{v^{(n+1)}_\Sigma}\big)\big(1 - C_{v^{(n+1)}_\Sigma}\big)^{-1}\big)C_\Sigma^-\big\Vert_{L^2(\Sigma) \to L^2(\Sigma)} \nonumber\\
\qquad\quad{}\times \big\Vert \tilde{\mu}^{(n)}\big(v_\Sigma^{(n)} - \tilde{v}_\Sigma^{(n)}\big)\big\Vert_{L^2(\Sigma)}, \nonumber\\
 \big\Vert \big(C_{v^{(n)}_\Sigma} - C_{v^{(n+1)}_\Sigma}\big)\big\Vert_{L^2(\Sigma) \to L^2(\Sigma)} \big\Vert \tilde{\mu}^{(n)}\big(v_\Sigma^{(n)} - \tilde{v}_\Sigma^{(n)}\big)\big\Vert_{L^2(\Sigma)} \nonumber\\
\qquad{} = O\left(\frac{1}{n}\right)O\left(\frac{1}{n^{1/2}\log^2n}\right) = O\left(\frac{1}{n^{3/2}\log^2n}\right). \label{122Second}
\end{gather}
So, combining (\ref{122}), (\ref{122First}), and (\ref{122Second}), we see that
\begin{gather*}
\big\vert \big\vert \big(\mu^{(n)} - \tilde{\mu}^{(n)}\big) - \big(\mu^{(n+1)} - \tilde{\mu}^{(n+1)}\big) \big\vert \big\vert_{L^2(\Sigma)} = O\left(\frac{1}{n^{3/2}\log^2 n }\right)
\end{gather*}
as desired.
\end{proof}

\section{Leading order asymptotics of the difference formula}
\label{IntegralSection}
In this appendix we calculate the integrals that arise in Proposition \ref{Replacement}.
\subsection{Some preliminaries}
\begin{prop}
The following bounds hold
\begin{gather}
\int_1^{1+1/n} \left\vert \Psi_{12}^2 \big(n^2f(s)\big) \left(\frac{F^2}{w_+}(s) + \frac{F^2}{w_-}(s) - 2\right) \right\vert \vert {\rm d}s\vert = O\left(\frac{1}{n^2\log^2 n}\right),\nonumber\\
\int_1^{1+1/n} \left\vert \Psi_{22}^2 \big(n^2f(s)\big) \left(\frac{F^2}{w_+}(s) + \frac{F^2}{w_-}(s) - 2\right) \right\vert \vert {\rm d}s\vert = O\left(\frac{1}{n^2\log^2 n}\right), \nonumber\\
\int_1^{1+1/n} \left\vert \Psi_{12} \big(n^2f(s)\big)\Psi_{22} \big(n^2f(s)\big) \left(\frac{F^2}{w_+}(s) + \frac{F^2}{w_-}(s) - 2\right) \right\vert \vert {\rm d}s\vert = O\left(\frac{1}{n^2\log^2 n}\right),\label{GNEstimates}
\end{gather}
furthermore
\begin{gather}
\int_1^{1+1/n} \Psi_{12}^2 \big(n^2f(s)\big) \left(\frac{F^2}{w_+}(s) + \frac{F^2}{w_-}(s) - 2\right) \dbar s = \frac{3}{16 \pi in^2 \log^2 n} + O\left(\frac{1}{n^2 \log^3n}\right).\label{GNIntegral}
\end{gather}
\end{prop}
\begin{proof}
Recall that by Proposition \ref{FCancellation}, as $s \to 1$,
\begin{gather*}
 \frac{F^2}{w_+}(s) + \frac{F^2}{w_-}(s) - 2 = -\frac{3\pi^2}{\log^2 \frac{2k}{1-s}} + O\left(\frac{1}{\log^3\vert 1- s \vert} \right).
\end{gather*}
In particular, note that for $1 \le s \le 1+1/n$, that
\begin{gather*}
\left\vert \frac{F^2}{w_+}(s) + \frac{F^2}{w_-}(s) - 2\right\vert = O\left(\frac{1}{\log^2 n}\right).
\end{gather*}
Therefore
\begin{gather*}
\int_1^{1+1/n} \left\vert \Psi_{12}^2 \big(n^2f(s)\big) \left(\frac{F^2}{w_+}(s) + \frac{F^2}{w_-}(s) - 2\right) \right\vert \vert {\rm d}s\vert\nonumber\\
\qquad{} \le \frac{c}{\log^2n} \int_1^{1+1/n}\! \big\vert \Psi_{12}^2 \big(n^2f(s)\big) \big\vert {\rm d}s
= \frac{c}{\log^2n} \int_0^{n^2f(1 + 1/n)}\! \big\vert \Psi_{12}^2 (y) \big\vert \left(\frac{1}{n^2} + O\big(n^{-3}\big)\right) {\rm d}y\nonumber\\
\qquad{} \le \frac{c}{n^2 \log^2n} \int_0^{n^2f(1 + 1/n)} \big\vert \Psi_{12}^2 (y) \big\vert {\rm d}y
\le \frac{c}{n^2 \log^2n} \int_0^\infty \big\vert \Psi_{12}^2 (y) \big\vert {\rm d}y
\le \frac{c}{n^2 \log^2n},
\end{gather*}
where we have used the fact that
\begin{gather*}
f(z) = \frac{z-1}{2} + O\big((z-1)^2\big),
\end{gather*}
and that $\Psi_{12}(y) \in L^2(0,\infty)$. Similar arguments show the other estimates in (\ref{GNEstimates}).

Paying more attention, we can make the calculation (\ref{GNIntegral}),
\begin{gather*}
 \int_1^{1+1/n}\Psi_{12}^2\big(n^2f(s)\big)\left(\frac{F^2}{w_+}(s) + \frac{F^2}{w_-}(s) - 2\right)\dbar s\nonumber\\
\qquad{} = \int_1^{1+1/n}\Psi_{12}^2\big(n^2f(s)\big)\left(-\frac{3\pi^2}{\log^2 (2k/\vert 1-s \vert)} + O\left(\frac{1}{\log^{3} \vert 1-s \vert}\right)\right) \dbar s.
\end{gather*}
Letting $s = 1 + \frac{t}{n^2}$,
\begin{gather}
 \int_1^{1+1/n}\Psi_{12}^2\big(n^2f(s)\big)\left(-\frac{3\pi^2}{\log^2 (2k/\vert 1-s \vert)} + O\left(\frac{1}{\log^{3} \vert 1-s \vert}\right)\right) \dbar s\nonumber\\
 = -\frac{3\pi^2}{n^2}\int_0^{n}\Psi_{12}^2\left(n^2f\left(1 + \frac{t}{n^2}\right)\right)\left(\frac{1}{\log^2 (2kn^2/ t)} + O\left(\frac{1}{\log^3 n}\right)\right) \dbar t\nonumber\\
 = -\frac{3\pi^2}{n^2}\int_0^{n}\Psi_{12}^2\left(n^2f\left(1 + \frac{t}{n^2}\right)\right)\left(\frac{1}{\left(2\log n + \log \frac{2k}{ t}\right)^2} + O\left(\frac{1}{\log n}\right)\right) \dbar t \nonumber\\
 = -\frac{3\pi^2}{4n^2\log^2n}\int_0^{n}\Psi_{12}^2\left(n^2f\left(1 + \frac{t}{n^2}\right)\right)\left(\frac{1}{\left(1 + \frac{\log \frac{2k}{ t}}{2\log n}\right)^2} + O\left(\frac{1}{\log n}\right)\right) \dbar t\nonumber\\
=-\frac{3\pi^2}{4n^2\log^2n}\int_0^{n}\Psi_{12}^2\left(n^2f\left(1 + \frac{t}{n^2}\right)\right)\left(\frac{1}{1 + \frac{\log \frac{2k}{t}}{\log n} + \frac{\log^2\frac{2k}{t }}{4\log^2 n}} + O\left(\frac{1}{\log n}\right)\right) \dbar t \label{TooComplicated}\\
=-\frac{3\pi^2}{4n^2\log^2n}\int_0^{n}\Psi_{12}^2\left(n^2f\left(1 + \frac{t}{n^2}\right)\right)\left(1 - \frac{1}{\log n}\frac{ \log \frac{2k}{t}+ \frac{\log^2\frac{2k}{t}}{4\log n}}{1 + \frac{\log \frac{2k}{t}}{\log n} + \frac{\log^2\frac{2k}{t }}{4\log^2n}} + O\left(\frac{1}{\log n}\right)\right) \dbar t.\nonumber
\end{gather}
Note that for $ 0 \le t \le n$, $- \infty < \log \frac{t}{2k} < \log n$, and
\begin{gather*}
\frac{\log \frac{2k}{t}}{2\log n} > -\frac{1}{2},
\end{gather*}
so
\begin{gather*}
1 + \frac{\log\frac{2k}{t}}{\log n} + \frac{\log^2 \frac{2k}{t}}{4 \log^2 n } = \left(1 + \frac{\log \frac{2k}{t }}{2\log n}\right)^2 \ge \frac{1}{4},
\end{gather*}
and therefore
\begin{gather*}
\frac{3\pi^2}{4n^2\log^2n}\int_0^n \left\vert\Psi_{12}^2\left(n^2f\left(1 + \frac{t}{n^2}\right)\right)\left(\frac{1}{\log n}\frac{ \log \frac{2k}{t}+ \frac{\log^2\frac{2k}{t }}{4\log n}}{1 + \frac{\log \frac{2k}{t}}{\log n} + \frac{\log^2\frac{2k}{t }}{4\log^2n}} + O\left(\frac{1}{\log n}\right)\right)\right\vert {\rm d}t\nonumber\\
\qquad{} \le \frac{c}{n^2\log^3n}\int_0^n \left\vert \Psi_{12}^2\left(n^2f\left(1 + \frac{t}{n^2}\right)\right)\right\vert \left(4\left\vert \log \frac{2k}{t}\right\vert + 4\left\vert \log^2 \frac{2k}{t}\right\vert\right) {\rm d}t.
\end{gather*}
Letting $y = n^2 f(1 + t/n^2)$, then ${\rm d}y = -f'(t/n^2) {\rm d}t = -\big(\frac{1}{2} + O\big(\frac{1}{n}\big)\big){\rm d}t$ and $t = 2y\big(1 + O\big(\frac{1}{n}\big)\big)$, so
\begin{gather}
 \frac{c}{n^2\log^3n}\int_0^n \left\vert \Psi_{12}^2\left(n^2f\left(1 + \frac{t}{n^2}\right)\right)\right\vert \left(4\left\vert \log \frac{2k}{t}\right\vert + 4\left\vert \log^2 \frac{2k}{t}\right\vert\right) {\rm d}t \nonumber\\
\qquad{} =\frac{c}{n^2\log^3n}\int_0^{n^2f(1 + 1/n)} \left\vert \Psi_{12}^2(y)\right\vert \left(\left\vert \log \frac{k}{y}\right\vert + \left\vert \log^2 \frac{k}{y}\right\vert + O\left(\frac{1}{n} + \frac{\log y}{n}\right)\right) {\rm d}y \nonumber\\
\qquad{} \le\frac{c}{n^2\log^3n}\int_0^\infty \left\vert \Psi_{12}^2(y)\right\vert \left(\left\vert \log y\right\vert + \left\vert \log^2y \right\vert\right) {\rm d}y \le \frac{c}{n^2\log^3n},\label{Integrability}
\end{gather}
where the integrability of (\ref{Integrability}) follows from the asymptotics of $\Psi_{12}(y)$ in (\ref{PsiAsymptotics}). Thus
\begin{gather*}
\frac{3\pi^2}{4n^2\log^2n}\int_0^n \left\vert\Psi_{12}^2\left(n^2f\left(1 + \frac{t}{n^2}\right)\right)\left(\frac{1}{\log n}\frac{ \log \frac{2k}{\vert t\vert}+ \frac{\log^2\frac{2k}{\vert t \vert }}{4\log n}}{1 + \frac{\log \frac{2k}{\vert t\vert}}{\log n} + \frac{\log^2\frac{2k}{\vert t \vert }}{4\log^2n}} + O\left(\frac{1}{\log n}\right)\right)\right\vert {\rm d}t\nonumber\\
\qquad{} = O\left(\frac{1}{n^2\log^3n}\right).
\end{gather*}
Therefore, picking up from (\ref{TooComplicated}),
\begin{gather*}
-\frac{3\pi^2}{4n^2\log^2n}\int_0^{n}\Psi_{12}^2\left(n^2f\left(1 + \frac{t}{n^2}\right)\right)\left(1 - \frac{1}{\log n}\frac{ \log \frac{2k}{\vert t\vert}+ \frac{\log^2\frac{2k}{\vert t \vert }}{4\log n}}{1 + \frac{\log \frac{2k}{\vert t\vert}}{\log n} + \frac{\log^2\frac{2k}{\vert t \vert }}{4\log^2n}} + O\left(\frac{1}{\log n}\right)\right) \dbar t \nonumber\\
\qquad{}=-\frac{3\pi^2}{4n^2\log^2n}\int_0^{n}\Psi_{12}^2\left(n^2f\left(1 + \frac{t}{n^2}\right)\right) \dbar t + O\left(\frac{1}{n^2\log^3n}\right)\nonumber\\
\qquad{} =-\frac{3\pi^2}{4n^2\log^2n}\int_0^{n}\left(-\frac{1}{\pi^2}\right)K_0^2\left(2nf^{1/2}\left(1 + \frac{t}{n^2}\right)\right) \dbar t + O\left(\frac{1}{n^2\log^3n}\right)\nonumber\\
\qquad{}= \frac{3}{4n^2\log^2 n} \int_0^n K_0^2\left(2nf^{1/2}\left(1 + \frac{t}{n^2}\right)\right) \dbar t + O\left(\frac{1}{n^2\log^3n}\right).
\end{gather*}
Making the change of variables $y = n^2f\big(1 + \frac{t}{n^2}\big)$ and noting that this implies ${\rm d}t = \big(2 + O\big(\frac{1}{n}\big)\big) {\rm d}y$,
\begin{gather*}
\frac{3}{4n^2\log^2 n} \int_0^n K_0^2\left(2nf^{1/2}\left(1 + \frac{t}{n^2}\right)\right) \dbar t + O\left(\frac{1}{n^2\log^3n}\right)\nonumber\\
\qquad{} = \frac{3}{2n^2\log^2 n} \int_0^{n^2f(1 + 1/n)} K_0^2\big(2y^{1/2}\big) \dbar y + O\left(\frac{1}{n^2\log^3n}\right).
\end{gather*}
$K_0(y)$ is exponentially decaying for $y >> 0$, see \cite[formula~(9.7.2)]{Stegun}. Therefore
\begin{gather*}
 \frac{3}{2n^2\log^2 n} \int_0^{n^2f(1 + 1/n)} K_0^2\big(2y^{1/2}\big) \dbar y + O\left(\frac{1}{n^2\log^3n}\right)\nonumber\\
\qquad{} = \frac{3}{2n^2\log^2 n} \int_0^{\infty} K_0^2\big(2y^{1/2}\big) \dbar y + O\left(\frac{1}{n^2\log^3n}\right) \nonumber\\
\qquad{}= \frac{3}{n^2\log^2 n} \int_0^{\infty} K_0^2\left(2u\right) u\dbar u + O\left(\frac{1}{n^2\log^3n}\right) \nonumber\\
\qquad{} = \frac{3}{2\pi i} \frac{1}{8 n^2\log^2 n} + O\left(\frac{1}{n^2\log^3n}\right)
 = \frac{3}{16 \pi i n^2 \log^2 n} + O\left(\frac{1}{n^2\log^3n}\right),
\end{gather*}
where we have used the fact that $\int_0^{\infty} K_0^2(v) v {\rm d}v = \frac{1}{2}$, see~\cite{Ryzhik}. We have proved~(\ref{GNIntegral}).
\end{proof}

Additionally, we have the following matrix calculation:
\begin{prop}
\begin{gather}
E(1) = \frac{1}{\sqrt{2}}\left(\begin{matrix}
1 & -i\\
-i & 1
\end{matrix}\right).\label{E1}
\end{gather}
\end{prop}
\begin{proof}
\begin{gather*}
E(z) =\left(\begin{matrix}
\dfrac{a + a^{-1}}{2} & \dfrac{a - a^{-1}}{2i} \\
\dfrac{a - a^{-1}}{-2i} & \dfrac{a + a^{-1}}{2}
\end{matrix}\right)\frac{1}{\sqrt{2}} \left(\begin{matrix} 1 & -i \\ -i & 1 \end{matrix}\right) f^{\sigma_3/4} \nonumber\\
\hphantom{E(z)}{} =\frac{1}{\sqrt{2}} \left(\begin{matrix}
\dfrac{a + a^{-1}}{2} & \dfrac{a - a^{-1}}{2i} \\
\dfrac{a - a^{-1}}{-2i} & \dfrac{a + a^{-1}}{2}
\end{matrix}\right)
\left(\begin{matrix}
f^{1/4}(z) & -i f^{-1/4}(z) \\
-f^{1/4}(z) & f^{-1/4}(z)
\end{matrix}\right) \\
\hphantom{E(z)}{} = \frac{1}{\sqrt{2}}\left(\begin{matrix}
a^{-1} f^{1/4}(z) & -iaf^{-1/4}(z) \\
-i a^{-1} f^{1/4}(z) & af^{-1/4}(z)
\end{matrix}\right).
\end{gather*}
Now
\begin{gather*}
af^{-1/4} = \frac{(z-1)^{1/4}}{(z+1)^{1/4}} \frac{1}{\big(\frac{1}{2} (z-1) + O\big((z-1)^2\big)\big)^{1/4}} = 1 + O(z-1).
\end{gather*}
So $af^{-1/4}(1) = 1$.

Therefore, we see that
\begin{gather*}
E(1) = \frac{1}{\sqrt{2}}\left(\begin{matrix}
1 & -i\\
-i & 1
\end{matrix}\right)
\end{gather*}
as desired.
\end{proof}

\subsection{The first integral}
\begin{prop}\label{EIntegralPropApp}
\begin{gather}
\int_\Sigma \tilde{\mu}(z) (v_\Sigma(z) - \tilde{v}_\Sigma(z)) \tilde{\mu}^{-1}(z) \dbar z = \frac{3}{16 n\log^2n} \left(\begin{matrix}
1 & -i\\
-i& -1
\end{matrix}\right) + O\left(\frac{1}{n\log^3n}\right).\label{EIntegralApp}
\end{gather}
\end{prop}
\begin{proof}
As in the proof of Proposition \ref{NormsApp}, we have
\begin{gather*}
 \int_\Sigma \tilde{\mu}(z) (v_\Sigma(z) - \tilde{v}_\Sigma(z)) \tilde{\mu}^{-1}(z) \dbar z \\
\qquad{} = \left(\int_{\Sigma \cap U_{1/n}(1)} + \int_{\Sigma \cap U_{1/n}(-1)} + \int_{\Sigma \backslash (U_{1/n}(1)\cup U_{1/n}(-1))} \right) \tilde{\mu}(z) (v_\Sigma(z) - \tilde{v}_\Sigma(z)) \tilde{\mu}^{-1}(z) \dbar z \\
\qquad{} = I_1 + I_2 + I_3.
\end{gather*}
The same argument as Proposition \ref{NormsApp}, gives
\begin{gather*}
I_3 = O\big(e^{-cn^{1/2}}\big).
\end{gather*}
For $I_1$
\begin{gather}
I_1 = -\int_1^{1 + \frac{1}{n}} \tilde{\mu}(z) (v_\Sigma(z) - \tilde{v}_\Sigma(z)) \tilde{\mu}^{-1}(z) \dbar z\nonumber\\
\hphantom{I_1}{} = -\int_1^{1 + \frac{1}{n}} \left(1 + O\left(\frac{1}{n}\right)\right)E(z) (2\pi n)^{\sigma_3/2} \Psi(n^2 f(z)) \phi^{-n\sigma_3}(z)\nonumber \\
\hphantom{I_1=}{}\times \left( \begin{matrix}
0 & 0 \\
\left(\frac{F^2}{w_+}(z) + \frac{F^2}{w_-}(z) - 2\right) \phi^{-2n}(z) & 0
\end{matrix}\right) \nonumber\\
\hphantom{I_1=}{} \times \phi^{n\sigma_3}(z)\Psi^{-1}\big(n^2f(z)\big) (2\pi n)^{-\sigma_3/2} E^{-1}(z)\left(1 + O\left(\frac{1}{n}\right)\right) \dbar z \nonumber\\
\hphantom{I_1}{}= -\int_1^{1 + \frac{1}{n}} \left(1 + O\left(\frac{1}{n}\right)\right)E(z) (2\pi n)^{\sigma_3/2}
 \Psi(n^2 f(z)) \left( \begin{matrix}
0 & 0 \\
1 & 0
\end{matrix}\right) \nonumber\\
\hphantom{I_1=}{}\times \Psi^{-1}\big(n^2f(z)\big)(2\pi n)^{-\sigma_3/2} E^{-1}(z)\left(1 + O\left(\frac{1}{n}\right)\right) \left(\frac{F^2}{w_+}(z) + \frac{F^2}{w_-}(z) - 2\right) \dbar z \nonumber\\
\hphantom{I_1}{}= -\int_1^{1 + \frac{1}{n}}\left(1 + O\left(\frac{1}{n}\right)\right)E(z) (2\pi n)^{\sigma_3/2}\left( \begin{matrix}
\Psi_{12}\Psi_{22} \big(n^2f(z)\big) & -\Psi_{12}^2\big(n^2f(z)\big) \\
\Psi_{22}^2\big(n^2f(z)\big) & -\Psi_{12}\Psi_{22}\big(n^2f(z)\big)
\end{matrix}\right) \nonumber\\
\hphantom{I_1=}{}\times (2\pi n)^{-\sigma_3/2} E^{-1}(z)\left(1 + O\left(\frac{1}{n}\right)\right) \left[ \left(\frac{F^2}{w_+}(z) + \frac{F^2}{w_-}(z) - 2 \right)\right] \dbar z\nonumber\\
\hphantom{I_1}{} = -2\pi n\int_1^{1 + \frac{1}{n}}\left(1 + O\left(\frac{1}{n}\right)\right)E(z) \left(\begin{matrix}
1 & 0 \\
0 & \dfrac{1}{2\pi n}
\end{matrix}\right)\left( \begin{matrix}
\Psi_{12}\Psi_{22} \big(n^2f(z)\big) & -\Psi_{12}^2\big(n^2f(z)\big) \\
\Psi_{22}^2\big(n^2f(z)\big) & -\Psi_{12}\Psi_{22}\big(n^2f(z)\big)
\end{matrix}\right) \nonumber\\
\hphantom{I_1=}{} \times \left(\begin{matrix}
\dfrac{1}{2\pi n} & 0 \\
0 & 1
\end{matrix}\right) E^{-1}(z)\left(1 + O\left(\frac{1}{n}\right)\right) \left[ \left(\frac{F^2}{w_+}(z) + \frac{F^2}{w_-}(z) - 2 \right)\right] \dbar z. \label{PlugInE0}
\end{gather}
Now, note that by the $n$-independence and analyticity of the functions $E(z)$, $E^{-1}(z)$, for $z \in (1, 1+1/n)$,
\begin{gather*}
E(z) = E(1) + O\left(\frac{1}{n}\right), \qquad E^{-1}(z) = E^{-1}(1) +O\left(\frac{1}{n}\right),
\end{gather*}
and clearly
\begin{gather*}
 \left(\begin{matrix}
1 & 0 \\
0 & \frac{1}{2\pi n}
\end{matrix}\right) = \left(\begin{matrix}1 & 0 \\ 0 & 0\end{matrix}\right) + O\left(\frac{1}{n}\right).
\end{gather*}
Therefore, subtracting term by term in (\ref{PlugInE0}), (\ref{PlugInE0}) becomes
\begin{gather}
 -2 \pi n\int_1^{1 + \frac{1}{n}} E(1) \left(\begin{matrix}1 & 0 \\ 0 & 0\end{matrix}\right) \left( \begin{matrix}
\Psi_{12}\Psi_{22} \big(n^2f(z)\big) & -\Psi_{12}^2\big(n^2f(z)\big) \\
\Psi_{22}^2\big(n^2f(z)\big) & -\Psi_{12}\Psi_{22}\big(n^2f(z)\big)
\end{matrix}\right) \left(\begin{matrix} 0& 0 \\ 0 & 1\end{matrix}\right) E^{-1}(1)\nonumber\\
\qquad {} \times \left[ \left(\frac{F^2}{w_+}(z) + \frac{F^2}{w_-}(z) - 2 \right) \dbar z\right] + 2\pi n\mathcal{E},\label{TossAway}
\end{gather}
where
\begin{gather*}
\mathcal{E}= O\!\left(\frac{1}{n}\right) O\!\left\Vert \!\left( \begin{matrix}
\Psi_{12}\Psi_{22} \big(n^2f(z)\big)\!\! & -\Psi_{12}^2\big(n^2f(z)\big) \\
\Psi_{22}^2\big(n^2f(z)\big)\!\! & -\Psi_{12}\Psi_{22}\big(n^2f(z)\big)
\end{matrix}\right)\! \left(\frac{F^2}{w_+}(z) + \frac{F^2}{w_-}(z) - 2 \right)\! \right\Vert_{L^1(1, 1+1/n)},
\end{gather*}
which, by (\ref{GNEstimates}), yields
\begin{gather*}
\mathcal{E} = O\left(\frac{1}{n^{3}\log^2 n}\right).
\end{gather*}
So (\ref{TossAway}) becomes
\begin{gather*}
 -2 \pi n\int_1^{1 + \frac{1}{n}} E(1) \left(\begin{matrix}1 & 0 \\ 0 & 0\end{matrix}\right) \left( \begin{matrix}
\Psi_{12}\Psi_{22} \big(n^2f(z)\big) & -\Psi_{12}^2\big(n^2f(z)\big) \\
\Psi_{22}^2\big(n^2f(z)\big) & -\Psi_{12}\Psi_{22}\big(n^2f(z)\big)
\end{matrix}\right) \left(\begin{matrix} 0& 0 \\ 0 & 1\end{matrix}\right) \nonumber\\
\qquad \quad{} \times E^{-1}(1)\left[ \left(\frac{F^2}{w_+}(z) + \frac{F^2}{w_-}(z) - 2 \right) \dbar z\right] + O\left(\frac{1}{n^{2}\log^2n}\right) \nonumber\\
\qquad{} = -2 \pi n\int_1^{1 + \frac{1}{n}} E(1)\left( \begin{matrix}
0 & -\Psi_{12}^2\big(n^2f(z)\big) \\
0 &0
\end{matrix}\right) E^{-1}(1) \left[ \left(\frac{F^2}{w_+}(z) + \frac{F^2}{w_-}(z) - 2 \right) \dbar z\right] \nonumber\\
\qquad \quad{} + O\left(\frac{1}{n^{2}\log^2n}\right)\nonumber \\
 \qquad{} = 2 \pi n\int_1^{1 + \frac{1}{n}} E(1)\left( \begin{matrix}
0 &1\\
0 &0
\end{matrix}\right) E^{-1}(1) \left[ \Psi_{12}^2\big(n^2f(z)\big) \left(\frac{F^2}{w_+}(z) + \frac{F^2}{w_-}(z) - 2 \right) \dbar z\right] \nonumber\\
 \qquad\quad{} + O\left(\frac{1}{n^{2}\log^2n}\right)\nonumber \\
 \qquad{} = 2 \pi nE(1)\left( \begin{matrix}
0 &1\\
0 &0
\end{matrix}\right) E^{-1}(1)\int_1^{1 + \frac{1}{n}} \Psi_{12}^2\big(n^2f(z)\big) \left(\frac{F^2}{w_+}(z) + \frac{F^2}{w_-}(z) - 2 \right) \dbar z\nonumber\\
 \qquad\quad{} + O\left(\frac{1}{n^{2}\log^2n}\right),
\end{gather*}
which, by (\ref{GNIntegral}), is
\begin{gather*}
 2 \pi nE(1)\left( \begin{matrix}
0 &1\\
0 &0
\end{matrix}\right) E^{-1}(1) \left( \frac{3}{16 \pi in^2 \log^2 n} + O\left(\frac{1}{n^2\log^3n}\right)\right) \nonumber\\
\qquad{} =\frac{3}{8 i n\log^2n} E(1)\left( \begin{matrix}
0 &1\\
0 &0
\end{matrix}\right) E^{-1}(1) + O\left(\frac{1}{n\log^3n}\right).\label{JustTheMatrixPart}
\end{gather*}
Substituting in (\ref{E1}),
\begin{gather}
\frac{3}{8 i n\log^2n}E(1)\left( \begin{matrix}
0 &1\\
0 &0
\end{matrix}\right) E^{-1}(1) + O\left(\frac{1}{n\log^3n}\right) \nonumber\\
\qquad{} =\frac{3}{8 i n\log^2n} \frac{1}{\sqrt{2}}\left(\begin{matrix}
1 & -i\\
-i & 1
\end{matrix}\right)\left( \begin{matrix}
0 &1\\
0 &0
\end{matrix}\right) \frac{1}{\sqrt{2}}\left(\begin{matrix}
1 & i\\
i & 1
\end{matrix}\right) + O\left(\frac{1}{n\log^3n}\right) \label{I1E}\\
\qquad{} =\frac{3}{16 i n\log^2n} \left(\begin{matrix}
i & 1\\
1& -i
\end{matrix}\right) + O\left(\frac{1}{n\log^3n}\right)
=\frac{3}{16 n\log^2n} \left(\begin{matrix}
1 & -i\\
-i& -1
\end{matrix}\right) + O\left(\frac{1}{n\log^3n}\right),\nonumber
\end{gather}
which is the right-hand side of (\ref{EIntegralApp}).

All that remains is to show that $I_2$ is of lower order. Following the same steps as the lead up to (\ref{PlugInE0}),
\begin{gather*}
I_2 = -2\pi n\int_{(\Sigma \backslash [-1,1]) \cap U_{1/n}(-1)}\sigma_3\left(1 + O\left(\frac{1}{n}\right)\right)E(-z) \left(\begin{matrix}
1 & 0 \\
0 & \dfrac{1}{2\pi n}
\end{matrix}\right) \nonumber \\
\qquad\times \left( \begin{matrix}
-\Psi_{12}\Psi_{22} \big(n^2f(-z)\big) & \Psi_{12}^2\big(n^2f(-z)\big) \\
-\Psi_{22}^2\big(n^2f(-z)\big) & \Psi_{12}\Psi_{22}\big(n^2f(-z)\big)
\end{matrix}\right) \nonumber\\
\qquad\times \left(\begin{matrix}
\dfrac{1}{2\pi n} & 0 \\
0 & 1
\end{matrix}\right) E^{-1}(-z)\left(1 + O\left(\frac{1}{n}\right)\right)\sigma_3 \left[ \left(\frac{F^2}{w}(z) -1 \right)\right] \dbar z.
\end{gather*}
Therefore,
\begin{gather*}
\vert I_2 \vert \le cn\int_{(\Sigma \backslash [-1,1]) \cap U_{1/n}(-1)}\left\vert \left( \begin{matrix}
-\Psi_{12}\Psi_{22} \left(n^2f(-z)\right) & \Psi_{12}^2\big(n^2f(-z)\big) \\
-\Psi_{22}^2\big(n^2f(-z)\big) & \Psi_{12}\Psi_{22}\big(n^2f(-z)\big)\end{matrix}\right)\right\vert \\
\hphantom{\vert I_2 \vert \le}{} \times \left\vert \left(\frac{F^2}{w}(z) -1 \right)\right\vert \vert {\rm d}z\vert.
\end{gather*}
As in (\ref{FTonLeft}), for $\vert z + 1 \vert \le \frac{1}{n}$,
\begin{gather*}
\left\vert \left(\frac{F^2}{w}(z) -1 \right)\right\vert = O\big(n^{-1/2}\big),
\end{gather*}
and so
\begin{gather*}
\vert I_2 \vert \le cn^{1/2} \int_{(\Sigma \backslash [-1,1]) \cap U_{1/n}(-1)}\left\vert \left( \begin{matrix}
-\Psi_{12}\Psi_{22} \left(n^2f(-z)\right) & \Psi_{12}^2\big(n^2f(-z)\big) \\
-\Psi_{22}^2\big(n^2f(-z)\big) & \Psi_{12}\Psi_{22}\big(n^2f(-z)\big)
\end{matrix}\right)\right\vert \vert {\rm d}z\vert.
\end{gather*}
Letting $y = n^2f(-z)$, in which case we have, as before, ${\rm d}z = \frac{2{\rm d}y}{n^2} \big(1 + O\big(\frac{1}{n}\big)\big)$, and recalling that $n^2f(-(\Sigma \cap U_{1/n}(-1))) \subset \Gamma$,
\begin{gather*}
\vert I_2 \vert \le cn^{-3/2} \int_{\substack{0 \le \vert y\vert \le O(n) \\ \arg y = \pm2\pi/3}}\left\vert \left( \begin{matrix}
-\Psi_{12}\Psi_{22} (y) & \Psi_{12}^2(y) \\
-\Psi_{22}^2(y) &\Psi_{12}\Psi_{22}(y)
\end{matrix}\right)\right\vert \vert {\rm d}y\vert = O\big(n^{-3/2}\big),
\end{gather*}
where we have used the fact that $\Psi_{12}, \Psi_{22} \in L^2(C_0)$, where $C_0$ is the contour $\vert \arg y \vert= 2\pi/3$, see (\ref{PsiAsymptotics}). Therefore we see that
\begin{gather*}
\int_\Sigma \tilde{\mu} (v_\Sigma - \tilde{v}_\Sigma) \tilde{\mu}^{-1}\dbar z = I_1 + O\big(n^{-3/2}\big)
=\frac{3}{16 n\log^2n} \left(\begin{matrix}
1 & -i\\
-i& -1
\end{matrix}\right) + O\left(\frac{1}{n\log^3n}\right),
\end{gather*}
where we have used (\ref{I1E}). Thus we have proved Proposition \ref{EIntegralPropApp}.
\end{proof}

\subsection{The second integral}
\begin{prop}\label{HIntegralPropApp}
\begin{gather}
\int_\Sigma \tilde{\mu}^{(n)}(z) \big(v_\Sigma^{(n)}(z) - \tilde{v}^{(n)}_\Sigma(z)\big) \big(\tilde{\mu}^{(n)}\big)^{-1}(z) \dbar z \nonumber\\
\qquad \quad{} - \int_\Sigma \tilde{\mu}^{(n+1)}(z) \big(v_\Sigma^{(n+1)}(z) - \tilde{v}^{(n+1)}_\Sigma(z)\big) \big(\tilde{\mu}^{(n+1)}\big)^{-1}(z) \dbar z \nonumber\\
\qquad{} = \frac{3}{16 n^2\log^2n} \left(\begin{matrix}
1 & -i\\
-i& -1
\end{matrix}\right) + O\left(\frac{1}{n^2\log^3n}\right).\label{HIntegralApp}
\end{gather}
\end{prop}
\begin{proof}While not sufficient to prove the above, note that
\begin{gather*}
\frac{3}{16 n \log^2n}\left(\begin{matrix}
1 & -i\\
-i& -1
\end{matrix}\right) - \frac{3}{16 (n+1) \log^2(n+1)}\left(\begin{matrix}
1 & -i\\
-i& -1
\end{matrix}\right) \nonumber\\
\qquad{}= \frac{3}{16 n^2\log^2n} \left(\begin{matrix}
1 & -i\\
-i& -1
\end{matrix}\right) + O\left(\frac{1}{n^2\log^3n}\right).
\end{gather*}
So this proposition essentially asserts that the leading order of the difference (\ref{HIntegralApp}) is the difference of the leading orders of the two terms, which are computed in (\ref{EIntegralApp}).

We have, as before,
\begin{gather*}
\int_\Sigma \tilde{\mu}^{(n)}(z) \big(v_\Sigma^{(n)}(z) - \tilde{v}^{(n)}_\Sigma(z)\big) \big(\tilde{\mu}^{(n)}\big)^{-1}(z) \dbar z \nonumber\\
\qquad {}- \int_\Sigma \tilde{\mu}^{(n+1)}(z) \big(v_\Sigma^{(n+1)}(z) - \tilde{v}^{(n+1)}_\Sigma(z)\big) \big(\tilde{\mu}^{(n+1)}\big)^{-1}(z) \dbar z
 = I_1 + I_2 + I_3,
\end{gather*}
where $I_1$, $I_2$, $I_3$ refer to the difference of the integrals of the above over $[1, 1+\delta]$, $\Sigma \cap U_{1/n}(-1)$, and $\Sigma \backslash (U_{1/n}(1)\cap U_{1/n}(-1))$ respectively. The same argument as in Proposition \ref{NormsApp} shows
\begin{gather*}
I_3 = O\big(e^{-cn^{1/2}}\big).
\end{gather*}

Following the same steps as in Proposition \ref{EIntegralPropApp}, but this time using the more precise expansion $R^{(n)}(z)= I + \frac{R_1(z)}{n} + O\big(\frac{1}{n^2}\big) $ in the expression for $\tilde{\mu}$, we see,
\begin{gather*}
 -\int_1^{1 + \frac{1}{n}} \tilde{\mu}^{(n)}(z) \big(v_\Sigma(z)^{(n)} - \tilde{v}^{(n)}_\Sigma(z)\big) \big(\tilde{\mu}^{(n)}\big)^{-1}(z) \dbar z\nonumber\\
\qquad{} = -2\pi n\int_1^{1 + \frac{1}{n}} \left(1 + \frac{R_1(z)}{n} + O\left(\frac{1}{n^2}\right)\right)E(z)\left(\begin{matrix}
1 & 0\\
0 & \dfrac{1}{2\pi n}
\end{matrix}\right) \nonumber \\
\qquad\quad{} \times\left( \begin{matrix}
\Psi_{12}\Psi_{22} \big(n^2f(z)\big) & -\Psi_{12}^2\big(n^2f(z)\big) \\
\Psi_{22}^2\big(n^2f(z)\big) & -\Psi_{12}\Psi_{22}\big(n^2f(z)\big)
\end{matrix}\right)\left(\begin{matrix}
\dfrac{1}{2\pi n} & 0\\
0 & 1
\end{matrix}\right) \nonumber\\
\qquad\quad{}\times E^{-1}(z)\left(1 - \frac{R_1(z)}{n} + O\left(\frac{1}{n^2}\right)\right) \left[ \left(\frac{F^2}{w_+}(z) + \frac{F^2}{w_-}(z) - 2 \right)\right]\dbar z, %\label{PlugInE1}
\end{gather*}
and therefore,
\begin{gather}
- \int_1^{1+\frac{1}{n}} \tilde{\mu}^{(n)}(z) \big(v_\Sigma^{(n)}(z) - \tilde{v}^{(n)}_\Sigma(z)\big) \big(\tilde{\mu}^{(n)}\big)^{-1}(z)\dbar z \nonumber\\
\qquad\quad {}+ \int_1^{1+\frac{1}{n+1}} \tilde{\mu}^{(n+1)}(z) \big(v_\Sigma^{(n+1)}(z) - \tilde{v}^{(n+1)}_\Sigma(z)\big) \big(\tilde{\mu}^{(n+1)}\big)^{-1}(z) \dbar z \nonumber\\
\qquad {}= -2\pi n\int_1^{1 + \frac{1}{n}} \left(1 + \frac{R_1(z)}{n} + O\left(\frac{1}{n^2}\right)\right)E(z) \left(\begin{matrix}
1 & 0\\
0 & \dfrac{1}{2\pi n}
\end{matrix}\right)\nonumber \\
\qquad\quad {}\times \left( \begin{matrix}
\Psi_{12}\Psi_{22} \big(n^2f(z)\big) & -\Psi_{12}^2\big(n^2f(z)\big) \\
\Psi_{22}^2\big(n^2f(z)\big) & -\Psi_{12}\Psi_{22}\big(n^2f(z)\big)
\end{matrix}\right) \nonumber\\
\qquad\quad{}\times\left(\begin{matrix}
\dfrac{1}{2\pi n} & 0\\
0 & 1
\end{matrix}\right) E^{-1}(z)\left(1 - \frac{R_1(z)}{n} + O\left(\frac{1}{n^2}\right)\right)
 \times \left[ \left(\frac{F^2}{w_+}(z) + \frac{F^2}{w_-}(z) - 2\right)\right]\dbar z \nonumber\\
 \qquad\quad {} +2\pi (n+1)\int_1^{1 +\frac{1}{n+1}}\left(1 + \frac{R_1(z)}{n+1} + O\left(\frac{1}{n^2}\right)\right)E(z) \left(\begin{matrix}
1 & 0\\
0 & \dfrac{1}{2\pi(n+1)}
\end{matrix}\right)\nonumber \\
\qquad\quad{}\times \left( \begin{matrix}
\Psi_{12}\Psi_{22} \big((n+1)^2f(z)\big) & -\Psi_{12}^2\big((n+1)^2f(z)\big) \\
\Psi_{22}^2\big((n+1)^2f(z)\big) & -\Psi_{12}\Psi_{22}\big((n+1)^2f(z)\big)
\end{matrix}\right) \nonumber\\
\qquad\quad {}\times \left(\begin{matrix}
\dfrac{1}{2\pi (n+1)} & 0\\
0 & 1
\end{matrix}\right) E^{-1}(z)\left(1 - \frac{R_1(z)}{(n+1)} + O\left(\frac{1}{n^2}\right)\right) \nonumber\\
\qquad\quad {}\times \left[ \left(\frac{F^2}{w_+}(z) + \frac{F^2}{w_-}(z) - 2 \right) \right]\dbar z. \label{HIntegralFirstMess}
\end{gather}

Now, let $y_1 = n^2f(z)$ and $y_2 = (n+1)^2f(z)$, note that if $z \in [1, 1+1/n]$ then $0 < y_1, y_2 < cn$. Let $f_1^{-1}$ refer to the inverse of $f$ within a neighborhood of~1. Then for $z\in [1, 1+1/n]$, and some constant~$a$,
\begin{gather*}
{\rm d}z = \frac{1}{n^2}\big(f_{1}^{-1}\big)'\left(\frac{y_1}{n^2}\right) {\rm d}y_1 = \frac{2{\rm d}y_1}{n^2}\left(1 + \frac{ay_1}{n^2} + O\left(\frac{1}{n^2}\right)\right), \nonumber\\
{\rm d}z = \frac{1}{(n+1)^2}\big(f_{1}^{-1}\big)'\left(\frac{y_2}{(n+1)^2}\right) {\rm d}y_2 = \frac{2{\rm d}y_2}{(n+1)^2}\left(1 + \frac{ay_2}{(n+1)^2} + O\left(\frac{1}{n^2}\right)\right).
\end{gather*}
Making the substitution $y_1 = n^2f(z)$ in the first integral of (\ref{HIntegralFirstMess}) and $y_2 = (n+1)^2f(z)$ in the second, we obtain
\begin{gather}
 -\frac{4\pi}{n}\int_0^{n^2f\left(1 +\frac{1}{n}\right)} \left(1 + \frac{R_1\big(f_1^{-1}\big(\frac{y_1}{n^2}\big)\big)}{n} + O\left(\frac{1}{n^2}\right)\right)E\left(f_1^{-1}\left(\frac{y_1}{n^2}\right)\right)\nonumber \\
\qquad{} \times \left(\begin{matrix}
 1 & 0\\
 0 & \dfrac{1}{2\pi n}
 \end{matrix}\right)\left( \begin{matrix}
\Psi_{12}\Psi_{22} (y_1 ) & -\Psi_{12}^2(y_1) \\
\Psi_{22}^2(y_1) & -\Psi_{12}\Psi_{22}(y_1)
\end{matrix}\right) \left(\begin{matrix}
\dfrac{1}{2\pi n} & 0\\
0 & 1
\end{matrix}\right) \nonumber\\
\qquad{}\times E^{-1}\left(f_1^{-1}\left(\frac{y_1}{n^2}\right)\right)\left(1 - \frac{R_1\big(f_1^{-1}\big(\frac{y_1}{n^2}\big)\big)}{n} + O\left(\frac{1}{n^2}\right)\right) \nonumber \\
\qquad{} \times \Bigg[ \left(\frac{F^2}{w_+}\left(f_1^{-1}\left(\frac{y_1}{n^2}\right)\right) + \frac{F^2}{w_-}\left(f_1^{-1}\left(\frac{y_1}{n^2}\right)\right) - 2 \right)
 \left(1 + \frac{ay_1}{n^2} + O\left(\frac{1}{n^2}\right) \right)\Bigg] \dbar y_1 \nonumber\\
 \qquad{} +\frac{4\pi}{n+1}\int_0^{(n+1)^2f\left(1 +\frac{1}{n+1}\right)}\! \left(1 + \frac{R_1\big(f_1^{-1}\big(\frac{y_2}{(n+1)^2}\big)\big)}{n+1} + O\left(\frac{1}{n^2}\right)\right)E\left(f_1^{-1}\left(\frac{y_2}{(n+1)^2}\right)\right) \nonumber \\
 \qquad{}\times\left(\begin{matrix}
 1 & 0\\
 0 & \dfrac{1}{2\pi (n+1)}
 \end{matrix}\right) \left( \begin{matrix}
\Psi_{12}\Psi_{22} (y_2 ) & -\Psi_{12}^2(y_2) \\
\Psi_{22}^2(y_2) & -\Psi_{12}\Psi_{22}(y_2)
\end{matrix}\right) \left(\begin{matrix}
 \frac{1}{2\pi (n+1)} & 0\\
 0 & 1
 \end{matrix}\right)\nonumber\\
\qquad{}\times E^{-1}\left(f_1^{-1}\left(\frac{y_2}{(n+1)^2}\right)\right)\left(1 - \frac{R_1\big(f_1^{-1}\big(\frac{y_2}{(n+1)^2}\big)\big)}{n+1} + O\left(\frac{1}{n^2}\right)\right) \nonumber \\
\qquad{} \times \Bigg[ \left(\frac{F^2}{w_+}\left(f_1^{-1}\left(\frac{y_2}{(n+1)^2}\right)\right) + \frac{F^2}{w_-}\left(f_1^{-1}\left(\frac{y_2}{(n+1)^2}\right)\right) - 2 \right) \nonumber\\
\qquad{}\times \left(1 + \frac{ay_2}{(n+1)^2} + O\left(\frac{1}{n^2}\right) \right) \dbar y_2.\label{OhDearGod}
\end{gather}

Note that $(n+1)^2f\big(1 + \frac{1}{n+1}\big) = O(n)$ and $(n+1)^2f\big(1 + \frac{1}{n+1}\big) - n^2f\big(1 + \frac{1}{n}\big) = O(1)$. Therefore, since all of the terms in the above integrands are bounded for large $y$, and since~$\Psi_{12}$,~$\Psi_{22}$ rapidly decay for large argument, we can replace the limits of the second integral in~(\ref{OhDearGod}) with the limits of the first and pick up an error that is exponentially small in~$n$. Next, recall that
\begin{gather*}
f^{-1}(y) - 1 = 2y + ay^2 + O\big(y^3\big)
\end{gather*}
and $y$ is in the region $\big\vert f^{-1}\big(\frac{y}{n^2}\big) - 1 \big\vert\le \frac{1}{n}$, with
\begin{gather*}
\frac{f^{-1}\big(\frac{y}{n^2}\big) }{f^{-1}\big(\frac{y}{(n+1)^2}\big)} = \frac{2\frac{y}{n^2}\big(1 + O\big(\frac{1}{n}\big)\big)}{2\frac{y}{(n+1)^2}\big(1 + O\big(\frac{1}{n}\big)\big)} = 1 + O\left(\frac{1}{n}\right).
\end{gather*}
So, by Proposition \ref{F-FAsymptotics}, we see that
\begin{gather*}
\frac{F^2}{w_+}\left(f_1^{-1}\left(\frac{y}{n^2}\right)\right) + \frac{F^2}{w_-}\left(f_1^{-1}\left(\frac{y}{n^2}\right)\right) - \frac{F^2}{w_+}\left(f_1^{-1}\left(\frac{y}{(n+1)^2}\right)\right) \\
\qquad{} + \frac{F^2}{w_-}\left(f_1^{-1}\left(\frac{y}{(n+1)^2}\right)\right)= O\left(\frac{1}{n\log^3n}\right).
\end{gather*}
Next, using the analyticity of $E$, $R_1$,
\begin{gather*}
 E\left(f_1^{-1}\left(\frac{y}{n^2}\right)\right) - E\left(f_1^{-1}\left(\frac{y}{(n+1)^2}\right)\right) = O\left(\frac{1}{n^2}\right)
\end{gather*}
and
\begin{gather*}
\left(1 + \frac{R_1\big(f_1^{-1}\big(\frac{y_2}{n^2}\big)\big)}{n}\right)-\left(1 + \frac{R_1\big(f_1^{-1}\big(\frac{y_2}{(n+1)^2}\big)\big)}{n+1}\right) = O\left(\frac{1}{n^2}\right),
\end{gather*}
and clearly
\begin{gather*}
\left(\begin{matrix}
1 & 0\\
0 & \dfrac{1}{2\pi n}
\end{matrix}\right) - \left(\begin{matrix}
1 & 0\\
0 & \dfrac{1}{2\pi (n+1)}
\end{matrix}\right) = O\left(\frac{1}{n^2}\right).
\end{gather*}
Thus, subtracting (\ref{OhDearGod}) term by term, (\ref{OhDearGod}) is
\begin{gather}
\left(-\frac{4\pi}{n} + \frac{4\pi }{n+1}\right)\int_0^{n^2f\left(1 +\frac{1}{n}\right)}\left(1 +O\left(\frac{1}{n}\right)\right)E\left(f_1^{-1}\left(\frac{y}{n^2}\right)\right)\nonumber \\
\qquad{} \times \left(\begin{matrix}
 1 & 0\\
 0 & \dfrac{1}{2\pi n}
 \end{matrix}\right)\left( \begin{matrix}
\Psi_{12}\Psi_{22} (y) & -\Psi_{12}^2(y) \\
\Psi_{22}^2(y) & -\Psi_{12}\Psi_{22}(y)
\end{matrix}\right) \left(\begin{matrix}
\dfrac{1}{2\pi n} & 0\\
0 & 1
\end{matrix}\right) \nonumber\\
\qquad{} \times E^{-1}\left(f_1^{-1}\left(\frac{y}{n^2}\right)\right)\left(1 + O\left(\frac{1}{n}\right)\right)\nonumber \\
\qquad {} \times \left[ \left(\frac{F^2}{w_+}\left(f_1^{-1}\left(\frac{y}{n^2}\right)\right) + \frac{F^2}{w_-}\left(f_1^{-1}\left(\frac{y}{n^2}\right)\right) -2\right)\right]\dbar y \nonumber\\
 \qquad{} + \frac{4\pi}{n} O\left(\frac{1}{n^2}\right)\left\Vert \left( \begin{matrix}
\Psi_{12}\Psi_{22} (y) & -\Psi_{12}^2(y) \\
\Psi_{22}^2(y) & -\Psi_{12}\Psi_{22}(y)
\end{matrix}\right) \right.\nonumber\\
\left. \qquad{}\times \left(\frac{F^2}{w_+}\left(f_1^{-1}\left(\frac{y}{n^2}\right)\right) + \frac{F^2}{w_-}\left(f_1^{-1}\left(\frac{y}{n^2}\right)\right) - 2\right) \right\Vert_{L^1(0, n^2f(1 + \frac{1}{n}))} \nonumber \\
\qquad{} + \frac{4\pi}{n} O\left(\frac{1}{n\log^3n}\right) \left\Vert \left( \begin{matrix}
\Psi_{11}\Psi_{22} (y ) & -\Psi_{12}^2(y) \\
\Psi_{22}^2(y) & -\Psi_{12}\Psi_{22}(y)\end{matrix}\right) \right\Vert_{L^1(0, n^2f(1 + 1/n))}. \label{Something100}
\end{gather}
The function $\frac{F^2}{w_+} + \frac{F^2}{w_-} - 2$ is bounded in compact sets and $\Psi_{12}$, $\Psi_{22}$ are both $L^2(0,\infty)$, therefore the $L^1$ norms are both bounded uniformly of~$n$. Therefore~(\ref{Something100}) is
\begin{gather}
-\frac{4\pi}{n(n+1)} \int_0^{n^2f\left(1 +\frac{1}{n}\right)}\left(1 +O\left(\frac{1}{n}\right)\right)E\left(f_1^{-1}\left(\frac{y}{n^2}\right)\right)\nonumber \\
\qquad {}\times \left(\begin{matrix}
 1 & 0\\
 0 & \dfrac{1}{2\pi n}
 \end{matrix}\right)\left( \begin{matrix}
\Psi_{11}\Psi_{22} (y) & -\Psi_{12}^2(y) \\
\Psi_{22}^2(y) & -\Psi_{12}\Psi_{22}(y)
\end{matrix}\right) \left(\begin{matrix}
\frac{1}{2\pi n} & 0\\
0 & 1
\end{matrix}\right) \nonumber\\
\qquad {}\times E^{-1}\left(f_1^{-1}\left(\frac{y}{n^2}\right)\right)\left(1 + O\left(\frac{1}{n}\right)\right)\nonumber \\
\qquad \times \left[ \left(\frac{F^2}{w_+}\left(f_1^{-1}\left(\frac{y}{n^2}\right)\right) + \frac{F^2}{w_-}\left(f_1^{-1}\left(\frac{y}{n^2}\right)\right) -2\right)\right]\dbar y + O\left(\frac{1}{n^2 \log^3 n}\right).\label{Something101}
\end{gather}
Inverting the change of variables $y = n^2 f(z)$, (\ref{Something101}) becomes
\begin{gather}
 -\frac{2\pi n}{n+1}\int_1^{1 + \frac{1}{n}} \left(1 + O\left(\frac{1}{n}\right)\right)E(z) \left(\begin{matrix}
 1 & 0\\
 0 & \dfrac{1}{2\pi n}
 \end{matrix}\right)\nonumber \\
\qquad{} \times \left( \begin{matrix}
\Psi_{12}\Psi_{22} \big(n^2f(z)\big) & -\Psi_{12}^2\big(n^2f(z)\big) \\
\Psi_{22}^2\big(n^2f(z)\big) & -\Psi_{12}\Psi_{22}\big(n^2f(z)\big)
\end{matrix}\right) \left(\begin{matrix}
\dfrac{1}{2\pi n} & 0\\
0 & 1
\end{matrix}\right)E^{-1}(z)\left(1 + O\left(\frac{1}{n}\right)\right) \nonumber \\
\qquad{} \times \left[ \left(\frac{F^2}{w_+}(z) + \frac{F^2}{w_-}(z) - 2\right) \right]\dbar z + O\left(\frac{1}{n^2\log^3n}\right). \label{Something102}
\end{gather}
However, the integral here is exactly $\frac{1}{n+1}$ times the integral in (\ref{PlugInE0}), and therefore (\ref{Something102}) evaluates to
\begin{gather}
 \frac{1}{n+1}\left(\frac{3}{16 n\log^2n} \left(\begin{matrix}
1 & -i\\
-i& -1
\end{matrix}\right) + O\left(\frac{1}{n\log^3n}\right)\right) \nonumber\\
\qquad{} = \frac{3}{16 n^2\log^2n} \left(\begin{matrix}
1 & -i\\
-i& -1
\end{matrix}\right) + O\left(\frac{1}{n^2\log^3n}\right),\label{I1H}
\end{gather}
which is the right-hand side of (\ref{HIntegralApp}).

To conclude the proof we need to show the contribution from the left endpoint, $I_2$, is of lower order. Following the same steps as those leading up to (\ref{HIntegralFirstMess}), we have
\begin{gather}
I_2 = -2\pi n\int_{(\Sigma\backslash [-1,1])\cap U_{1/n}(-1)} \sigma_3\left(1 + \frac{R_1(-z)}{n} + O\left(\frac{1}{n^2}\right)\right)E(-z) \left(\begin{matrix}
1 & 0\\
0 & \dfrac{1}{2\pi n}
\end{matrix}\right)\nonumber \\
\hphantom{I_2 =}{} \times \left( \begin{matrix}
-\Psi_{12}\Psi_{22} \big(n^2f(-z)\big) & \Psi_{12}^2\big(n^2f(-z)\big) \\
-\Psi_{22}^2\big(n^2f(-z)\big) & \Psi_{12}\Psi_{22}\big(n^2f(-z)\big)
\end{matrix}\right) \nonumber\\
\hphantom{I_2 =}{} \times\left(\begin{matrix}
\dfrac{1}{2\pi n} & 0\\
0 & 1
\end{matrix}\right) E^{-1}(-z)\left(1 - \frac{R_1(-z)}{n} + O\left(\frac{1}{n^2}\right)\right)\sigma_3 \left[ \left(\frac{F^2}{w}(z)-1\right)\right]\dbar z \nonumber\\
\hphantom{I_2 =}{} +2\pi (n+1)\int_{(\Sigma\backslash [-1,1])\cap U_{1/n}(-1)}\!\sigma_3\left(1 + \frac{R_1(-z)}{n+1} + O\left(\frac{1}{n^2}\right)\right)E(-z) \left(\begin{matrix}
1 & 0\\
0 & \dfrac{1}{2\pi(n+1)}
\end{matrix}\right)\nonumber \\
\hphantom{I_2 =}{} \times \left( \begin{matrix}
-\Psi_{12}\Psi_{22} \big((n+1)^2f(-z)\big) & \Psi_{12}^2\big((n+1)^2f(-z)\big) \\
-\Psi_{22}^2\big((n+1)^2f(-z)\big) & \Psi_{12}\Psi_{22}\big((n+1)^2f(-z)\big)
\end{matrix}\right) \nonumber\\
\hphantom{I_2 =}{} \times \left(\begin{matrix}
\dfrac{1}{2\pi (n+1)}\! & 0\\
0 & 1
\end{matrix}\right) E^{-1}(-z)\left(1 - \frac{R_1(-z)}{(n+1)} + O\left(\frac{1}{n^2}\right)\right)\sigma_3 \left[ \left(\frac{F^2}{w}(z) -1 \right) \right]\dbar z.\!\!\!\!\label{Something110}
\end{gather}
Subtracting (\ref{Something110}) term by term,
\begin{gather}
\vert I_2 \vert \le c \int_{(\Sigma\backslash [-1,1])\cap U_{1/n}(-1)} \left\vert \left( \begin{matrix}
-\Psi_{12}\Psi_{22} \big(n^2f(-z)\big) & \Psi_{12}^2\big(n^2f(-z)\big) \\
-\Psi_{22}^2\big(n^2f(-z)\big) & \Psi_{12}\Psi_{22}\big(n^2f(-z)\big)
\end{matrix}\right)\right\vert \left\vert \frac{F^2}{w}(z) -1\right\vert \vert {\rm d}z\vert \nonumber\\
\hphantom{\vert I_2 \vert \le}{}
+ cn \int_{(\Sigma\backslash [-1,1])\cap U_{1/n}(-1)} \left\vert \left(\begin{matrix}
-\Psi_{12}\Psi_{22} \big((n+1)^2f(-z)\big) & \Psi_{12}^2\big((n+1)^2f(-z)\big) \\
-\Psi_{22}^2\big((n+1)^2f(-z)\big) & \Psi_{12}\Psi_{22}\big((n+1)^2f(-z)\big)
\end{matrix}\right) \right. \nonumber\\
\left. \hphantom{\vert I_2 \vert \le}{}- \left(\begin{matrix}
-\Psi_{12}\Psi_{22} \big(n^2f(-z)\big) & \Psi_{12}^2\big(n^2f(-z)\big) \\
-\Psi_{22}^2\big(n^2f(-z)\big) & \Psi_{12}\Psi_{22}\big(n^2f(-z)\big)
\end{matrix}\right)\right\vert \left\vert \frac{F^2}{w}(z) -1\right\vert \vert {\rm d}z\vert.\label{Something111}
\end{gather}
Recall from (\ref{FTonLeft}) that for $\vert z +1 \vert \le \frac{1}{n}$,
\begin{gather}
 \left\vert \frac{F^2}{w}(z) -1\right\vert = O\big(n^{-1/2}\big).\label{Something112}
\end{gather}
Making the change of variables $y = n^2 f(-z)$ and using (\ref{Something112}), (\ref{Something111}) becomes
\begin{gather*}
cn^{-5/2} \int_{\substack{0 \le \vert y\vert \le O(n) \\ \arg y = \pm2\pi/3}} \left\vert \left( \begin{matrix}
-\Psi_{12}\Psi_{22} (y) & \Psi_{12}^2(y) \\
-\Psi_{22}^2(y) & \Psi_{12}\Psi_{22}(y)
\end{matrix}\right)\right\vert \vert {\rm d}y\vert \nonumber\\
\qquad{} + cn^{-3/2} \int_{\substack{0 \le \vert y\vert \le O(n) \\ \arg y = \pm2\pi/3}} \left\vert \left( \begin{matrix}
-\Psi_{12}\Psi_{22} \left(\dfrac{(n+1)^2}{n^2}y\right) & \Psi_{12}^2\left(\dfrac{(n+1)^2}{n^2}y\right) \\
-\Psi_{22}^2\left(\dfrac{(n+1)^2}{n^2}y\right) & \Psi_{12}\Psi_{22}\left(\dfrac{(n+1)^2}{n^2}y\right)
\end{matrix}\right) \right. \nonumber\\
\left. \qquad{} - \left( \begin{matrix}
-\Psi_{12}\Psi_{22} (y) & \Psi_{12}^2(y) \\
-\Psi_{22}^2(y) & \Psi_{12}\Psi_{22}(y)
\end{matrix}\right) \right\vert \vert {\rm d}y\vert \nonumber\\
\quad{}\le c_1n^{-5/2} + c_2n^{-5/2} \int_{\substack{0 \le \vert y\vert \le O(n) \\ \arg y = \pm2\pi/3}} \left\vert y\sup_{\left(y, \frac{(n+1)^2}{n^2}y\right)}\frac{{\rm d}}{{\rm d}x} \left( \begin{matrix}
-\Psi_{12}\Psi_{22} \left(x\right) & \Psi_{12}^2(x) \\
-\Psi_{22}^2(x) & \Psi_{12}\Psi_{22}(x)
\end{matrix}\right)\right\vert \vert {\rm d}y\vert. % \label{SomethingElse1}
\end{gather*}
Recall that $\Psi_2(y) = O(\log y)$ and $\Psi_2'(y) = O\big(\frac{1}{y}\big)$ as $y \to 0$, and $\Psi_2(y), \Psi_2'(y) = O\big(e^{-cy^{1/2}} \big)$ as $y \to \infty$, see~(\ref{PsiAsymptotics}). Therefore
\begin{gather*}
\int_{\substack{0 \le \vert y\vert <\infty \\ \arg y = \pm2\pi/3}} \left\vert y\sup_{\left(y, \frac{(n+1)^2}{n^2}y\right)}\frac{{\rm d}}{{\rm d}x} \left( \begin{matrix}
-\Psi_{12}\Psi_{22} \left(x\right) & \Psi_{12}^2(x) \\
-\Psi_{22}^2(x) & \Psi_{12}\Psi_{22}(x)
\end{matrix}\right)\right\vert \vert {\rm d}y\vert = O(1),
\end{gather*}
and so
\begin{gather*}
\vert I_2 \vert = O\big(n^{-5/2}\big),
\end{gather*}
which, together with (\ref{I1H}) implies
\begin{gather*}
\int_\Sigma \tilde{\mu}^{(n)}(z) \big(v_\Sigma^{(n)}(z) - \tilde{v}^{(n)}_\Sigma(z)\big) \big(\tilde{\mu}^{(n)}\big)^{-1}(z) \dbar z \nonumber\\
\qquad\quad{} - \int_\Sigma \tilde{\mu}^{(n+1)}(z) \big(v_\Sigma^{(n+1)}(z) - \tilde{v}^{(n+1)}_\Sigma(z)\big) \big(\tilde{\mu}^{(n+1)}\big)^{-1}(z) \dbar z \nonumber\\
\qquad{} = \frac{3}{16 n^2\log^2n} \left(\begin{matrix}
1 & -i\\
-i& -1
\end{matrix}\right) + O\left(\frac{1}{n^2\log^3n}\right)
\end{gather*}
completing the proof of Proposition \ref{HIntegralPropApp}.
\end{proof}

\subsection*{Acknowledgments}
The work of the second author was supported in part by DMS Grant \# 1300965. The authors gratefully acknowledge the comments and suggestions about the result in this paper by Arno Kuijlaars and Andrei Martinez-Finkelshtein.

\pdfbookmark[1]{References}{ref}
\LastPageEnding

\end{document}